\documentclass[11pt, reqno]{amsart}

\usepackage[letterpaper, hmargin=1.0in, vmargin=1.0in]{geometry}
\usepackage{import}
\usepackage{xifthen}
\usepackage{pdfpages}
\usepackage{transparent}
\usepackage{cancel}
\usepackage[normalem]{ulem}

\usepackage{tikz}
\usepackage{float}
\usepackage{mathtools, amsmath, amsfonts, amssymb, amsthm}
\usepackage{pgffor}
\usepackage{etoolbox}
\usepackage[shortlabels]{enumitem}
\usepackage[unicode=true, colorlinks=true, linkcolor=blue]{hyperref}
\hypersetup{pdfauthor={Name}}
\usepackage{setspace}

\usepackage[T1]{fontenc}
 \usepackage{xcolor}

 \newcommand{\R}{\mathbb{R}}
 
 \newcommand{\N}{\mathbb{N}}

 \newcommand{\eps}{\varepsilon}

 \newcommand{\lam}{\lambda}
 \newcommand{\Om}{\Omega}

 \newcommand{\T}[1]{\Theta\Rnd{#1}}
 \renewcommand{\t}{\theta}
 \newcommand{\sig}{\sigma}
 \newcommand{\g}{\gamma}
 \newcommand{\del}{\delta}

 \newcommand{\al}{\alpha}
 \newcommand{\be}{\beta}

\newcommand{\Once}{\mathsf{Once}}
    \newcommand{\All}{\mathsf{All}}

 \newcommand{\Nor}[2]{\mathcal{N}\left(#1, #2\right)}
 \newcommand{\Exp}[1]{\exp\left(#1\right)}
 \newcommand{\Abs}[1]{\left|#1\right|}
 
 \renewcommand{\Box}[1]{\left[#1\right]}
 \newcommand{\Rnd}[1]{\left(#1\right)}
 \newcommand{\Bra}[1]{\left\{#1\right\}}
 \renewcommand{\bar}[1]{\overline{#1}}

 \makeatletter
 \newcommand{\E}[1][\@nil]{%
 	\def\tmp{#1}%
 	\ifx\tmp\@nnil
 	\mathbb{E}
 	\else
 	\mathbb{E}\left[#1\right]
 	\fi}
 
 \renewcommand{\P}[1][\@nil]{%
 	\def\tmp{#1}%
 	\ifx\tmp\@nnil
 	\mathbb{P}
 	\else
 	\mathbb{P}\Rnd{#1}
 	\fi}
 \makeatother

 \renewcommand{\hat}[1]{\widehat{#1}}
 
 \newcommand{\mnorm}[1]{{\left\vert\kern-0.25ex\left\vert\kern-0.25ex\left\vert #1 
 		\right\vert\kern-0.25ex\right\vert\kern-0.25ex\right\vert}}

 \newcommand{\1}{\mathbf{1}}
 \newcommand{\floor}[1]{\lfloor #1 \rfloor}
 \newcommand{\B}{\mathbb{B}}
 
 \newcommand{\cG}{\mathcal{G}}

 \newcommand{\cP}{\mathcal{P}}
 
 \newcommand{\cB}{\mathcal{B}^\mathsf{e}}

 \newcommand{\f}[2]{\frac{#1}{#2}}

 \newcommand{\tv}[1]{d_{\mathrm{TV}}\Rnd{#1}}

 \newcommand{\rlem}[1]{Lemma \ref{#1}}

 \newcommand{\ceil}[1]{{\lceil #1 \rceil}}
 \renewcommand{\floor}[1]{{\lfloor #1 \rfloor}}

 \renewcommand{\l}{\ell}
 \renewcommand{\c}{\mathsf{c}}

\newcommand{\dt}{\mathrm{d}t}
\newcommand{\FS}{\mu^{\mathrm{FS}}}
\newcommand{\aFS}{\mu^{a,\mathrm{FS}}}

\newcommand{\q}[1]{{#1}^{\f{1}{4}}}
\newcommand{\p}{\mathsf{p}}

\renewcommand{\B}{\mathbf{B}}
\newcommand{\A}{\mathbb{A}}
\newcommand{\Bav}[3]{{\B^{#1}_{#2}[#3]%
}}
\newcommand{\Lav}[3]{{\L^{#1}_{#2}[#3]
}}

\newcommand{\gtr}{\succ}

\renewcommand{\L}{{\mathcal{L}}}

\usepackage{mathtools}
\mathtoolsset{showonlyrefs}

 \newtheorem{thm}{Theorem}[section]
 \newtheorem{lem}[thm]{Lemma}
 \newtheorem{ppn}[thm]{Proposition}
 \newtheorem{cor}[thm]{Corollary}
 \newtheorem{defn}[thm]{Definition}
 \theoremstyle{definition}
 \newtheorem{rmk}[thm]{Remark}

\numberwithin{thm}{section}

\numberwithin{equation}{section}
\numberwithin{figure}{section}

\newcommand{\cC}{\ensuremath{\mathcal C}}

\newcommand{\cU}{\ensuremath{\mathcal U}}

\newcommand{\calB}{\mathcal{B}}
\newcommand{\calC}{\mathcal{C}}

\newcommand{\calL}{\mathcal{L}}

\newcommand{\calP}{\mathcal{P}}

\newcommand{\calS}{\mathcal{S}}
\newcommand{\calT}{\mathcal{T}}
\newcommand{\calU}{\mathcal{U}}

\newcommand{\bbA}{\mathbb{A}}

\newcommand{\bbE}{\mathbb{E}}

\newcommand{\bbI}{\mathbb{I}}

\newcommand{\bbN}{\mathbb{N}}

\newcommand{\bbP}{\mathbb{P}}

\newcommand{\bbR}{\mathbb{R}}
\newcommand{\bbS}{\mathbb{S}}
\newcommand{\bbT}{\mathbb{T}}

\newcommand{\sfe}{\mathsf e}

\newcommand{\sfl}{{\mathsf l}}

\newcommand{\ux}{\underline{x}}
\newcommand{\uy}{\underline{y}}

\newcommand{\uX}{\underline{X}}
\newcommand{\uY}{\underline{Y}}
\newcommand{\uZ}{\underline{Z}}
\newcommand{\uW}{\underline{W}}

\newcommand{\lb}{\left(}
\newcommand{\rb}{\right)}
\newcommand{\lbr}{\left\{}
\newcommand{\rbr}{\right\}}

\def\({\left(}
\def\){\right)}

\newcommand{\e}{\varepsilon}
 
\newcommand{\ind}{\mathbf{1}}
\newcommand{\tc}{\, |\, } 
\newcommand{\si}{\sigma}

\begin{document}

\title{Characterizing Gibbs states  for area-tilted Brownian lines}

\author{Mriganka Basu Roy Chowdhury, Pietro Caputo and Shirshendu Ganguly}
\address{Mriganka Basu Roy Chowdhury\\ University of California, Berkeley}
\email{mriganka\_brc@berkeley.edu}
\address{Pietro Caputo\\ Universit\`{a} Roma Tre}
\email{pietro.caputo@uniroma3.it}
\address{Shirshendu Ganguly\\ University of California, Berkeley}
\email{sganguly@berkeley.edu}

\begin{abstract}
Gibbsian line ensembles 
are families of  Brownian lines 
arising in many natural contexts such as the level curves of three dimensional Ising interfaces, the solid-on-solid model, multi-layered polynuclear growth, trajectories of eigenvalues as the entries of the corresponding matrix perform diffusions, to name a few. In particular, %
line ensembles with area tilt potentials play a significant role in the study of wetting  and entropic repulsion phenomena.   
An 
important example 
is a class of non-intersecting
Brownian lines 
above a hard wall,
which are subject to geometrically growing area tilt potentials, which we call the $\lambda$-\emph{tilted line ensemble}, where $\lambda>1$ is the parameter governing the geometric growth. The model was introduced in \cite{CIW18} as a putative scaling limit of the level lines of entropically repulsed solid-on-solid interfaces. While this model has infinitely many lines and is non-integrable, the case of  the single line,  known as the Ferrari-Spohn diffusion, is well studied  \cite{ferrarispohn2005}. 
In this article we address the problem of classifying \emph{all} Gibbs measures for  $\lambda$-tilted line ensembles.  A stationary infinite volume Gibbs measure was already constructed  in \cite{CIW18,CIW19}, and the uniqueness of this translation invariant Gibbs measure was established in  \cite{geomarea}. 
Our main result here is a strong characterization for 
 Gibbs measures of $\lambda$-tilted line ensembles in terms of a two parameter family. Namely, we show that the extremal Gibbs measures are completely characterized by the behavior of the top line $X^1$ at positive and negative infinity, which must satisfy the parabolic growth  
$$X^{1}(t)=t^2+L\,|t|\,\mathbf{1}_{t<0}+R\,|t|\,\mathbf{1}_{t>0}+o(|t|)\,,\quad  \text{ as } \;\;|t| \to \infty,$$ where $L,R$ are real parameters, including $-\infty,$ with $L+R<0$, while the lower lying lines 
remain uniformly confined.  The case $L=R=-\infty$ corresponds to the unique translation invariant Gibbs measure. The result bears some analogy to the Airy wanderers, an integrable model introduced and studied in \cite{10.1214/09-AOP493} in the context of the Airy line ensemble, except in this case only the top line can wander, owing to the geometrically increasing nature of the area tilting factor. 
A crucial step in our proof, which  holds independent significance, is a complete characterization of the extremal Gibbs states associated to a single area-tilted Brownian excursion, describing the equilibrium states of the top line when the lower lines are absent, and which can be interpreted as non-translation invariant versions of the Ferrari-Spohn diffusion.
\end{abstract}
\maketitle{}

\newcommand\PAR{{\textsf{PBR}}}

\setcounter{tocdepth}{1}
\tableofcontents

\thispagestyle{empty}

\section{Introduction and main results}
Random fields consisting of interacting Brownian lines form a significant  class of probabilistic models referred to as line ensembles (LE). These emerge as natural limiting objects of several fundamental  models ranging from statistical mechanics to random matrix theory.  

A key example is Dyson Brownian Motion (DBM)  describing the evolution of the eigenvalues of random matrices from the Gaussian unitary ensemble. Here the interaction is given by the constraint of non-intersection. In a breakthrough work \cite{corwinhammond}, Corwin and Hammond constructed the edge scaling limit of DBM,  a fundamental limiting random field known as the Airy line ensemble, which is a central object in the Kardar-Parisi-Zhang (KPZ) universality class \cite{prahofer2002scale,matetski2021kpz,dauvergne2022directed}. While the determinantal structure of DBM was used to prove finite dimensional convergence, the requisite tightness and regularity estimates were obtained using probabilistic and geometric ideas exploiting a Gibbs resampling invariance property of DBM. 
A family of modified  Airy processes, termed as the Airy wanderers, giving rise to different universality classes, was introduced in \cite{10.1214/09-AOP493}.

\begin{figure}[ht]
        \centering{\qquad\qquad\quad
                \input{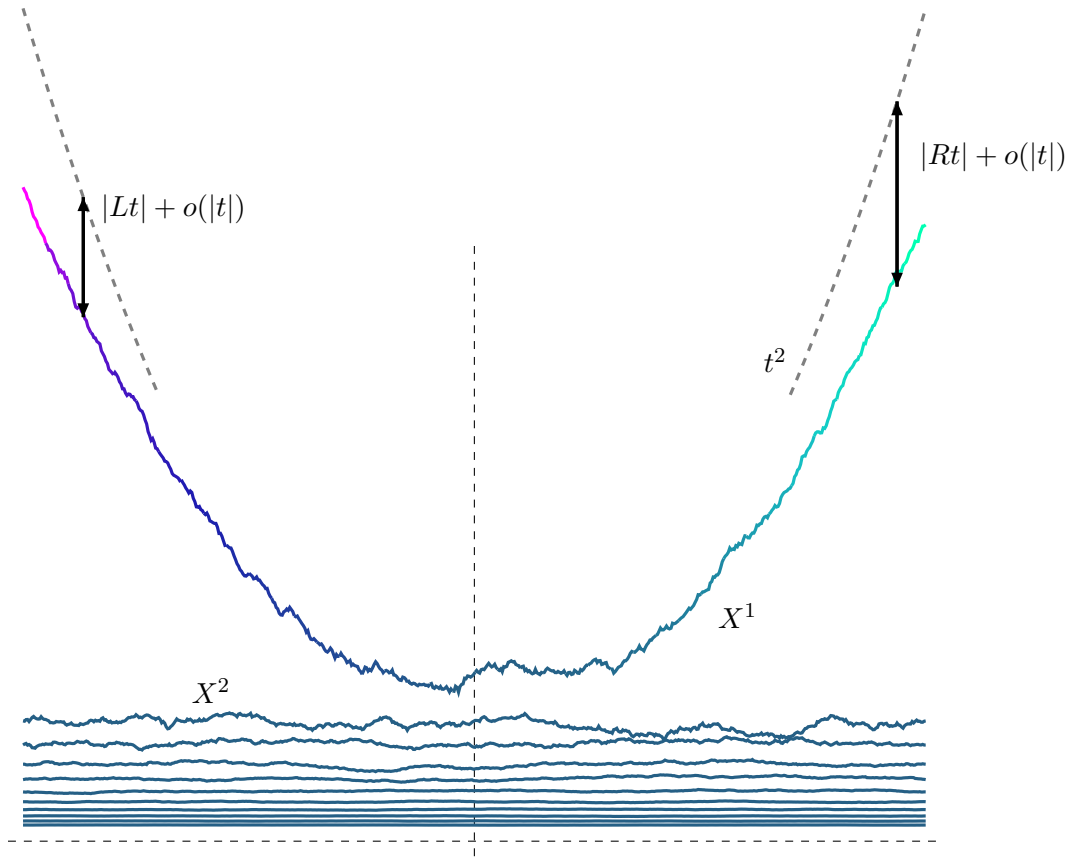}
        \caption{Illustration of an extremal Gibbs state where the top line satisfies  $X^{1}(t)=t^2+L\,|t|\,\mathbf{1}_{t<0}+R\,|t|\,\mathbf{1}_{t>0}+o(|t|)$ as $|t| \to \infty,$ and the lower lines are uniformly confined.}
        }\label{fig:gibbs}
    \end{figure}
 
 A related class of models is obtained by adding an area tilt potential which attracts each Brownian line to an underlying hard wall, while keeping the non-intersection constraint. The case of a single line is already an important one. The main object in this case was constructed by Ferrari and Spohn \cite{ferrarispohn2005} as the scaling limit of Brownian motion with circular or parabolic barriers, and is known as the Ferrari-Spohn (FS) diffusion. It was then shown that the FS diffusion is the natural limiting object for a class of interface models arising in the analysis of the wetting phenomenon in statistical mechanics \cite{ioffeshlosmanvelenik,prewet1,prewet2}.
 
If several lines are considered,  and all lines are subject to the same area tilt potential one obtains another interesting line ensemble known as the Dyson-Ferrari-Spohn (DFS) ensemble, see \cite{ioffevelenikwachtel,ioffe2016low,ferrari2023airy2,dimitrov2023uniform}. The exchangeability of the model, through the Karlin-McGregor formula, allows a determinantal description of the DFS ensemble. 
If instead the area tilts have varying strengths, then the integrability is lost. An example is the model with geometric area tilts introduced in \cite{CIW18, CIW19}, referred to as the $\lambda$-tilted LE, where the $i$-th line, counting from the top, feels a strength proportional to $\lambda^i$, and $\lambda>1$ is a fixed  parameter determining the rate of geometric growth of the strength of the tilting.
The $\lambda$-tilted LE is believed to be the limiting object for the level lines of entropically repulsed interfaces in statistical mechanics models such as the 3-dimensional Ising and the (2+1)-dimensional solid-on-solid (SOS) model at  low temperature, see \cite{bricmont1986random,caputoetal2016,gheissari2023entropic,serio2023scaling,caddeoetal}.  We refer to \cite{CIW18} for a more detailed discussion of the motivation behind the $\lambda$-tilted LE. We also refer to  \cite{ferrari2003step,johansson2005arctic,borodin2010gibbs,ioffe2019formation} for further examples of line ensembles arising as scaling limits of  the level lines of microscopic crystal shapes. 

Due to the rapid increase of the strength of the attraction to the hard wall of the low lying lines, the $\lambda$-tilted LE exhibits a strong form of confinement and there exists a well defined stationary line ensemble with infinitely many lines which can be defined by taking infinite volume limits of fields induced by zero or free boundary conditions, see \cite{CIW18, CIW19,DLZ}. Furthermore, several properties such as uniqueness, ergodicity and optimal tail estimates for this limiting object were recently derived in \cite{geomarea}.   

Since the seminal work of Corwin and Hammond \cite{corwinhammond} it has been become evident that a compelling approach to studying line ensembles is to treat them as Gibbs measures. The local resampling properties inherited by the Markov property of Brownian motion, known as Brownian-Gibbs in this context, is indeed one of the key tools for analyzing such random fields. This leads naturally to the investigation of questions of uniqueness, ergodicity, breaking of translation invariance symmetry and characterization of the extremal Gibbs measures. While these are well studied questions  within the context of random fields indexed by a discrete parameter such as lattice spin systems and random interfaces \cite{georgii2011gibbs,sheffield2005random}, their investigation in the setting of line ensembles is relatively new. We refer to the recent extensive work of Aggarwal and Huang \cite{aggarwal2023strong} for a strong characterization result in the case of the Airy line ensemble.

The aim of this paper is to give a complete characterization of the extremal Gibbs measures for area-tilted line ensembles. The main focus will be on %
$\lambda$-tilted LE consisting of infinitely many lines with geometric area tilt. While the case of stationary Gibbs measures was settled in \cite{geomarea}, which proves a strong uniqueness statement for such line ensembles, the case of non-stationary Gibbs measures remained fully open. Here we reveal an explicit geometric structure which characterizes the non-stationary extremal Gibbs measures in terms of parabolically growing boundary conditions.   
Related results can be obtained for the DFS ensemble with a finite number of lines as well, but we have chosen not to include them in this paper to maintain a manageable length. See however the remark at the end of the introduction for a more specific comment.

We now turn to a detailed description of our main results. A key ingredient in our analysis involves understanding the behavior of the top line of the $\lambda$-tilted LE. Therefore, the first step of our analysis characterizes the possible equilibrium states of the top line when the lower lines are absent. This complete characterization of the extremal Gibbs states associated to a single area-tilted path, which will serve as one of the key inputs for our multi-line analysis, is our first main result. 

\subsection{Gibbs measures for the area-tilted one-line ensemble} 
We start with the definition of the one-line ensemble corresponding to area-tilted Brownian excursions. We let $\cC=\cC(\R,\R)$ denote the space of continuous functions from $\R$ to $\R$, equipped with the topology of uniform convergence over compact sets. Similarly, for real numbers $\ell<r$ we write $\cC_{\ell,r}$ for the set of continuous functions from $[\ell,r]$ to $\R$. We use $\calP$ and $\calP_{\ell,r}$ for the set of probability measures over $\cC$ and $\cC_{\ell,r}$ respectively, equipped with the corresponding Borel $\si$-algebra. We also write $\calP^+\subset \calP$ for the probability measures supported on positive paths.  For a probability measure $\mu$ we sometimes write $X\sim \mu$ if $X$ is a random variable with law $\mu$ and  $ \E_{X \sim \mu}$ for the corresponding expectation. 

\begin{defn}[Area-tilted Brownian excursion] %
 \label{def:one line} Let $\B^{x, y}_{\ell, r}\in \calP_{\ell,r}$ be the law of the Brownian bridge on $[\ell, r]$ pinned at $x$ and $y$ at time $\ell$ and $r$ respectively.
    For any $a\geq 0$, $x,y>0$, the area-tilted Brownian excursion %
    from $x$ to $y$ with strength $a$, denoted by $\L^{x, y}_{\ell, r}$ is the tilted form of $\B^{x, y}_{\ell, r}$ defined by
    \begin{align*}
        \f{d\L^{x, y}_{\ell, r}}{d\B^{x, y}_{\ell, r}}(X) = \f{1}{Z} \,\Exp{-a \int_{\ell}^r X(t) \,\dt} \1_{X \in \Om^+_{\ell, r}}
    \end{align*}
    where $\Om^+_{\ell, r} = \{X\in \cC_{\ell,r}: X(t) > 0,\ \forall t \in (\ell, r)\}$ is the set of continuous positive functions on the interval $(\ell, r)$, and $Z=Z(a;x,y;\ell,r)$ is the normalization factor
    \begin{align*}
        Z =\E_{X \sim \B^{x, y}_{\ell, r}}\Box{\Exp{-a \int_{\ell}^r X(t) \,\dt}\1_{X \in \Om^+_{\ell, r}}}.
    \end{align*}
We sometimes write $\L^{a;x, y}_{\ell, r}$ if the parameter $a$ needs to be emphasized.
Unless otherwise stated, the parameter $a$  will be taken to be equal to 2; see Remark \ref{rem:a=2}. 
\end{defn}

Note that $Z > 0$, and thus $\L^{x, y}_{\ell, r}$ is well defined, for any $x, y > 0$.
    Using a limiting sequence $x \to 0$, it is also possible to define the measures $\L^{0,y}_{\ell, r},\L^{y,0}_{\ell, r}$, and $\L^{0,0}_{\ell, r}$, see e.g.\ \cite{corwinhammond,CIW19}.
We often use the shorthand notation $\B^{x}_{\ell, r}$ for $\B^{x, x}_{\ell, r}$ and %
 $\L^x_{\ell, r}$ for $\L^{x, x}_{\ell, r}$. Moreover, when the time  interval is symmetric we write $\L^{x}_T = \L^{x, x}_{-T, T}$ and $\L^{x, y}_T = \L^{x, y}_{-T, T}$, $T>0$.

\begin{defn}[Area-tilted Brownian-Gibbs measure]
\label{def:abg}
    A probability measure $\mu\in\calP^+$ is said to have the $a-$area-tilted Brownian-Gibbs property %
    for some parameter $a > 0$, if for every $\ell<r \in \R, x, y > 0$, the conditional distribution of $X$ on the interval $(\ell,r)$ given $X$ on $\R\setminus(\ell,r)$ has the law of the area-tilted Brownian excursion with parameter $a$ on $[\ell, r]$ with the endpoints $X(\ell), X(r)$. 
    Formally, if $X \sim \mu$, and $F : \Om^+_{\ell, r} \to \R$ is bounded measurable, then %
    \begin{align*}
        \E[F(X) \mid \cB_{\ell, r}] = \E_{Y \sim \L^{a;X(\ell), X(r)}_{\ell, r}}[F(Y)]
    \end{align*}
    where $\cB_{\ell, r} = \sig(X_t : t \notin (\ell, r))$ is the $\si$-algebra generated by the variables $X(t)$, $t\in\R\setminus(\ell,r)$.
    We call $\cG=\cG(a)\subset \calP^+$ the set of probability measures with the above %
    property. Unless otherwise stated, we take $a=2$.  %

\end{defn}

It follows from general facts that $\cG$ is a simplex, namely that every $\mu\in\cG$ can be  written  uniquely as a convex combination of extremal elements of $\cG$, see e.g.\ Theorem 7.26 in \cite{georgii2011gibbs}.
We denote the latter set by $\cG_{\rm ext}.$
An example of $\mu\in\cG$ is the Ferrari-Spohn state $\FS$. The latter is the law of a stationary, ergodic  process known as Ferrari-Spohn diffusion $\{X(t),\,t\in\bbR\}$, such that $X(0)$
has density proportional to 
\begin{equation}\label{eq:airyfs} 
{\rm Ai}\(\sqrt[3]{2a}\,x-\omega_1\)^2\ind_{x>0},
\end{equation} 
where ${\rm Ai}(\cdot)$ is the Airy function and $-\omega_1$ denotes its largest zero. We refer to \cite{ferrarispohn2005} for the original construction and to \cite{ioffeshlosmanvelenik} for the representation in terms of area-tilted excursions.

For any $h\in\R$, $\mu\in\cG$, we write $\t_h\mu$ for the push forward under time translation by $h$, that is $X(\cdot)\sim \mu$ iff %
 $X(\cdot - h) \sim\t_h\mu$.
The Ferrari-Spohn state is translation invariant: $\t_h\FS=\FS$ for all $h\in\R$.
Our main result concerning one-line ensembles establishes that $\FS$ is the only translation invariant state and that all non-translation invariant extremal states grow parabolically up to a linear correction at infinity and are in fact characterized by their linear correction term.
We introduce the set of pairs of admissible linear slopes,
 \begin{align}\label{lineartilts}
\calT&=\{(L,R)\in[-\infty,\infty)^2:\; L+R<0\} 
\nonumber
\\&=\calT_{-\infty,-\infty}\cup\calT_{-\infty,\R}\cup \calT_{\R,-\infty} \cup 
\{(L,R)\in \R^2:\; L+R<0\},
    \end{align}
where $\calT_{-\infty,-\infty}= \{-\infty\}\times\{-\infty\}$, while $\calT_{-\infty,\R}= \{-\infty\}\times \R$ and 
$\calT_{\R,-\infty}= \R\times\{-\infty\}$.
The following theorem states in particular that there is a bijection $(L,R)\leftrightarrow \nu_{L,R}$ between  $\calT$ and $\cG_{\rm ext}$, and that the map is such that the path $X_{L,R}$ with law $\nu_{L,R}$ has the asymptotic behavior
    \begin{align}\label{lineartiltas}
X_{L,R}(t) &= t^2 + L |t| +o(|t|)\,, \quad t\to -\infty\,,
\\
X_{L,R}(t) &= t^2 + Rt +o(t)\,, \quad t\to +\infty.
    \end{align}

\begin{thm}[The set $\cG_{\rm ext}$]\label{thmoneline}
For any $\mu\in  \cG_{\rm ext}$, there exists a pair $(L,R)\in\calT$ such that %
the following two limits exist $\mu$-almost surely: %
    \begin{align}\label{linearlim}
        L %
        = \lim_{t \to -\infty} \f{X(t) - t^2}{|t|}, \qquad 
        R %
         =\lim_{t \to \infty} \f{X(t) - t^2}{t}.
    \end{align}
Conversely, for any $(L,R)\in\calT$ there exists a unique $\nu_{L,R}\in \cG_{\rm ext}$ such that \eqref{linearlim} holds for $\mu = \nu_{L,R}$. If $(L,R)\in\calT_{-\infty,-\infty}$, then $\nu_{-\infty,-\infty}=\FS$,
 and $\FS$ is the only translation invariant element of $\cG$.
  Moreover, for any $h\in\R$, and $(L,R)\in\calT$, 
  \begin{align}\label{nulrshifts}
\theta_h\nu_{L,R}= \nu_{L+2h,R-2h},
   \end{align}
   with the convention $-\infty+x=-\infty$, for all $x\in\R$. 
\end{thm}

It will be convenient to adopt the notation $\mu_K = \nu_{-2K,-2K}$, which, by the existence and uniqueness assertion above, is a well defined extremal Gibbs state for all $K\in(0,\infty]$. Our proof will also show that $\mu_K$ is  symmetric with respect to reflection around the origin. Note that any $\mu \in\cG_{\rm ext}$, except for the special cases of $\nu_{L,R}$ with $(L,R)\in\calT_{-\infty,\R}\cup \calT_{\R,-\infty} $,  can be recovered as a time shift  of $\mu_K$, for some %
$K \in (0,\infty]$ since, according to \eqref{nulrshifts}, for any $h\in\R$  and $K\in (0,\infty]$  %
one has 
 \begin{align}\label{nulrshiftsK}
\theta_h\mu_K  = \nu_{-2K+2h,-2K-2h}.
   \end{align}

As we will see, a natural way to construct the measures $\nu_{L,R}$ is by taking the infinite volume limit $T\to\infty$  of the measures $\L^{x_T,y_T}_T$ on the interval $[-T,T]$ with boundary conditions
\begin{equation}\label{eq:bc_parab}
x_T = T^2 + L T \,,\qquad y_T = T^2 + R T\,,
\end{equation}
with the case $L=-\infty$ (resp. $R=-\infty$) obtained by a suitable limiting procedure. 
Finally, as the proof will show, it will turn out that when $L$ and $R$ are finite, the $o(|t|)$ terms in \eqref{lineartiltas} will essentially be diffusive, i.e., $O(\sqrt{|t|}),$ owing to Brownian fluctuations.

\begin{rmk}\label{rem:a=2}
We will make use of an equivalent representation of the area-tilted excursion $X\sim\L^{x,y}_T$ as the path $Y+\p$ where $\p$ is the parabola 
\[
\p:\,t\mapsto t^2\,,
\] and $Y$ has the law of the Brownian bridge $\B^{x', y'}_{T}$ on the interval $[-T,T]$, with boundary data $x'=x-T^2, y'=y-T^2$, conditioned to stay above the negative parabola $-\p$. %
This representation is an immediate consequence of Girsanov's formula, see Lemma \ref{lem:girsanov} below.
We remark that if the strength $a>0$ is taken to be different from $2$ then the same representation holds with the parabola $\p$ %
replaced by $\p^{(a)}:\,t\mapsto \frac{a}2\,t^2$. In particular, the same statements from Theorem \ref{thmoneline} hold, but we have to replace the parabolas $T^2 + L T$ and $T^2 + R T$ in \eqref{eq:bc_parab} and \eqref{linearlim} by 
$\frac{a}2\,T^2 + L T$ and $\frac{a}2\,T^2 + R T$.
\end{rmk}

\bigskip
We turn to a description of our main results concerning the  $\lambda$-tilted LE with infinitely many lines.

\subsection{Gibbs measures for infinitely many lines with geometric area tilt}
We start by defining, analogous to the one-line case, the Gibbs measures corresponding to a finite interval with finitely many lines. 
For  $\ell < r$, $n\in\bbN$, and $\ux,\uy\in\bbR^n$, 
let ${\mathbf B}^{\ux,\uy}_{n;\ell , r} $ denote the law of $n$ independent Brownian bridges $\uX=(X^1(s),\dots,X^n(s))$, $s\in[\ell , r]$, pinned at 
$X^i(\ell)=x_i$ and $X^i(r)=y_i$, $i=1,\dots,n$.
   Let $\A_n^+ \subseteq (\R^+)^n$ be defined  by
    \begin{align*}
      \A_n^+ = \{ \ux \in \R^n : x_1 > x_2 > x_3 > \ldots > x_n > 0 \},
    \end{align*}
    and write $\Omega^{+}_{n;  \ell, r}$ for the set of positive non-intersecting $n$-tuples  $\uX$, 
\begin{equation}\label{eq:Omega-Set} 
\Omega^{+}_{n; \ell , r} = \lbr \uX\in(\cC_{\ell,r})^n:~ \uX (t )\in \bbA_n^+\,,\; \; \forall\, t\in (\ell , r)\rbr.
\end{equation}
This set is equipped with the topology of uniform convergence over compact sets, and   with the corresponding Borel $\si$-algebra. 
Let $\calP_{n;\ell,r}^{+}$ denote the set of probability laws on $\Omega^{+}_{n; \ell , r}$.

The potential is such that the $i$-th path $X^i$ has an area tilt of the form 
 \begin{equation}\label{eq:areatilt} 
\exp\left(-a\lambda^{i-1}\int_{\ell}^rX^i(s) d s
\right)\,,\quad i=1,\dots,n,
\end{equation}
where $a>0$ and $\lambda>1$ are fixed constants. As before, without loss of generality we fix $a=2$ unless otherwise stated. Everything in the sequel will hold for any fixed value of $\lambda>1$, which will often be omitted in our notation.

\begin{defn}[Finite $\lambda$-tilted LE]  \label{def:many-lines}
Given 
 $n\ge 1$, and $\ux,\uy\in\bbA_n^+$, $\ell<r$, we consider the probability measure $\L^{\ux, \uy}_{n;\ell, r}\in \calP_{n;\ell,r}^{+}$ defined by 
    \begin{align}\label{eq:finitegibbs}
        \f{d\L^{\ux, \uy}_{n;\ell, r}}{d\B^{\ux, \uy}_{n;\ell, r}}(\uX) = \f{1}{Z} \,\Exp{-a \sum_{i=1}^n\lambda^{i-1}\int_{\ell}^r X^i(t) \,\dt} \1_{\uX \in \Om^+_{n;\ell, r}}
    \end{align}
    where $Z=Z(a,\lambda;n;\ux,\uy; \ell,r)$ is the normalizing constant  
        \begin{align*}
        Z =\E_{\uX \sim \B^{\ux, \uy}_{n;\ell, r}}\Box{\Exp{-a
        \sum_{i=1}^n\lambda^{i-1}\int_{\ell}^r X^i(t) \,\dt }\1_{\uX \in \Om^+_{n;\ell, r}}
        }.
    \end{align*}
The measure $\L^{\ux, \uy}_{n;\ell, r}$ and the corresponding random paths $\uX\in\Omega_{n;\ell,r}^{+}$ are called the $\lambda$-tilted LE with $n$ lines. We write $\L^{a;\ux,\uy}_{n;\ell, r}$ if the overall strength  $a$ needs to be emphasized, and will tacitly assume that $a=2$ otherwise. 
\end{defn}
Clearly, when $n=1$ we are back to the one-line model from Definition \ref{def:one line}. The partition function $Z$, and therefore the measure $\L^{\ux, \uy}_{n;\ell, r}$ is well defined for all $\ux,\uy\in\bbA_n^+$, and one can take limits to obtain a well defined measure with $\ux$ and $\uy$ in the closure of $ \bbA_n^+$. In particular, the measure $\L^{\underline 0,\underline 0}_{n;\ell, r}$ with zero boundary conditions is well defined, see \cite{CIW18}. As in the one-line case, we often use  the shorthand notation $\L^{\ux}_{n;T}$ for $\L^{\ux,\ux}_{n;-T, T}$. 

In order to define the infinite line ensemble we introduce the notation $\A^+_\infty$ for the $n = \infty$ version of the set $\A^+_n$. The sample space is   
$\Omega^+_{\infty} = \cC\lb \bbR,\bbA_\infty^+\rb$, the set of continuous functions $f:\bbR\mapsto\bbA_\infty^+$, regarded as an infinite collection of ordered paths, equipped with the topology of uniform convergence of any finite number of paths on compact time intervals, and   with the corresponding Borel $\sigma$-field $\calB$. Similarly, $\Omega^+_{\infty; \ell , r}= \cC\lb [\ell,r],\bbA_\infty^+\rb$ is the sample space for finite time intervals.  
 We write $\calP^+_\infty$ and $\calP^+_{\infty;\ell,r}$ for the set of probability measures on $\Omega^+_{\infty}$ and $\Omega^+_{\infty;\ell , r}$ respectively. 
For each $n\in\bbN$, and finite time interval $[\ell,r]\subset\bbR$, let 
\begin{gather}\label{eq:Bext} 
\calB_{n; \ell , r}^{\sfe} 
= \sigma\lb X^i (t ):\, \text{{either $t{\notin} (\ell , r)$ {or} $i >n$}}\rb
\end{gather}
denote the %
external $\sigma$-algebra. %
Given %
a continuous function $h:\bbR\mapsto \bbR_+$, we write $\L_{n ;\ell , r}^{\ux, \uy}[h]\in\calP_{n;\ell,r}^{+}$
for the law of the $n$-line ensemble with floor $h$, that is the measure obtained by conditioning $\L^{\ux, \uy}_{n;\ell, r}$
in \eqref{eq:finitegibbs} on the event that the lowest path $X^n$  satisfies $X^n(s)> h(s)\,,\;\forall s\in(\ell,r)$.
\begin{defn}[$\lambda$-tilted LE]\label{def:laBG}
A probability measure $\mu\in\calP_\infty^+$ is said to have the 
{\em Brownian-Gibbs %
property} 
with respect to $(a , \lambda )$-geometric area tilts (or in short simply the BG property) %
if 
for any bounded measurable $F:\Omega^+_{\infty;\ell,r}\mapsto \bbR$, the corresponding conditional expectations $\bbE[\cdot\tc\calB^{\sfe}_{n ;\ell , r}] $ satisfy 
\begin{equation}\label{eq:BG-prop} 
\bbE\left[ \, F(\uX)\tc
\calB^{\sfe}_{n ;\ell , r} \right] 
= 
\bbE_{\uY \sim
 \L_{n ;\ell , r}^{\uX^{\le n} (\ell) , \uX^{\le n} (r )}[X^{n+1}]}
 \left[\,F(\uY\,,\uX^{>n}) 
\right], 
\end{equation}
$\mu$-a.s. for any $-\infty <\ell <r <\infty$ and $n\in \bbN$.
In \eqref{eq:BG-prop}, we use the notation $\uX^{\le n}=(X^1,\dots,X^n)$ and $\uX^{>n}=(X^{n+1},X^{n+2},\dots)$.  A probability measure $\mu\in\calP_\infty^+$ with the above BG property is called a $\lambda$-tilted Gibbs measures, or $\lambda$-tilted LE. 
 \end{defn} 
We note that the BG property only specifies the conditional law of finitely many paths on a finite domain. Thus, if  $\underline{X}\sim \mu$ satisfies the BG property then for any constant $c>0,$   $\underline{X}+c=(X^{i}+c)_{i\ge 1}$ also satisfies the BG property, since the area increase for each curve on a given finite domain is deterministic, and the Brownian bridge density is invariant under shifting by a constant. This shift can be also thought of as raising the hard floor from $0$ to $c$. To ensure that the floor indeed stays at zero, throughout this article we will be always considering LEs satisfying the following property.

\begin{defn}[Asymptotically pinned to zero and the set $\cG_\infty$] %
\label{def:asympin} A probability measure $\mu\in\cP^+$ is said to be asymptotically pinned to zero %
if for any $\e>0$ and $T>0,$ there exists $k=k(T,\e)$ such that 
\begin{equation}\label{eq:asymppin}
\bbP\left(\sup_{s\in [-T,T]}X^{k}(s)\le \e\right)\ge 1-\e.
\end{equation}
We write $\cG_\infty=\cG_{\infty}(a,\lambda)$ for the set of $\lambda$-tilted Gibbs measures which are asymptotically pinned to zero.  
\end{defn}
 As before, the set $\cG_\infty$ is a simplex and we write $\cG_{\infty,\rm ext}$ for the set of extremal Gibbs measures.  An example of a Gibbs measure %
 $\mu\in\cG_\infty$ 
 can be obtained by taking the limits as $T\to\infty $ and $n\to\infty$ (in arbitrary order) of the zero boundary measures $\L^{\underline 0}_{n;T}$, see \cite{CIW19}. We call $\mu^0$ this Gibbs measure, and refer to it as the 
 zero boundary LE.  Moreover, it is known that the same measure can be obtained by taking limits with free or bounded boundary conditions, see \cite{DLZ} and \cite{geomarea}. It was also shown in \cite{geomarea} that $\mu^0$ is mixing, and thus ergodic, and is the only stationary element of $\cG_\infty$. We write again $\theta_h$ for the time translation acting on probability measures $\calP_\infty^+$, so that $\theta_h\mu^0=\mu^0$ for all $h\in\R$. 
 
  Our main result is a complete characterization of the set $\cG_{\infty,\rm ext}$. Perhaps somewhat surprisingly on first glance, it turns out that these measures are uniquely determined by the behavior at infinity of their top line $X^1$. As in the case of the one-line ensemble, the key identifying feature is the linear deviation from the main parabolic behavior. In fact, the statement is essentially identical to that of Theorem \ref{thmoneline} provided one replaces the single line by the top line $X^1$. On the other hand the lower lying paths satisfy a strong form of confinement. In particular, letting $\calT$ be the set of linear slopes defined by \eqref{lineartilts}, 
  we establish a bijection $(L,R)\leftrightarrow \nu_{\infty,L,R}$ between  $\calT$ and $\cG_{\infty,\rm ext}$, such that the paths $\uX_{L,R}$ with law $\nu_{\infty,L,R}$ have the asymptotic behavior
    \begin{gather}\label{ltiltsandconf}
X^1_{L,R}(t) = t^2  +L |t| +o(|t|)\,, \quad t\to -\infty\,,\nonumber
\\
X^1_{L,R}(t) = t^2 + Rt +o(t)\,, \quad t\to +\infty\,,\nonumber\\
\lambda ^{k/3} X_{L,R}^k(t)= O(1)\,, 
    \end{gather}
    uniformly in $k\ge 2$.

\begin{thm}[The set $\cG_{\infty,\rm ext}$]\label{thmmanylines}
For any $\mu\in  \cG_{\infty,\rm ext}$, there exists a pair $(L,R)\in\calT$ such that %
the following two limits for the top path $X^1$ exist $\mu$-almost surely: %
    \begin{align}\label{linearlimla}
        L %
        = \lim_{t \to -\infty} \f{X^1(t) - t^2}{|t|}, \qquad 
        R %
         =\lim_{t \to \infty} \f{X^1(t) - t^2}{t}.
    \end{align}
Conversely, for any $(L,R)\in\calT$ there exists a unique $\nu_{\infty,L,R}\in \cG_{\infty,\rm ext}$ such that \eqref{linearlimla} holds for $\mu = \nu_{\infty,L,R}$. When $L=R=-\infty$, then $\nu_{\infty,-\infty,-\infty}=\mu^0$ is the zero boundary field,
 the only translation invariant element of $\cG_\infty$.
  Moreover, for any $h\in\R$, and $(L,R)\in\calT$, 
  \begin{align}\label{nulrshiftsla}
\theta_h\nu_{\infty,L,R}= \nu_{\infty,L+2h,R-2h},
   \end{align}
   with the convention $-\infty+x=-\infty$, for all $x\in\R$. 
\end{thm}

As the proof of Theorem \ref{thmmanylines} will show, the $o(|t|)$ terms in \eqref{ltiltsandconf} can be shown to be fluctuation terms of order $\sqrt {|t|}$. On the other hand,  the confinement of the lower lying paths $X^j$, $j\ge 2$ in 
  \eqref{ltiltsandconf} is quantified by the following statement. 

\begin{thm}[Confinement of lower lying paths]\label{thmmanylinesconf}
There exists a constant $C$ such that
for all integers $ k\ge 1$, and $S>0$, for all %
$\mu \in \cG_{\infty,\rm ext}$, if $\uX\sim \mu$, then 
\begin{gather}
\label{eq:tight3aq}
\sup_{s\in\R}\,
\bbE\left[\max_{u\in[-S,S]}X^{k+1}(s+u) \right]\le C\lambda^{-k/3}[1+\log(1+S\lambda^{2k/3})].
\end{gather}
Moreover, there exist  constants $C,c>0$ such that for all $k\ge 1$,
 for any  $\mu \in \cG_{\infty,\rm ext}$, the $(k+1)$-th line of  $\uX\sim \mu$ satisfies, for all $t>0$,  
\begin{equation}\label{eq:stretched}
\sup_{s\in\R}\,\bbP\left(X^{k+1}(s)> t\lambda^{-k/3}\right)\le C\exp{\left(-c 
\,t^{3/2}\right)}\,. %
\end{equation}
\end{thm}
The estimates in Theorem \ref{thmmanylinesconf} are essentially optimal. They were shown to hold for the zero boundary LE $\mu^0$ in \cite{geomarea}, in which case they include the top line as well, that is the bounds above extend to any $k\ge 0$. The important feature here is that while by Theorem \ref{thmmanylines} any element of $\cG_{\infty,\rm ext}$ except $\mu^0$ is not translation invariant, with the top path growing parabolically at infinity, all of its lower lying paths $X^{\ge 2}$ still behave roughly as the lower lying paths of $\mu^0$. %

In analogy with the one-line case, one way to construct the measures $\nu_{\infty,L,R}$ is by taking the limits $T\to\infty$ and $n\to\infty$  of the measures $\L^{\ux_T,\uy_T}_{n,T}$, the $\lambda$-tilted LE %
on the interval $[-T,T]$, with boundary conditions
\begin{equation}\label{eq:bc_parabla}
\ux_T \equiv T^2 + L T \,,\qquad \uy_T \equiv T^2 + R T\,.
\end{equation}
Here one takes all paths with the same parabolically growing height. However, as will be seen, the only relevant boundary is the one concerning the top path $X^1$ and the same limiting measure $\nu_{\infty,L,R}$  is achieved by taking the top path with the prescribed parabolic growth while the lower lying paths can be given essentially any other compatible boundary value. %

\subsection{Proof ideas and related work} We end this section with a brief overview of the main proof ideas and their relation to previous work.
For the one-line case in Theorem \ref{thmoneline}, %
a key role is played by the alternate representation in terms of a Brownian bridge with a parabolic barrier, see Remark \ref{rem:a=2}, which we term as the parabolic barrier representation ($\PAR$). 
We will sometimes refer to this as the ``dual'' picture, consequently referring to the original description as the ``primal'' picture.  Such a representation had originally appeared in the seminal work  \cite{ferrarispohn2005} where the FS diffusion was first introduced and investigated. More recent applications of this representation can for instance be found in \cite{CIW18,dimitrov2023uniform}.
  \smallskip
    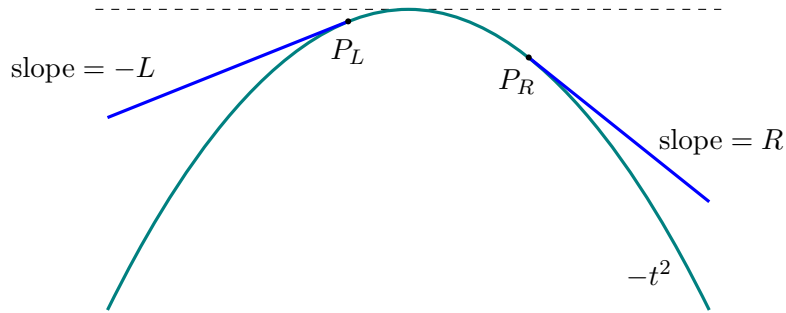
\begin{figure}[ht]
        \centering
                \begin{tikzpicture}[scale=.8]
\colorlet{colora}{magenta}
\colorlet{colorb}{teal}
\colorlet{colorc}{blue}
\draw[dashed] (-5.2,0) -- (5.2,0);
\draw[colorb,very thick] (-5,-5) parabola bend (0,0) (5,-5);
\draw[very thick,colorc] (-5,-1.8) -- (-1,-0.2);
\draw[very thick,colorc] (5,-3.2) -- (2,-0.8);
\draw (5.2,-2.2) node {$\mathrm{slope} = R$};
\draw (1.8,-1.2) node {$P_R$};
\draw (-5.4,-1.0) node {$\mathrm{slope} = -L$};
\draw (-1.0,-0.7) node {$P_L$};
\draw (4.0,-4.4) node {$-t^2$};
\fill (-1,-0.2) circle[radius=0.05];
\fill (2,-0.8) circle[radius=0.05];
\end{tikzpicture}
        \caption{Sketch of the macroscopic behavior of $Y=X-\p$, where $X\sim \nu_{L,R}$,  $L+R<0$, and $-\p$ is the negative parabola $-t^2$. The tangency locations are $P_L=(L/2,-L^2/4)$ and 
    $P_R=(-R/2,-R^2/4)$.     %
        }
        \label{fig:hydrodynamic}
    \end{figure}
This leads to a rather geometric way of interpreting Theorem \ref{thmoneline}. Namely, for any Gibbs state $X\sim \nu_{L,R}$, with $L+R<0$, the associated path $Y=X-\p$ in the $\PAR$ 
corresponds to a hydrodynamic picture consisting of a pair of straight lines with slopes $-L$ and $R$. These lines are tangent to the negative parabola $-\p$ at points $P_L,P_R$, with horizontal locations $L/2$ and $-R/2$, respectively; see Figure \ref{fig:hydrodynamic}. To the left of $P_L$ and to the right of $P_R$, %
the path $Y$ follows essentially two independent Brownian trajectories with slopes $-L$ and $R$ respectively, while non-trivial interactions with the parabola $-\p$ primarily occur on the interval $[L/2, -R/2].$ As $L,R\to-\infty$, the measure converges to the stationary FS diffusion. 
Moreover, the above picture should also convince the reader that if the slopes were such that $L+R\ge 0$, then the height $Y(0)$ at the origin would blow up, and thus tightness would be lost. This explains, at least intuitively, why $L+R$ must be negative. Such tangency considerations were also central in the recent works  \cite{aggarwal2020arctic,ganguly2022sharp}.

The proof of existence in Theorem \ref{thmoneline}, say for the symmetric case $L=R=-2K$, involves looking at finite measures $\mu_{K,T}$ on $[-T,T]$ with boundary conditions growing as $T^{2}-2KT$ and showing that these sequences of measures locally converge weakly. A key ingredient is to show that for large times $T_1 \gg T_2\gg s$ the one-point  marginal for both $\mu_{K,T_1}$ and $\mu_{K,T_2}$ at time $s$,  %
is close in total variation distance to a Gaussian variable with mean $s^2 - 2Ks$ and variance $s.$
Without going into more details, let us mention that the proof of uniqueness also proceeds by proving the last statement about the marginal Gaussian behavior even for a general one-line Gibbs state simply under the assumption of the correct asymptotic behavior.

With this as an input, we now briefly overview the pieces involved in the analysis of the multi-line case. We stress that, since the model is not integrable, we're compelled to make extensive use of probabilistic tools like monotonicity and coupling techniques.  At a very high level the key observation is the following. Note that for the single line case, any growth rate smaller than $\p$, say by a super-linear correction term, implies that it must then be  the stationary FS diffusion and hence uniformly confined. On the other hand, growth faster than $\p$ results in the loss of tightness. %
Given this, let us  consider the case of simply two lines $(X^1, X^2)$ where the area tilt factors for $X^1$ and $X^2$ are $2$ and $2\lambda>2$ respectively. Since $X^1$ is bound to be below $\p$ owing to tightness, $X^2$ is also below $\p$. However, applying the above reasoning to $X^2,$ one observes that growth slower than $\lambda \p$ ought to imply uniform confinement, and thus while $X^1$ may grow roughly as much as $\p$, $X^2$ must remain bounded. 
This rationale continues to apply even %
when there are more than two lines, explaining why, in the multi-line case, only  the top line can exhibit asymptotic growth. 

Thus our approach to proving Theorem \ref{thmmanylines}, and Theorem \ref{thmmanylinesconf}, involves decoupling the top line from the remaining lines. The mathematical implementation of this however needs to consider the countering upwards push, say, the second curve faces from the third curve etc.  The analysis of this involves a delicate inductive framework, iteratively bounding the $i^{th}$ curve, given control on the $(i+1)^{th}$ curve. Once this is achieved, the arguments for the one-line case are, at least in spirit, back in play again. However, several technical issues remain. 

Let us offer a brief glimpse of the main steps involved in the arguments. For the proofs of existence, we introduce a family of Gibbs measures on finite domains $[-T,T]$ with boundary conditions chosen in such a way that: a)  %
the local restrictions, say on a compact interval, form an approximately monotone sequence as $T \to \infty$, and b) the top line $X^1$ has the prescribed parabolic behavior. This allows us  to obtain weak convergence to a Gibbs measure with the the correct asymptotic behavior for the top path.
Regarding uniqueness, recall that the proof in the one-line case involved showing that regardless of the exact boundary values, the marginal distribution at an interior point was close to a Gaussian with a certain mean and variance. This along with the Gibbs property was enough to complete the proof. In the infinite line case however, one has to consider the locations of the infinitely many lines. A key idea that we employ here is a reverse coupling strategy. Namely, if the boundary values of $\uX\sim \mu_1$  are higher than those of $\uY \sim \mu_2$, say at some given  interval $[-T,T],$ then under the usual monotone coupling one has the stochastic domination $\uX \succ \uY.$ However, if the boundary conditions do not differ much, then one can construct a reverse coupling, under which with high probability, $\uY\succ \uX$ on a mesoscopically large interval $[-T^c, T^c]$, for some $c\in(0,1)$, which will suffice for our purposes.  Such a strategy also featured in the prior work \cite{geomarea}.

It might be instructive to state our results in the language of  the very recent work \cite{aggarwal2023strong}, which, settling a long-standing problem, provides a  strong characterization of the Airy LE in terms of the parabolic behavior of the top path. In this context the Brownian-Gibbs property states that the conditional distribution of any finite number of lines on a compact domain given everything else is simply that of non-intersecting Brownian lines respecting the boundary data. The main result in  \cite{aggarwal2023strong}
says that if $\uX$ is an infinite LE with the Brownian-Gibbs property,  such that %
for any $\e>0$ there exists a constant $\kappa(\e)>0$ with 
\begin{equation}\label{aghuang1}
\P[\left|X^{1}(t)+\p(t)\right|\le \e t^2 +\kappa(\e)] \ge 1-\e,\qquad t\in\R,
\end{equation} 
then $\uX$ is the parabolic Airy LE, up to an independent random shift. 
A potential analogous statement in our setting maybe formulated as stating that the growth rate of $X^{1}$ smaller than $(1-\e)\p$ is enough to force it to be the FS diffusion in the one-line case and $\mu^0$ in the multi-line case. However as Theorems \ref{thmoneline} and \ref{thmmanylines} show, the same conclusion holds under the weaker assumption $X^{1}(T)\le T^2-\omega(T),$ where $\omega(T)/T\to\infty$ as $T\to\infty$. Thus, perhaps the more accurate formulation of our result in the language of \eqref{aghuang1} 
 for the multi-line case rather reads as follows. 
If $\uX\sim \mu\in\cG_{\infty}$ is such that for any $M>0$ there exists a constant $\kappa(M)>0$ with 
\[
\P\left(X^1(t) - \p(t) \leq -M|t| +  \kappa(M)\right) > 1- \tfrac1M\,,\qquad  t\in\R\,,
\]
then $\uX\sim \mu^0$ is the zero boundary $\lambda$-tilted LE. Moreover, Theorem \ref{thmmanylines} shows that the following refined statement holds as well:  If $\uX\sim \mu\in\cG_{\infty}$ is such that for any $\e>0$ there exists a constant $\kappa(\e)>0$ with 
\[
\P\left(\left|X^1(t) - \p(t) + L|t| \ind_{t<0} + R|t|\ind_{t>0}\right| \,\le \,\e |t| + \kappa(\e)\right) > 1- \e\,,\qquad  t\in\R\,,
\]
for some constants $L,R$ with $L+R<0$, then $\uX\sim \nu_{\infty,L,R}$ is the unique Gibbs measure with the prescribed asymptotic behavior for the top line $X^1$. 
Nonetheless, there are significant differences between the settings of \cite{aggarwal2023strong} and this article. The most major one being that for the Airy LE, there is no floor and hence no a priori control on the lower curves. Thus, much of the work in \cite{aggarwal2023strong} is devoted to establishing the requisite control on the lower curves starting simply with the hypothesis \eqref{aghuang1}, by developing novel estimates involving DBM. On the other hand, while DBM estimates are not available in our case, the presence of the floor and 
Definition \ref{def:asympin}  automatically provides some control on the lower curves.

We conclude this section with a brief remark on related examples of non translation invariant Brownian-Gibbs measures. 
As alluded to above,
a class of non-stationary measures related to the Airy LE, called the Airy wanderers, was investigated in \cite{10.1214/09-AOP493}.  There, using integrable techniques, non-stationary Gibbs measures were constructed in which the first few lines 
follow linear paths instead of parabolic ones. 
Related to the Airy LE, and more aligned with the setting of our paper, is the DFS diffusion, where the area tilt factor for each path  is set to one, resulting in exchangeability and, consequently, integrable properties, see  \cite{ioffevelenikwachtel,ferrari2023airy2}. In this scenario, while such integrable inputs may be used, the purely geometric and probabilistic arguments developed in our paper should already allow a characterization of all Gibbs states for the DFS with an arbitrary finite number of lines.  
Specifically, if there are $k$ lines, each of them can have its own left and right linear drifts away from the common parabolic growth, and non-intersection  forces the drift values $(L_i,R_i)$ for the $i^{th}$ path to be ordered across $1, \ldots, k,$ giving rise to a $2k$ parameter family of measures. 
The edge limit in the case $k \to \infty$ for the stationary DFS  was shown to be the Airy LE in \cite{dimitrov2023uniform}. The non-stationary measures in principle can lead to infinitely many wanderers. 
We do not pursue this formally in this article deferring it to future work.

\subsection{Organization of the article} In Section \ref{sec2}, we collect a preliminary set of results which will appear repeatedly throughout the sequel. Sections \ref{sec:asymp} and \ref{sec:proofofth1} are devoted to the proof of Theorem \ref{thmoneline}. In particular, in Section \ref{sec:asymp} we establish the asymptotic behavior of Gibbs states for the one-line case, whereas in Section \ref{sec:proofofth1} we address the questions of existence and uniqueness and prove Theorem \ref{thmoneline}. Section \ref{sec5} treats the infinite line case, proving Theorem \ref{thmmanylines} relying on the key Proposition \ref{ppn:refinedbnd}. The proof of the latter is the content of Section \ref{sec:inductive}. 

\subsection{Acknowledgements} The authors thank Senya Shlosman for asking a question that led to the investigation carried out in this article. The research was initiated when PC was visiting UC Berkeley as a Miller Visiting Professor. MBRC was supported in part by NSF grant DMS-$1855688$. SG was partially supported by NSF grant DMS-$1855688$, NSF CAREER grant
DMS-$1945172$, and a Sloan Research Fellowship.

\section{Preliminaries}\label{sec2}

In this section we record a number of preliminary facts %
that will be repeatedly used throughout the rest of the article. 
\subsection{Monotonicity}
Recall the definition \eqref{eq:Omega-Set} of the sample space $\Omega^+_{n;\ell,r}$ and the set of probability measures $\calP^+_{n;\ell,r}$. 
Stochastic domination is defined w.r.t.\ the following partial order on $ \Om^+_{n;\ell, r} $:
\begin{equation}\label{eq:partialorder}
\uX\preceq \uY\;\;\iff \;\;X^i(t)\le Y^i(t)\,,\;\quad\forall t\in(\ell,r)\,,\; \forall i=1,\dots,n\,.
  \end{equation}
  If the inequalities are strict, we use $\prec$ instead of $\preceq$.
For two line ensembles $\uX\sim \mu$ and $\uY\sim \nu$, with $\mu,\nu\in\calP^+_{n;\ell,r}$, we write $\mu\succeq \nu$ if there exists a coupling $\Gamma$  of $(\mu,\nu)$ such that $\Gamma(\uX\succeq \uY)=1$.
We also use the notation $\ux\preceq \uy$ iff $x_i\leq y_i$ for all $i=1,\dots,n$ when $\ux,\uy\in\R^n$. 
If $\uZ,\uW\in\Om^+_{n;\ell, r}$ are fixed, with $\uZ\prec\uW$, we write $\L^{a; \ux, \uy}_{n;\ell, r}[\uZ,\uW]$ for 
the probability measure $\L^{a; \ux, \uy}_{n;\ell, r}$
conditioned on the event $\{\uZ\prec \uX\prec \uW\}$.
The main monotonicity properties of our line ensembles are expressed by the following lemma which appears as Lemma 2.3 in \cite{geomarea}.
\begin{lem}\label{lem:monotonicity}
For any $\ell<r$,  $\ux\prec \ux'$, $\uy\prec \uy'$,  $a\geq a'\geq 0$, $n\in\N$, and $\uZ\prec \uZ'$, $\uW\prec \uW'$, with $\uZ\prec \uW$ and $\uZ'\prec \uW'$, one has
\begin{equation}\label{eq:FKG}
\L^{a; \ux, \uy}_{n;\ell, r}[\uZ,\uW]\preceq \L^{a'; \ux', \uy'}_{n;\ell, r}[\uZ',\uW'].
\end{equation}
\end{lem}

\subsection{Scaling}
For any fixed $\lambda>1$ and $n\in\bbN$, the probability measures $\L^{a; \ux, \uy}_{n;\ell, r}$ satisfy the following scaling relation, inherited from Brownian scaling and the linearity of area tilt potentials. We refer to Lemma 1.1 in \cite{CIW18}
for a proof. 
\begin{lem} 
\label{lem:scaling} 
Fix $\lambda>1$ and $n\in\bbN$. For all $a>0$, and $ \ux , \uy\in\bbA_n^+$, for all $\ell<r$, one has
\[
\uX\sim \L^{a; \ux, \uy}_{n;\ell, r}\;\;\Longleftrightarrow\;\;\uY\sim 
\L^{a\lambda; \lambda^{-1/3}\ux, \lambda^{-1/3}\uy}_{n;\lambda^{-2/3}\ell, \lambda^{-2/3}r},
\] 
where $\uY$ is the rescaled field 
\begin{equation}\label{eq:BrownScale}  
\uY (t ) = \frac{1}{\lambda^{1/3}} \uX (\lambda^{2/3}t ) \,,\qquad t\in[\lambda^{-2/3}\ell,\lambda^{-2/3}r]\,.
\end{equation}
\end{lem}
We remark that, as a consequence of Lemma \ref{lem:monotonicity} and Lemma \ref{lem:scaling}, 
if $X^i_{n,T}$ denotes the $i$-th line of the ensemble $\L^{\underline 0}_{n;T}$ with zero boundary condition on the interval $[-T,T]$, then for all $n\ge 2$ and $T>0$, $X^2_{n,T}$ is stochastically dominated by $\lambda^{-1/3}X^1_{n-1,\lambda^{2/3}T}$.
This follows by first removing the ceiling $X^1_{n,T}$ imposed on $X^2_{n,T}$, which yields a $(n-1)$-line ensemble with overall strength $a\lambda$ on $[-T,T]$,  and then by applying the scaling relation \eqref{eq:BrownScale} with $r=-\ell=\lambda^{2/3}T$. More generally, by removing all top $k$ paths and applying the scaling relation  $k$ times,  one has  that  for any $n>k\ge 1$, 
$X^{k+1}_{n,T}$ is stochastically dominated by $\lambda^{-k/3}X^1_{n-k,\lambda^{2k/3}T}$.

\subsection{Strong Gibbs property}
Following \cite{corwinhammond}, we observe that a $\lambda$-tilted LE has the strong Gibbs property, that is the Brownian-Gibbs property in \eqref{eq:BG-prop} extends to the case where the interval $[\ell,r]$ is replaced by a 
 stopping domain. This will be crucial for a number of coupling arguments to be used in the sequel. 

We say that the random variables $(\tau_\ell,\tau_r)$ form a stopping domain for the first $n$ lines of the line ensemble if   $\{\tau_\ell \le t\}\cap\{\tau_r \ge s\}\in\calB_{n; t , s}^{\sfe}$ for all $t<s$. Define $\calU=\{(\ell,r,\Omega^+_{\infty;\ell,r}),\, \ell<r\}$ and recall the notation from Definition \ref{def:laBG}. 
 The argument of Lemma 2.5  in \cite{corwinhammond} proves the following statement. 
\begin{lem}\label{sabg}
Let $\mu\in\calP^+_\infty$ be a $\lambda$-tilted LE in the sense of Definition \ref{def:laBG}. Then for any $n\in\N$, if $(\tau_\ell,\tau_r)$ is a stopping domain for the first $n$ lines, and
 $F:\cU\mapsto \bbR$ is bounded and measurable, 
then %
\begin{equation}\label{eq:BG-props} 
\bbE\left[ \, F(\tau_\ell,\tau_r,\uX)\tc
\calB^{\sfe}_{n ;\tau_\ell , \tau_r} \right] 
= 
\bbE_{\uY \sim
 \L_{n ;\tau_\ell , \tau_r}^{\uX^{\le n} (\tau_\ell) , \uX^{\le n} (\tau_r )}[X^{n+1}]}
 \left[\,F(\tau_\ell,\tau_r,\uY\,,\uX^{>n}) 
\right], \qquad \mu-\text{a.s.}
\end{equation}
\end{lem}

\subsection{Parabolic barrier representation}
Let us consider a single Brownian bridge and its area-tilted deformation introduced in Definition \ref{def:one line}. It will be convenient to use the following notation.
For any $a>0$ we write $\p^{(a)}(t) = \f{a}{2}\,t^2$, and if $a=2$ we simply write $\p$ for $\p^{(a)}$. For any $h\in\calC_{\ell,r}$, write $\B^{x, y}_{\ell, r}[h]=\B^{x, y}_{\ell, r}(\cdot\,|\,X\succ h) $ for the distribution obtained from $\B^{x, y}_{\ell, r}$ by conditioning on the event 
\[
\{X\succ h\}=\{X(t)> h(t)\,,\;\forall \;t\in(\ell,r)\}.
\]

\begin{lem}[Parabolic barrier representation (\PAR)]
    \label{lem:girsanov}
    For any  $a > 0$, for all $x,y>0$ and $\ell<r$,  
    \[
 Y+\p^{(a)} \sim \L^{a;x, y}_{\ell, r}\;\; \Longleftrightarrow  \;\; Y\sim \B^{x-\p^{(a)}(\ell), y-\p^{(a)}(r)}_{\ell, r}[-\p^{(a)}]
 .
    \] 
\end{lem}
\begin{proof}
By Girsanov's theorem, if $Y\sim \B^{x-\p^{(a)}(\ell), y-\p^{(a)}(r)}_{\ell, r}$ then $X=Y+\p^{(a)}$ has a density with respect to $\B^{x, y}_{\ell, r}$ which is proportional to 
\[
\exp\left(\int_\ell^r h(t)dB(t) 
\right)\,,\quad \text{where }\;\; h(t) =   \f{d}{dt} \p^{(a)} =  a\,t\,. 
\]
The conclusion follows using 
\[
\int_\ell^r h(t)dB(t)  = - \int_\ell^r h'(t)B(t)dt +B(r)h(r) - B(\ell)h(\ell)  = - a\int_\ell^r B(t)dt + {\rm constant}\,,
\]
and the fact that $Y\succ - \p^{(a)}$ iff $X\succ 0$. 
\end{proof}

\subsection{A first tightness estimate}
Suppose $X\sim \L^{x,y}_{T}$ is the area-tilted excursion on $[-T,T]$, with strength $a=2$, and with boundary values
  $x,y\ge 0$ such that 
   \begin{equation}\label{eq:tight1e}  
\max\{x,y\}\leq T^2-\alpha T\,, 
\end{equation}
for some $\alpha >0$. %
Lemma \ref{lem:girsanov} and the following lemma immediately imply the tightness of the height at the origin $X(0)$ as $T\to\infty$. 
\begin{lem}[Tail bounds at the origin] %
    \label{lem:tailbounds}
    Let %
    $Y\sim \B^{y_1, y_2}_{T}[-\p]$ where $y_1, y_2 \leq -\al T$ for some $\al > 0$.
Then there exists a constant $C>0$ depending only on $\al$ such that for all $T\ge 1$,
    \begin{align*}
        \P(Y(0) \geq t) \leq C\,e^{-2\al t}, \qquad t\geq 0.
    \end{align*}
\end{lem}
\begin{proof}
    By monotonicity, it suffices to prove the upper bound for $Y\sim \B^{-\al T}_{T}[-\p]$.
If $X \sim \B^{-\al T}_{T}$, then 
    \begin{align}
        \label{eq:tailorig}
       \P(Y(0) \geq t)= \P(X(0) \geq t\, \big|\, X + \p \gtr 0) \leq \f{\P(X(0) \geq t)}{\P(X + \p \gtr 0)}.
    \end{align}
    By conditioning on $X(0)$ we write %
    \begin{align*}
        \P(X + \p \gtr 0) &= \int_{-\infty}^{\infty} \P(X + \p \gtr 0\, \big|\, X(0) = s) f(s) ds,
    \end{align*}
    where $f$ is the density of $X(0)$.
Define 
\[
p_s := \P({X(t) + t^2 > 0, \forall t \in [0, T]} \, \big|\, X(0) = s).
\]
By  monotonicity, $p_s$ is increasing in $s$, and as shown in \rlem{lem:brownianav2} below, one has $p_s \geq 1/2$ for all $s \geq t_0$ for some constant $t_0$ depending only on $\al$.
Thus
    \begin{align*}
        \P(X + \p \gtr 0) \geq \int_{t_0}^\infty p_s^2 f(s) ds \geq
         \f{1}{4}\,\P(X(0) \geq t_0),
    \end{align*}
    where the $p_s^2$ term appears since conditioning on $X(0)$ makes the left and right sides of $X$ i.i.d.
 Substituting in \eqref{eq:tailorig}, 
    \begin{align*}
        \P(X(0) \geq t\, \big|\, X + \p \gtr 0) \leq 4 \cdot \f{\P(X(0) \geq t)}{\P(X(0) \geq t_0)}.
    \end{align*}
    Recall that $X(0) \sim \Nor{-\al T}{T/2}$. If 
    $\xi \sim \Nor{0}{1}$, and $\phi$ denotes the density of $\xi$, then the well known inequalities 
    \[
    \frac{t}{t^2+1}\, \phi(t) \leq \P(\xi \geq t) \leq \frac1t\,\phi(t)
    \]
imply 
   \begin{align*}
         \P(X(0) \geq t\ \big|\ X + \p \gtr 0)\leq 4\,\f{u(t_0,T)^2+1}{u(t_0,T)u(t,T)}\,
         \exp\left(-\frac12(u(t,T)^2-u(t_0,T)^2)\right)
    \end{align*}
    where $u(s,T):=(\al T + s)\sqrt{2/T}$. Since $u(t,T)^2-u(t_0,T)^2\geq 4\al(t-t_0)$ if $t\geq t_0$, taking $t_0$  large enough depending only on $\al$, then for all $T\geq 1$ one can bound the right hand side above by $8e^{-2\al(t-t_0)}$. Adjusting the value of the constants concludes the proof.
\end{proof}

\subsection{FS diffusion}
The following lemma gathers some basic facts about the Ferrari-Spohn diffusion process introduced in \cite{ferrarispohn2005}. 
Recall Definition \ref{def:one line} of the set of Gibbs measures $\cG$. 
We write $\aFS$ for the Ferrari-Spohn diffusion with strength $a>0$ and write simply $\FS$ when $a=2$. 
 \begin{lem} %
    \label{lem:fsfact}
The measure $\FS\in\calP^+$ satisfies the following. 
    \begin{enumerate}
        \item As $T\to\infty$ the measures $\L^{0}_{T}$ converge weakly to $\FS$; 
        \item $\FS$ is the unique stationary element of $\cG$;
        \item for any $a>0$, one has $X\sim \FS$ iff  $X_a = (a/2)^{-1/3} X ((a/2)^{2/3} \cdot ) \sim \aFS$; 
        \item  if $X \sim \aFS$ then $\P(X(0) >t ) = \Exp{-\left(\tfrac{2\sqrt {2a}}{3}+o(1)\right)t^{3/2}}$ as $t \to\infty$; 
        \item if $X \sim \aFS$ then 
          \begin{align}
            \label{eq:fsmax}
                \P[\max_{s\in[-T, T]} X(t) >t ] \leq  CT a^{2/3}\Exp{-c \sqrt a \,t^{3/2}}
            \end{align}
            for some constants $c,C>0$ independent of $T\geq 1,t>0$ and $a>0$. 
            
    \end{enumerate}
\end{lem} %
\begin{proof}
The weak convergence of $\L^{0}_{T}$ is guaranteed by monotonicity, and the limiting measure must be a stationary element of $\cG$. This follows as a special case of  a general argument which applies to an arbitrary number of lines, see \cite{CIW19}. Moreover, as a consequence of the uniqueness results in \cite{geomarea} it is immediate to deduce that there exists a unique stationary Gibbs measure. 
On the other hand the Markov property and stationarity of the FS diffusion process %
show that $\FS$ is a stationary element of $\cG$; see e.g.\ \cite{ioffeshlosmanvelenik}. 
This proves the first two items in the lemma. 
The scaling is an immediate consequence of Lemma \ref{lem:scaling} and the convergence stated in the first item.  The last two items follow from the scaling in item 3 and well known asymptotic estimates for the Airy function, see e.g.\  Lemma 5.3 from \cite{geomarea}.
\end{proof}

    \subsection{Coming down estimates}
The next lemma quantifies the pull down effect induced by the area tilt. Assuming a  boundary height $H$, it provides useful, albeit %
not entirely optimal, predictions regarding the time it takes for the line to come down %
to a constant height.

\begin{lem}
\label{lem:comingdown}
There exist   absolute constants $C,c>0$ such that for all $H>0$ and $T \ge H^{3/2}$, if $X \sim \L^{H}_{T}$, then 
   \begin{align*}
        \P\left(X(t)\ge  C\,,\;\forall t\in[-T,T]\right) \leq C\Exp{-c\,T}.
    \end{align*}
\end{lem}

\begin{proof}
    By monotonicity, and the assumption $T \ge H^{3/2}$, we can actually assume $H=T^{2/3}$ for the rest of the proof. Moreover, by adjusting the value of the constant $C$ we may assume that $T$ is larger  than some absolute constant $T_0$.  
    Define the event $E = \{X(t)\ge  C\,,\;\forall t\in[-T,T]\}$, where $C$ is a constant to be chosen later.
  By  definition, taking  $Y \sim \B^H_{ T}$,  \begin{align}
        \label{eq:comingdownmain}
        \P(E) = \f{\E[\1_E \1_{Y \succ 0}\Exp{-2\int_{-T}^T Y(t) \,\dt}]}{\E[\1_{Y \succ 0}\Exp{-2\int_{-T}^T Y(t) \,\dt}]}.
    \end{align}
In the numerator, the integral inside the exponential is always at least $2TC$,
so the numerator is at most $\exp(-4CT)$.
To prove a lower bound on the denominator we restrict to the event defined as follows.
Let $J = T - \sqrt{H}$
and consider the set of {good trajectories} %
    \begin{align*}
        G &= \Bra{Y \in\Omega^+_{-T, T} :\; Y(\pm J) \in [1, 2],\;\; \max_{|t|\le J}\,Y(t) \le 3,\;\;\max_{J\le |t|\le T}  Y(t) \leq 2H }.
    \end{align*}
    For any $Y \in G$, we have
    \begin{align*}
        \int_{-T}^T Y(t) \,\dt \leq 2H \cdot \sqrt{H} \cdot 2 + 3 \cdot 2T = 10\,T. %
    \end{align*}
Using the Markov property,
    \begin{align}
        \label{eq:good}
        \P(Y \in G) &\ge 
        P\left(Y(t) \in [0, 2H]\,,\; \forall \,\, J\le |t|\le T\,,\;Y(\pm J) \in [1, 2]\right)\\      
& \qquad \qquad \qquad \times \inf_{y_+,y_-\in [1, 2]}\P\left(Y(t) \in [0, 3]\,,\; |t|\le J\; |\; Y(\pm J) =y_\pm\right). %
    \end{align}
    We lower bound the two factors individually.
A standard estimate for tube probabilities of Brownian motions %
shows that the last term is at least $\Exp{-cJ} \geq \Exp{-cT}$ for some absolute constant $c>0$.
For the first term in \eqref{eq:good}, observe that $Y \sim \B^H_{T}$ satisfies
    \begin{align*}
        (Y(-J), Y(J)) \sim \Nor{\begin{bmatrix}H \\ H\end{bmatrix}}{\begin{bmatrix} \f{H^{1/2}(2T - H^{1/2})}{2T} & \f{H}{2T} \\ \f{H}{2T} & \f{H^{1/2}(2T - H^{1/2})}{2T} \end{bmatrix}}
    \end{align*}
Since $T= H^{3/2}$, the density of this vector on $[1, 2]^{2}$ is at least $ \Exp{-C_1T}$ for a constant $C_1$ for all $T$ large enough.
Now fix $y_+,y_-\in   [1, 2]$ and condition on $ Y(\pm J) =y_\pm$.
Given this conditioning, $Y|_{[-T, -J]}$ and $Y|_{[J, T]}$ are independent.
Furthermore, there is another constant $C_2$ such that uniformly over $y_-\in[1,2]$, %
    \begin{align}\label{eq:to1a}
        \P\left(Z(t) \in [0, 2H]\,,\; \forall t\in[0,\sqrt{H}]\right) \geq \Exp{-C_2 \sqrt H}
    \end{align}
    where $Z \sim \B^{H, y_-}_{0, \sqrt{H}}$.  Although this is not an optimal lower bound (a little thought shows that the true bound is a constant), it will suffice for our purposes.
To prove %
\eqref{eq:to1a}, we may replace the $[0, \sqrt{H}] \times [0, 2H]$ box with a strip of width $3$ around the line connecting $(0, H)$ and $(\sqrt{H}, 1)$. Our choice of the strip has height 2 above this line, and height 1 below this line:
    \begin{align*}
        \P[Z(t) \in {[0, 2H]}\,,\; \forall t\in{[0,\sqrt{H}]}
        ]
&\geq \P[Z(t) \in H - (H - 1)t H^{-1/2} + {[-1,2]}%
] 
    \end{align*}
        if $H$ is sufficiently large.
This is lower bounded by $\Exp{-C_2 H^{1/2}}$
    using again a standard estimate for tube probabilities of Brownian motions. %
    Putting everything together, from \eqref{eq:good} we have obtained
    \begin{align*}
        \P(G) \geq \Exp{-C_3T}
    \end{align*}
    where $C_3$ is some absolute constant and $T$ is large enough.
Therefore from \eqref{eq:comingdownmain} we have $        \P(E) \leq \Exp{-cT}$
    for a constant $c$ as long as $4C > C_3$, and for all $T$ large enough. Adjusting the value of the constants  ends the proof.
\end{proof}

We also need the following statement for general values of the strength parameter $a$.
\begin{cor}%
    \label{lem:comingdown2}
   There exist absolute constants $C,c>0$ such that for any $a\ge 1$, $H>0$, 
and $T \geq H^{3/2}a^{-1/6}$,   if $X \sim \L^{2a;H}_{ T}$ then
      \begin{align}\label{eq:comedowneq}
        \P\left(X(t)\ge  C\,a^{-1/3}\,,\;\forall t\in[-T,T]\right) \leq C\Exp{-c\,a^{2/3}T}.
    \end{align}
     Moreover,  for any $a\ge 1$, $H>0$, 
and $T \ge 1$,   if $X \sim \L^{2a;H}_{ T}$ then
     \[
      \P[\max_{t\in[-T, T]} X(t) \ge 2H]\leq 
 CTa^{2/3}\Exp{-c\,a^{1/2}H^{3/2}}\,.      
 \]
\end{cor}
\begin{proof}
    Given $X$ as in the statement, let $Y = a^{1/3} X(\cdot / a^{2/3})$. Then $Y \sim \L^{Ha^{1/3}}_{Ta^{2/3}}$ by Lemma \ref{lem:scaling}. The lower bound $T \geq H^{{3/2}}a^{{-1/6}}$ ensures that  $Ta^{2/3} \geq (Ha^{1/3})^{3/2}$, so using Lemma \ref{lem:comingdown} applied to $Y$, we know that $Y$ comes down to a constant $C$ with probability $\geq 1 -C \Exp{-cTa^{2/3}}$. Therefore $X$ comes down to $Ca^{-1/3}$ with the same probability. This proves the first bound. 
    For the second one, observe that by monotonicity, the asserted probability is smaller for $Y = H + Z$ where $Z \sim \mu^{2a,{\rm FS}}$. Therefore, \rlem{lem:fsfact} implies 
    \begin{align*}
        \P[\max_{t\in[-T, T]} X(t) \geq 2H] &\leq \P[\max_{t\in[-T, T]} Z(t) \geq H] \leq CT a^{2/3}\Exp{-c \,a^{1/2} \,H^{3/2}}
    \end{align*}
    This establishes the second part.
\end{proof}

    The next lemma uses the results of Corollary \ref{lem:comingdown2} to show that if the endpoints $H$ of $X \sim \L^{H}_{ T}$ are small enough compared to the domain size $2T$, then the process comes down and stays bounded by  height $T^{\del}$, for any small $\delta>0$,  for most of the interval.
    Furthermore, outside this interval the process does not grow too far from its original height $H$.

\begin{lem}%
    \label{lem:doublesidedcomingdown}
    For any $\delta>0$, there exist constants $C,c>0$ such that %
    for all $a \geq 1, H>0, T \geq 3H^{3/2}a^{-1/6}$, if $X \sim \L^{2a;H}_{T}$
then     \begin{align*}
        &\P\left(
        \max_{T -H^{3/2}a^{-1/6}\le |t|\le T}X(t)\le 4H\,,\; \max_{|t|\le T -H^{3/2}a^{-1/6}}X(t)\le T^\delta\right)\\
        & %
       \qquad \qquad 
\geq  1- C\Exp{-c\,a^{1/2}\min(H^{3/2}, T^{3\del/2})}.
    \end{align*}
\end{lem}

We observe that the bound $T^{\del}$ for the middle stretch is not necessarily optimal, and one may instead obtain a bound of $\T{\log T}$ at the price of a weaker control on the probability.
    However, %
    the polynomial bound will suffice for our purposes.
    Furthermore, notice that the probability bound is not useful when $H$ is too small.
However, in later applications %
it will suffice for us to consider $H$ to be a power of $T$ for which the lemma will produce effective bounds.
\begin{proof}
    Let $S = T - H^{3/2}a^{-1/6}$. Let $Y = H + W$ where $W \sim \mu^{2a,{\rm FS}}$. 
    Then $X$ is stochastically dominated by $Y$, and
    $\P(Y(\pm S) \leq 2H) \geq 1 - C\Exp{-cH^{3/2}a^{1/2}}$ by Lemma \ref{lem:fsfact}, part 4. 
     Let $Z \sim \L^{2a;2H}_{S, T}$. %
     Under the event $Y(S) \leq 2H$, $X|_{[S, T]}$ is stochastically bounded by $Z$. To $Z$ we apply Corollary \ref{lem:comingdown2} to obtain
    \begin{align*}
      \P(Z \leq 4H) &\geq 1 - H^{3/2}a^{-1/6}a^{2/3}\Exp{-c H^{3/2} a^{1/2}} \\
                    &= 1 - H^{3/2}a^{1/2} %
                    \Exp{-cH^{3/2}a^{1/2}} 
                    \geq 1 - C\Exp{-c'H^{3/2} a^{1/2}},
    \end{align*}
    for some absolute constants $C,c'>0$. 
Moreover, by \eqref{eq:comedowneq}, with probability $1- C\Exp{-c a^{2/3} H^{3/2}}$ there exists a stopping domain $(\tau_\ell,\tau_r)$ with $\tau_\ell<-S$ and $\tau_r>S$ such that $X(\tau_\ell),X(\tau_r)$ are both smaller than $Ca^{-1/3}\le C$, for some other absolute constant $C$. 
Given this stopping domain, according to Lemma \ref{sabg} we may resample $X|_{[\tau_\l, \tau_r]}$, which is stochastically dominated by $C+W$ with $W\sim \mu^{2a,{\rm FS}}$. Using Lemma \ref{lem:fsfact}, part 5, for $W$, applied to the interval $[-T,T],$ we then have %
    \begin{align*}
        \P[\max_{t\in[\tau_\l, \tau_r]} X(t) \leq C + T^\del] \geq 1 - CTa^{2/3}\Exp{-c{T^{3\delta/2}a^{1/2}}} 
        \geq 1 - C'\Exp{-c'{T^{3\delta/2}a^{1/2}}}\,
    \end{align*}
    for some new constants $C',c'>0$.  
    This finishes the proof.
\end{proof}

Finally, in the following section we record a couple of lemmas which while easy to prove, will be invoked several times throughout the rest of the paper.

\subsection{Tangency geometry and Brownian bridge avoidance probabilities}
The first is a simple calculation about the $x$-tangency locations to the parabola $-\p^{(2a)}: t\mapsto -at^2$ from the point $(T, -\al T^{2})\in\R^2$ for arbitrary parameters $a>\al>0$. 
\begin{lem}[Tangency locations]
    \label{lem:tangency}
    For all $0 < \al < a$, consider the tangent line $\ell$ from the point with coordinates $(T, -\al T^{2})$ to the parabola $-\p^{(2a)}:  t \mapsto -at^{2}$. The $x$-coordinates of the two points where $\ell$  touches $-\p^{(2a)}$ are given by $T\Box{1 \pm \sqrt{1 - \al/a}}$.
\end{lem}
\begin{proof}
    The result follows from simple calculus.
    At a generic point $(x, -\p^{(2a)}(x))$, the slope of the tangent is $-2ax$, so that the tangent line is $t \mapsto -2axt + ax^2$. This line passes through $(T, -\al T^{2})$, so that
    \begin{align*}
      -2axT + ax^{2} = -\al T^{2} \iff -2xT + x^2 = -(\al/a) T^{2}.
    \end{align*}
    Adding $T^{2}$ to both sides,
    \begin{align*}
      (T - x)^{2} = (1 - \al/a)T^{2} \implies x = \Box{1 \pm \sqrt{1 - \al/a}} T. 
    \end{align*}
\end{proof}
Unless otherwise specified, when we will refer to the tangency location from a point to a parabola we will mean the internal one, namely the one with the minus sign in Lemma \ref{lem:tangency}.
The next lemma delivers an estimate for the probability of a Brownian bridge avoiding an arbitrary function which will also be used when the function is specialized to be a parabola.
{
\begin{lem}
    \label{lem:brownianav2}
    Let $V\sim \B^0_{0,T}$ be a standard Brownian bridge on $[0, T]$,  and let $f$ be %
    the function $t \mapsto T^{\del}\sqrt t + T^{-10}$, for a constant $\delta\in(0,1/2)$. Then, for some absolute constants $C,c>0$, for all $T>0$, 
    $$\P[V \text{ \rm hits } f] \leq C\exp\big(-c\,T^{2\del}\big).$$
\end{lem}
While the lemma holds in vast generality, the precise form of the function is chosen for convenience of application in our arguments.
\begin{proof}
    Observe that $f(t) \geq T^{-10}$ when $t \leq T^{-21}$ and $f(t) \geq T^{\del}t^{0.5}$ when $t > T^{-21}$. By an union bound
    \begin{align*}
      \P[V \text{ hits } f] \leq \P[V|_{[0, T^{-21}]} \text{ hits } t \mapsto T^{-10}] + \P[V|_{[T^{-21}, T]} \text{ hits } t \mapsto T^{\del}t^{0.5}].
    \end{align*}
    We bound the terms individually. By Brownian scaling, by a factor of $T^{21}$ and $\sqrt{T^{21}}$ in time and space respectively, the first term satisfies, for some  absolute constants $C,c>0$,
    \begin{align*}
      \P[V|_{[0, T^{{-21}}]} \text{ hits } t \mapsto T^{-10}] = \P[V_{[0, 1]} \text{ hits } t \mapsto T^{0.5}] \leq C\exp\big(-c\,T\big)%
      ,
    \end{align*}
    the last inequality following from the fact that $\max_{t \in [0, 1]}V(t)$ is sub-Gaussian.
    For the second half, we consider partition $[T^{-21}, T]$ into scales given by the intervals $I_{k} = [T^{-21}2^{k - 1}, T^{-21}2^{k}]$ for $k = 1, \ldots, \T{\log T}$. For each scale $k$, we have
    \begin{align*}
      \P[V|_{I_{k}} \text{ hits } t \mapsto T^{\del}t^{0.5}] & \leq \P[V|_{[0, T^{-21}2^{k}]} \text { hits } T^{\del - 10.5} 2^{(k-1)/2}] .
      \end{align*}
      Using Brownian scaling, by $T^{21}2^{-k}$ in time and $T^{10.5}2^{-k/2}$ in space, the above is equal to
    \begin{align*}
                                                             \P[V|_{[0, 1]} \text{ hits } t \mapsto T^{\del} 2^{-1/2}] 
                                                               \leq C\exp\big(-c\,T^{2\del}\big), %
    \end{align*}
    where the last inequality again follows from sub-Gaussian tails on $\max_{{t \in [0, 1]}} V(t)$. A union bound then finishes the proof.
\end{proof}
Using the same dyadic decomposition, the following parabolic avoidance estimate follows which we state without proof. This can be seen quickly, for instance, by observing that the probability of Brownian bridge hitting the parabola on the interval $[2^{k}\sqrt{h},2^{k+1}\sqrt{h}]$ is $\exp\big(-\T{2^{3k}h^{3/2}}\big).$ 
\begin{cor}\label{paraavoidancecor} For $h>1,$ consider the parabola $ t \mapsto h + t^2$. For $T\ge 1,$ the probability that the standard Brownian bridge $ \B^0_{0, T}$ avoids the parabola is at least $1 - {\Exp{-\T{h^{3/2}}}}.$ 
\end{cor}

\section{Asymptotic behavior}
\label{sec:asymp}
\subsection{Overall parabolic growth with linear correction}
As alluded to in the introductory section, the first  key observation in the proof of Theorem \ref{thmoneline} is that for any Gibbs measure $\mu\in\cG$, the behavior at infinity is parabolic with a linear correction. Recall Definition \ref{def:abg} of the set $\cG$ of one-line area-tilted Gibbs measures and assume that $a=2$.
\begin{lem} %
    \label{lem:limexists}
   Suppose $\mu\in\cG$ and %
   $X \sim \mu$.
Then the limits  $L=L(\mu), R=R(\mu)$ defined in \eqref{linearlim} exist $\mu$-almost surely as extended real random variables.
\end{lem}
\begin{proof}
The %
proof %
involves showing that almost surely, 
\[
\overline R:=\limsup_{t\to \infty}  
\,(X(t)-t^2)/t=\liminf_{t\to \infty} \,(X(t)-t^2)/t=:\underline R.
\] 
The same argument applies to the case  $t\to -\infty$ with the left limits $ \overline L, \underline L$.
We note that these random variables could take the value $\pm\infty$ at this point. %

We are going to prove that for any $x\in\R$ 
\begin{align}\label{eq:implic2} 
  \P(\underline R < x < \overline R)=0,
  \end{align} 
which implies $\underline R\ge \overline R$ a.s. Since $\underline R \le \overline R$, this proves  
 that $\underline R =\overline R$ a.s.

To prove \eqref{eq:implic2} it will be convenient to rewrite $x$ as $x=-2b$, $b\in\R$. 
  Fix  $b \in \R, \del > 0,$ and define %
   the events
    \begin{align*}
        \Once_k &= \{ \exists\ t \in [k, k + 1) : (X(t) - t^2)/t \geq -2b \}, \quad k \in \N, \\
        \All_{\ell, r} &= \{\forall\ t \in [\ell, r] : (X(t) - t^2)/t \geq -2(b + \del) \}, \quad \ell < r \in \R.
    \end{align*}
   The definitions are natural in the sense that $ \Once_k$ for large values of $k$ serve as witness for the event  $\overline R > -2b ,$ %
while $\All_{\ell, r}$ for  large values of $\ell $ and $r$ and small values of $\delta$ will ensure that $\underline R \ge -2b.$

We start by proving that  $\Once_k \implies \All_{\ell, r}$ with high probability if $k$ is large enough, for suitable choices of $\ell,r$.
 To carry this out formally we need %
 to define %
 the ``last witness time of $\Once_k$'' as the random time $\tau = \tau_k$ given by
    \begin{align*}
        \tau = \sup \{t \in [k, k + 1) : (X(t) - t^2)/t \geq -2b \}, \quad \text{on } \Once_k,
    \end{align*}
  and by  {$\tau = k$ on $\Once_k^\c$}. %
Therefore, for each sufficiently large $k$, $[b, \tau]$ is  a stopping domain.
Consider two times $\ell < r$ such that $\ell > 10b$ and choose $k > 2r^2$, which ensures in particular that $[\ell,r]\subset [b,\tau]$. 
We first %
show that $\P(\All_{\ell, r} \,|\, \Once_k)$ is high.
Once we prove this claim, the remainder of the proof reduces to taking the appropriate limits.

We can apply Lemma \ref{sabg} for one individual line, 
with the stopping domain $[b, \tau]$. 
Conditionally on the external $\sig$-algebra $\cB_{b, \tau}$, on  $\Once_{k}$, by monotonicity the probability of 
$\All_{\ell, r}$ does not increase if we lower the boundary %
condition to $X(b) = 0$ and $X(\tau) = \tau^2 - 2b\tau$. Thus we may now resample 
the area-tilted line on $[b,\tau]$ with these boundary conditions. %
Switching to the \PAR (see Lemma \ref{lem:girsanov}), %
    \[
    \P(\All_{\ell, r} \,|\, \Once_k)\ge \inf_{s\in[k,k+1)}\P(Y^s(t) \geq -2(b + \del)t,\ \forall t \in [\ell, r] ) ,
    \] 
    where  $Y^s \sim \B^{-b^2, -2bs}_{b, s}[- \p].$ Using monotonicity we may remove the floor $-\p$ and push the left endpoint down to $-2b^2$ (this makes the algebra cleaner), 
to obtain 
    \begin{align*}
        \P(Y^s(t) \geq -2(b + \del)t,\ \forall t \in [\ell, r] ) 
        & \geq \P(Z^s(t) \geq -2(b + \del)t,\ \forall t \in [\ell, r]),
        \end{align*}
where $Z^s \sim \B^{-2b^2, -2bs}_{b, s}$.
The line joining $(b, -2b^2)$ and $(s, -2bs)$ is $t \mapsto -2bt$.
Affine invariance for Brownian bridge implies that $Z^s(t) = -2bt + V^s(t)$ where $V^s \sim \B^{0}_{b, s}$.
We may then rewrite the above probability as $\P(V^s(t) \geq -2\del t,\ \forall t \in [\ell, r]) $. Moreover, using standard Brownian fluctuation estimates and the fact that by construction  $10b < \ell < r$ and $s\ge k > 2r^2$, one has 
\[  
\P(V^s(t) \geq -2\del t,\ \forall t \in [\ell, r]) \geq %
1 - C\Exp{-c\,{\ell \del^2}},
\]
where $C,c>0$ are absolute constants.
This shows that:
    \begin{align}\label{eq:oka}
        \P(\All_{\ell, r} \, |\,  \Once_k) \geq 1 - C\Exp{-c\,{\ell \del^2}}\,.
    \end{align}
    Now we replace the stopping domain $[b, \tau_{k}]$ by the stopping domain $\left[b, \tau_{k, m}\right]$ where $k<m$ and
    \begin{align*}
        \tau_{k, m} = \sup \{t \in [k, m + 1) : (X(t) - t^2)/t \geq -2b \}, \quad \text{on } \Once_k \cup \ldots \cup \Once_m,
    \end{align*}
   with $ \tau_{k, m}=k$ on the complementary event. The very same argument as above implies
    \begin{align}
        \label{eq:ktaom}
        \P(\All_{\ell, r} \,|\, \Once_k \cup \Once_{k + 1} \cup \ldots \cup \Once_m) \geq 1 - 
        C\Exp{-c\,{\ell \del^2}}, %
    \end{align}
    for any $m > k$. Abbreviating ``infinitely often'' as i.o., taking limits as $m \to \infty$ and then $k\to\infty$ in \eqref{eq:ktaom} 
we obtain, for all $\ell>10b$ and $r>\ell$,
    \begin{align*}
        \P[\All_{\ell, r} \,|\, \Once_j \text{ i.o.}] \geq 1 - C\Exp{-c\,{\ell \del^2}}.
    \end{align*}
    Sending $r \to \infty$,
    \begin{align*}
        \P[\forall t \in [\ell, \infty) : \f{X(t) - t^2}{t} \geq -2(b + \del)\, \Big|\, \Once_j \text{ i.o.}] \geq 1 -C\Exp{-c\,{\ell \del^2}}.
    \end{align*}
    Passing to the limit $\ell \to \infty$,
    \begin{align*}
        \P[\underline R
        \geq - 2(b + \del)\, \big|\, \Once_j \text{ i.o}] = 1.
    \end{align*}
 Letting $\del \to 0$ we get
    \begin{align}\label{eq:rxb}
        \P[\underline R %
        \ge - 2b\, \big|\, \Once_j \text{ i.o}] = 1.
    \end{align}
    We can now prove \eqref{eq:implic2}. Set $x=-2b$ and 
    observe that  $\{\overline R > x\}\subseteq\{\Once_j \text{ i.o.}\}$. Therefore, \eqref{eq:rxb} implies 
   \begin{align*}
    \P[\underline R < x < \overline R] &\le  \P\left(\underline R < x,  \,\Once_j \text{ i.o.}\right) = 0. 
        \end{align*}
\end{proof} 

\subsection{Extremal measures}
Recall the  well known fact that extremal Gibbs measures must be tail trivial, i.e., if $\mu\in\cG_{\rm ext}$ then  $\mu(A)\in\{0,1\}$ for all $A \in \cB_\infty$, where  $\cB_\infty = \cap_{k \geq 0} \cB_{-k, k}$ denotes the tail-$\sig$-algebra (see e.g.\  \cite[Theorem 7.7]{georgii2011gibbs}). %
Since the limits $L$ and $R$ from Lemma \ref{lem:limexists} are both tail-measurable random variables, we obtain the following fact. 
\begin{lem}
    \label{lem:deterministic}
    If $\mu\in\cG_{\rm ext}$, %
    then $L = \lim_{t \to -\infty} (X(t) - t^2)/|t|$ and $R = \lim_{t \to \infty}  (X(t) - t^2)/t$ exist a.s.\ and are deterministic.
\end{lem}

\renewcommand{\A}{\mathsf{OutLB}}

We now turn to a sequence of statements that will allow us to conclude that the constants $L,R$ from Lemma \ref{lem:deterministic} must satisfy $(L,R)\in\calT$, where $\calT$ is the set  defined in \eqref{lineartilts}.
\begin{lem}[$L<0$ or $R<0$] 
    \label{lem:lr0}
    Let $\mu\in\cG_{\rm ext}$ and let $L,R$ be the constants from Lemma \ref{lem:deterministic}. %
Then it is not possible to have both $L\ge 0$ and $R\ge 0$. %
\end{lem}
\begin{proof}
    The proof is by contradiction, and hence we start with the assumption that $L, R \geq 0$.
Fix $s>0$.
We will first show that for $\mu\in\cG_{\rm ext}$ and $X\sim \mu$, under the non-negativity hypothesis on $L,R,$ the  probability of the event $X(\pm s) \geq s^2 - 10 {\sqrt s}$ is positive uniformly  in $s$.
This, for growing $s$ will then be used to contradict tightness of the marginal distribution at the origin, thereby finishing the proof. 
   
    To prove the claim about $X(\pm s)$, let $\eps = s^{-1/2}$, %
     and define %
     the events
    \begin{align*}
        \A_m = \Bra{\forall\ t \in (-m, m)^\c : X(t) - t^2\geq - \eps\,|t|%
        }, \qquad m \in \N,
    \end{align*}
    with $\A$ being an acronym for ``Outside lower bound''.
    These are increasing events, and by definition, the non-negativity of $L, R$ implies that almost surely, $ \A_m $ holds eventually.
Thus there is a finite $T_0 =T_0(s)\in \N$ such that for all $T\ge T_0$,
\begin{align}
  \label{eq:eventuallywhp}
  \P(\A_T) \geq \tfrac{1}{2}.
\end{align}
Choose a $T \geq \min\{2s,T_0\}$. The reason for the choice will be clear at the end of the inequalities proved below. 
Using the Brownian-Gibbs property to resample in $[-s, s] \subseteq [-T, T]$, by monotonicity one has
    \begin{align*}
        \P[X(\pm s) \geq s^2 - 10{\sqrt s}\ \big|\ \A_T] &\geq \P[X(\pm s) \geq s^2 - 10 {\sqrt s}\ \big|\ X(\pm T) = T^2 - \eps T].
 \end{align*}
Switching to {\PAR}, letting $Y \sim \B^{-\eps T}_{ T}$,
the RHS in the above display is
\[
 \P[Y(\pm s) \geq -10{\sqrt s}\ \big|\ Y + \p \gtr 0].
\]
Due to monotonicity, we can replace the floor conditioning above simply by $Y(0) = 0$ without raising the probability.
So the above is larger or equal to 
\[
\P[Y(\pm s) \geq -10 {\sqrt s}\ \big|\ Y(0) = 0].
\]
Given $Y(0) = 0$, $\{Y(t) : t \leq 0\}$ and $\{Y(t) : t \geq 0\}$ are i.i.d. (up to reflection), so the above display is
equal to
\[
\P[Y(s) \geq -10{\sqrt s}\ \big|\ Y(0) = 0]^2.
\]
Now observe that given $Y(0) = 0,$ and $Y(T) = -\eps T$, we can write $Y(t) = -\eps t + Z(t)$ where $Z \sim \B^0_{0, T}$.
So the above probability is equal to
\[
\P[Z(s) \geq \eps s - 10 {\sqrt s}]^2 =
                                                       \P[Z(s) \geq {\sqrt s} - 10 {\sqrt s} = -9 {\sqrt s}]^2 \geq \t > 0\,,
                                                       \]
   where $\t$ is an absolute constant, %
since we have taken $T \geq 2s$. %
Here we use the fact that $Z(s) \sim 
\Nor{0}{\f{s \cdot T}{T - s}}$ and that $T/(T - s) \in [1, 2]$. 
The above chain of inequalities then yields the claimed, uniformly in $s$, probability bound
    \begin{align}\label{eq:pa1o}
        \P[X(\pm s) \geq s^2 - 10 {\sqrt s}] \geq \P[X(\pm s) \geq s^2 - 10 {\sqrt s}\ \big|\ \A_T] \P[\A_T] \geq \t/2.
    \end{align}
    To finish the proof we turn to obtaining the contradiction indicated above. 
    For  all $s>0$, the Brownian-Gibbs property at $[-s,s]$ and monotonicity imply 
    \begin{align*}
        \P[X(0) \geq {\sqrt s}] &\geq \P[X(\pm s) \geq s^2 - 10{\sqrt s}]\P[X(0) \geq {\sqrt s}\, \big|\, X(\pm s) = s^2 - 10{\sqrt s}] .
    \end{align*}
    The first factor is at least $\t/2$ by \eqref{eq:pa1o}. 
Switching to  {\PAR}, the second factor 
is 
equal to %
    \begin{align*}
        \P[Y(0) \geq {\sqrt s}\ \big|\ Y + \p \gtr 0] \geq \P[Y(0) \geq {\sqrt s}],
    \end{align*}
where  $Y \sim \B^{-10{\sqrt s}}_{s}$ and we use monotonicity to remove the floor $-\p$.
    Since $Y(0) \sim \Nor{-10{\sqrt s}}{s/2}$, it follows that $\P[Y(0) \geq {\sqrt s}]$ is uniformly positive in $s>0$.  
    In conclusion, $\P(X(0) \geq {\sqrt s})$ is uniformly positive in $s>0$, and thus $X(0)$ cannot be tight, which is the desired contradiction. 
\end{proof}

\smallskip

\begin{lem}[$L<\infty$ and $R<\infty$] 
    \label{lem:lrinf}
    If $\mu\in\cG_{\rm ext}$ then  the constants  from Lemma \ref{lem:deterministic} satisfy $L<\infty$ and $R<\infty$.
\end{lem}
\begin{proof}
By symmetry we may restrict to proving that $R<\infty$.     The proof is again by contradiction, that is we assume $R=\infty$ and show that this contradicts tightness of the marginal distribution at the origin. %
Fix   $M>0$ and define the events
    \begin{align*}
        \A^+_m = \Bra{\forall\ t\geq m: X(t) - t^2 \geq M\,t}, \qquad m \in \N,
    \end{align*}
    These are increasing events, and by definition, $R=+\infty$ implies that almost surely, $ \A_m $ holds eventually.
Thus there is a finite $T_0 =T_0(M)\in \N$ such that for all $T\ge T_0$, $  \P(\A^+_T) \geq 1/2$.
   The Brownian-Gibbs property at $[-1,T]$ and monotonicity imply 
    \begin{align*}
        \P(X(0) \geq M/2) &\geq \P(X(T) \geq T^2 + MT)\,\P[X(0) \geq M/2\, \big|\, X(-1) = 0,\,X(T)=T^2 + MT] .
    \end{align*}
    The first factor is at least $\P(\A^+_T) \geq 1/2$. 
Switching to  {\PAR}, the second factor 
is 
equal to %
    \begin{align*}
        \P[Y(0) \geq M/2\, \big|\, Y + \p \gtr 0] \geq \P[Y(0) \geq M/2],
    \end{align*}
where  $Y \sim \B^{-1,MT}_{-1, T}$ and we use monotonicity to remove the floor. Now, $Y(0)$ is a Gaussian random variable with mean $-1+\frac{MT+1}{T+1}\geq M/2$ if $M$ and $T$ are large enough. Therefore $ \P[Y(0) \geq M/2]\geq 1/2$. In conclusion,
$\P(X(0) \geq M/2)\geq 1/4$. Since $M$ is arbitrarily large this ends the proof. 
\end{proof}

\smallskip
\begin{lem}
    \label{lem:lpr}
    Let $\mu\in\cG_{\rm ext}$ and let $L,R$ be the constants from Lemma \ref{lem:deterministic}.
Then $(L,R)\in\calT$, with $\calT=\{(L,R)\in[-\infty,\infty)^2: \; L + R <0\}$ as in \eqref{lineartilts}.
\end{lem}
\begin{proof}%
By Lemma \ref{lem:lrinf}, we know  that $L$ and $R$ are both less than $+\infty$, and all we have to show is that $L+R<0$, which includes the case where one of the two, or both, are $-\infty$. It suffices to show that %
assuming $L+R\ge 0$ yields a contradiction.
For any  $h \in \R$, observe that $\theta_h\mu\in \cG_{\rm ext}$ and that $Y\sim \theta_h\mu$ satisfies $Y(t)-t^2=X(t - h) - (t-h)^2 -2ht +h^2$. Thus, if  $L(\theta_h\mu),R(\theta_h\mu)$ are the limits \eqref{linearlimla} associated to the measure $\theta_h\mu$,
it follows that the shifted measure $\theta_h\mu$ satisfies %
$L(\theta_h\mu) = L + 2h$, and $R(\theta_h\mu) = R - 2h$.
Choosing $h = R/2$,  
    \begin{align*}
        L(\theta_h\mu) = L + R, \qquad 
        R(\theta_h\mu) = 0.
    \end{align*}
Thus, if $L+R\ge 0$, we have obtained a measure $ \theta_h\mu\in \cG_{\rm ext}$  whose left and right limits are both nonnegative. By \rlem{lem:lr0} this yields a contradiction.
\end{proof}

With the above preparation we are now in a position to dive into the proof of the remaining parts of Theorem \ref{thmoneline}.
\section{One-line Gibbs measures}
\label{sec:proofofth1}
In this section we prove Theorem \ref{thmoneline}. The proof is split into several parts and for the reader's aid we provide a brief road map. Note that we know by Lemma \ref{lem:fsfact} that $\FS$ is the unique stationary Gibbs measure. Moreover, we have already proven in \rlem{lem:deterministic} and \rlem{lem:lpr} that any extremal measure has the deterministic left and right limiting linear slopes $(L,R)\in\calT$. Thus it remains to show that for any $(L,R)\in\calT$ there exists a unique $\nu_{L,R}\in\cG$ with the given limits, and uniqueness will also imply extremality, that is $\nu_{L,R}\in\cG_{\rm ext}$. This will be addressed as follows.
    \begin{itemize}
        \item Proposition \ref{ppn:fscouple} covers the case $L = R = -\infty$, which corresponds to  $\FS$.
                
        \item Lemma \ref{lem:existence} and \rlem{lem:uniqueness} address the case of $(L, R)\in\R^2$, $L+R<0$.
        
        \item Finally, Lemma \ref{lem:existence2} and Lemma \ref{lem:uniqueness2} deal with the cases %
        $L = -\infty$ and $R\in\R$ or $L\in\R$ and $R = -\infty$.
    \end{itemize}
    We refer to Section \ref{sec:completed}
for a summary of our conclusions.

\subsection{ Ferrari-Spohn diffusion\;$\Longleftrightarrow$\; $L = R = -\infty$ } \label{sec:fsuni}
First  observe that if $\mu=\FS$ then the limits \eqref{linearlimla} must be $L = R = -\infty$ a.s. Indeed, this follows easily from the asymptotic behavior highlighted in Lemma \ref{lem:fsfact} and a Borel-Cantelli argument. 
The next proposition establishes the converse, namely that any Gibbs measure with these limits must be $\FS$. 
 \begin{ppn} %
    \label{ppn:fscouple}
    Let $\mu\in\cG$ and suppose that the left and right limits \eqref{linearlimla} satisfy $L = R = -\infty$ a.s.
Then $\mu=\FS$. %
\end{ppn}
\begin{proof}
  The proof proceeds by coupling $\mu$ to $\FS$.
    A few ingredients are involved. The first two steps involve showing that starting with $X \sim \mu$ as in the statement, there are large times at which the height of $X$ becomes $O(1)$ and hence comparable to $\FS$ which will then aid in coupling the two samples. This is accomplished by first showing that there are arbitrary large times $M$ such that $X(\pm M)=O(\sqrt{M}),$ which is  then upgraded with the aid of Lemma \ref{lem:comingdown}, to an $O(1)$ height at times $\pm M/2$.
Finally, given such a bounded boundary data, we show that with high probability $X$ can be coupled to the diffusion $Y\sim \FS$ inside a smaller domain, say $[-M/3,M/3]$.
Since $M$ can be arbitrarily large, this will be sufficient.
The above three steps are carried out sequentially below.
\\
    \def\Ce{C_{\eps}}
    \def\A{\mathsf{OutUB}}
    \def\GE{\mathsf{Good}}
    
    \newcounter{claim}

    \stepcounter{claim}
    \noindent $(\theclaim)$ \textbf{Coming down to $O(\sqrt{M})$: }
    Suppose $X\sim\mu\in\cG_{\rm ext}$ with  $L = R = -\infty$. We first show that for all $\eps>0$, there exist  constants $\Ce = O(\sqrt{\log(1/\eps)})$ and $M_0=M_0(\eps)>0$ such that  $X(\pm M) \leq \Ce {\sqrt M}$ with probability at least $1 - \eps$ for all $M\ge M_0$.

Define the event
    \begin{align*}
        \A_s &= \{X(t) \leq t^2 - 2M|t|, \quad \forall \,t \in (-s, s)^\c\}\,, \qquad s \in \N.
    \end{align*}
Reasoning as in \eqref{eq:eventuallywhp},  there exists $T = T(\eps)$ such that $\P(\A_T)\geq 1-\eps/4$. 
Next, choose a constant $M_0=M_0(\eps)$ such  that $\P(X(0) \leq M_0^2) \geq 1 - \eps/4$. 
Note that this simply uses the tightness of the one-point distribution of $X(0)$.
Therefore,
    \begin{align*}
        \P(\GE) &\geq 1 - \eps/2, \quad \text{where } \;
        \GE :=\A_T \cap \{X(0) \leq M_0^2\}
    \end{align*}
    We now condition on the event $\GE$.
Monotonicity implies that for all $M\ge M_0$, 
    \begin{align*}
&        \P[X(\pm M) \leq \Ce {\sqrt M}] \geq \P(\GE)\;\P[X(\pm M) \leq \Ce {\sqrt M}\ \big|\ \GE] \\
                                             &\qquad \geq (1 - \eps/2)\;  \P[X(\pm M) \leq \Ce {\sqrt M}\ \big|\ X(\pm T) = T^2 - 2MT + M^2, X(0) = M^2]  \\
                                             &\qquad \geq (1 - \eps/2) \; \P[X(M) \leq \Ce {\sqrt M}\ \big|\ X(0) = M^2, X(T) = T^2 - 2MT + M^2]^2
\end{align*}
Here we use the Gibbs property of $\mu$ and the fact that given $X(0), X(\pm T)$, with $X(T) = X(-T)$,  the left and right halves of $X$ are i.i.d.. 
Switching to \PAR,  the last display is equal to
  \begin{align*}
                                      (1 - \eps/2) \cdot \P[Y(M) \leq -M^2 + \Ce {\sqrt M}\ \big|\ Y + \p \gtr 0]^2
\end{align*}
where $Y\sim \B^{M^2,-2MT+M^2}_{0, T}$.
The line $t \mapsto M^2 - 2Mt$ always lies above the parabola $-\p(t)= -t^2$, so by monotonicity, we may replace the parabolic floor with this linear floor and reduce to estimating %
 \begin{align*}
                                     (1 - \eps/2) \cdot \P[Y(M) \leq -M^2 + \Ce {\sqrt M}\ \big|\ Y + 2Mt - M^2 \gtr 0]^2.
\end{align*}
However, given $Y + 2Mt - M^2 \gtr 0$, the random path $V=Y + 2Mt - M^2$ has the distribution of a Brownian excursion between $0$ and $T$ with both endpoints at 0.
Since $2Mt - M^2 = M^2$ when $t = M$, the above is lower bounded by
\begin{align*}
 (1 - \eps/2) \cdot \P[V(M) \leq \Ce {\sqrt M}]^2 \geq 1 - \eps,
    \end{align*}
where we choose $\Ce=O(\sqrt{\log(1/\eps)})$ such that $\P(V(M) \leq \Ce {\sqrt M})^2 \geq 1 - \eps/2$.
That this choice can be made follows from the standard fact that, for the Brownian excursion $V,$ the random variable  $V(M)$ is sub-Gaussian with variance  $O(M)$  as long as $T \geq 2M$.
In conclusion, we have obtained  
\begin{align}\label{eq:codo}
\P[X(\pm M) \leq C_\eps {\sqrt M}]\ge 1 - \eps.
\end{align}

    \def\CD{\mathsf{ComesDown}} 
    \stepcounter{claim}
\noindent    $(\theclaim)$ \textbf{Coming down to a constant:}
   Consider the event 
        \begin{align}\label{eq:claim2}
            \CD =\bigg\{\exists\ t_\l \in [-M, -M/2], t_r \in [M/2, M] : X(t_\l), X(t_r) \leq C\bigg\}.        \end{align}
          We wish to show that, for a suitable constant $C>0$, one has  $\P  (\CD) \geq 1 - 10\eps$ for all $M$ large enough (depending on $C$ and $\eps$).
    By the previously established estimate \eqref{eq:codo}, it suffices to show that
    \begin{align*}
        \P&\Rnd{{\CD\, \big|\, X(\pm M) \leq \Ce {\sqrt M}}} \geq 1 - 9\eps.
    \end{align*}

{    By monotonicity, we may condition on $X(\pm M) = \Ce {\sqrt M}$ instead. Under this conditioning, 
with probability at least $1 - \eps$, one has $X(\pm M/2) \leq \Ce' {\sqrt M}$ for some new constant  $\Ce'$ which still  satisfies the  $O(\sqrt{\log(1/\eps)})$} bound. To see this, observe that the value at $X(\pm M/2 )$ is stochastically dominated by the same value when $X$ has zero area tilt, and in that case the bound follows from standard  Brownian fluctuations. Adjusting the value of the constant $\Ce$ and using monotonicity, it is then sufficient to prove  
    \begin{align*}
        \P[\CD\,\big|\, X(\pm M)= X(\pm M/2) = \Ce {\sqrt M} ] \geq 1 - 8\eps.
    \end{align*}
    However, given this conditioning $X|_{[-M, -M/2]}$ and $X|_{[M/2, M]}$ are i.i.d., so it is sufficient to show
    \begin{align*}
        \P\Rnd{\exists\ t_r \in [M/2, M] : X(t_r) \leq C \,\big|\, X(M)=X(M/2) = \Ce {\sqrt M}} \geq 1 - 4\eps. 
    \end{align*}
    Now we may apply \rlem{lem:comingdown} as long as e.g.\ $M/4\geq \Ce^{3/2} M^{3/4}$, i.e, when $M \geq M_0(\eps)$ for some constant $M_0(\eps)$.
This shows that 
    \begin{align*}
        \P\left(X(t)\ge  C\,,\;\forall t\in[M/2,M]\,\big|\, X(M)=X(M/2) = \Ce {\sqrt M}
\right) \leq C\,\Exp{-c\,M}, %
    \end{align*}
    if $M$ is sufficiently large.
This establishes the second claim, namely that  $\P  (\CD) \geq 1 - 10\eps$ for all $M$ large enough.\\

    \stepcounter{claim}
    \noindent    $(\theclaim)$ \textbf{Coupling to finite FS.}
 The idea here is to start with a coupling where we independently sample $X\sim\mu$ and $Y\sim\calL^0_T$ for some large $T$. The latter converges to $\FS$ for $T\to\infty$ from  \rlem{lem:fsfact}, and we may refer to it as the %
 ``finite FS''. We show that there is a stopping domain $[\sig_{\l}, \sig_{r}]$ which contains a fixed interval $[-L,L]$ around $0$, such that $X$ and $Y$ agree at times $\sig_{\l}$ and $ \sig_{r}$ with high probability. On that event, we can resample $X, Y$ inside $[\sig_{\l}, \sig_{r}]$ so that they agree everywhere on $[-L,L]$, while on the complementary event we do not resample. Thanks to the strong Gibbs property this is a new coupling of the two measures, %
 showing that the restrictions of $\mu$ and the finite FS to $[-L,L]$ are close in total variation. %
 We can then conclude that $\mu$ and $\FS$ agree on any  fixed time interval by taking suitable limits. 

We define a first stopping domain $[\tau_\ell,\tau_r]$ as follows. 
  We set  \[
\tau_\l= \inf\Bra{t \in [-M,-M/2] : X(t) \leq C}\,  ,\qquad \tau_r = \sup\Bra{t \in [M/2, M] : X(t) \leq C}\,,
  \]
   where $C$ is as in \eqref{eq:claim2} and set $\tau_r = M/2 $ if there is no $t\in [M/2, M]$ with $X(t) \leq C$, and $\tau_\ell = -M/2$ if there is no $t\in [-M,-M/2]$ with $X(t) \leq C$.  Note that both conditions are satisfied on the event  \eqref{eq:claim2} 
and that 
$[\tau_\l, \tau_r]$ is a stopping domain containing $[-M/2, M /2]$.
 Let $Y \sim \L^{0}_{M/2}$ be the finite FS measure on $[-M/2, M/2]$ and sample $X,Y$ independently. 
We want to construct  another stopping domain $[\sig_\l,\sig_r]$ such that $X(\sig_\l) = Y(\sig_\l)$ and $X(\sig_r) = Y(\sig_r)$ and show that with high probability $[\sig_\l,\sig_r] \supseteq [-M/3, M/3]$. %
To that end, define the event
    \def\MT{\mathsf{Meets}} %
    \begin{align*}
        \MT := \big\{ \exists\ t_\l \in [-M/2, -M/3], t_r \in [M/3, M/2] : X(t_\l) \leq Y(t_\l), X(t_r) \leq Y(t_r) \big\}.
    \end{align*}
  We remark that by continuity of paths, the event $\MT$ implies the existence of the desired stopping domain $[\sig_\l,\sig_r] \supseteq [-M/3, M/3]$.  Therefore, to prove that the latter exists with probability
$\geq 1 - 20\eps$ it is sufficient to show that $\P(\MT) \geq 1 - 20\eps$.
In view of the second step described above, we can reduce our task to showing that 
    \begin{align*}
        \P(\MT\, |\, \CD) \geq 1 - 10\eps.
    \end{align*}
    Since $[\tau_\l, \tau_r]$ is a stopping domain, we may condition on $\cB_{[\tau_\l, \tau_r]}$, so that $X|_{[\tau_\l, \tau_r]} \sim \L^{X(\tau_\l), X(\tau_r)}_{\tau_\l, \tau_r}$.
On $\CD$, we have $X(\tau_\l), X(\tau_r) \leq C$, so by monotonicity it is sufficient to show that
    \begin{align*}
        \P[\MT'] \geq 1 - 10\eps
    \end{align*}
    where $\MT'$ is defined as  $\MT$ with $X' \sim \L^{C}_{\tau_\l, \tau_r}$ instead of $X$. 

Let $V$ be a sample from a measure such that $V - C \sim \FS$, coupled monotonically to $X'$ such that $V \succeq X'$, with $V$ independent from $Y$. Also let $Z$ be the area-tilted Brownian excursion pinned at $0$ at every even integer between $-M/2$ and $M/2$, independent from $X', V$ and coupled monotonically to $Y$, so that $Y \succeq Z$. 
Let $\MT''$ be the version of $\MT$ with $V$ and $Z$ instead of $X$ and $Y$. Observe that $\MT'' \implies \MT'$.
It is therefore sufficient to show that $\P(\MT'') \geq 1 - 10\eps$.

    \begin{figure}[ht]
        \centering
        \input{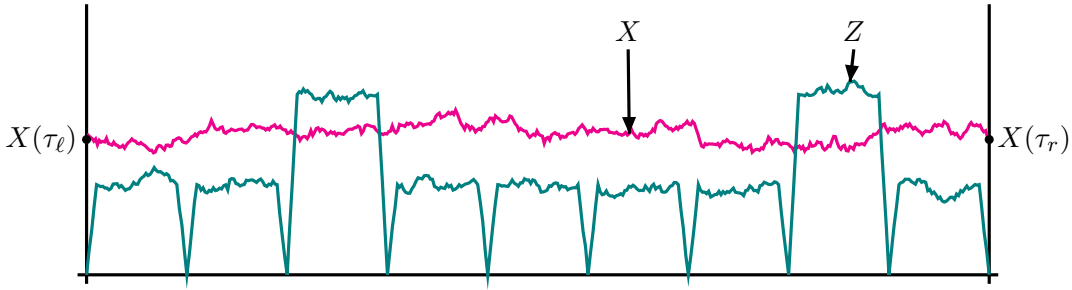}
        \caption{The pinned process $Z$ intersects the process $X$ on account of fluctuations.}
        \label{fig:fscouple}
    \end{figure}

    To show this, let $n \in \{\floor{M/2}, \floor{M/2} - 1\}$, and $m \in \{\ceil{M/3}, \ceil{M/3} + 1\}$ be even, so that $Z$ is pinned at zero for every even integer $2k \in [-n, -m] \cup [m, n]$.
Define $2k_\l + 1$ as the point where $Z(2k + 1)$ is maximized among $2k + 1 \in [-n, -m]$ and similarly define $2k_r + 1$ as the point where $Z(2k + 1)$ is maximized among $2k + 1 \in [m,n]$.
By independence of $Z|_{[2k, 2k + 2]}$ for different intervals $[2k, 2k + 2] \subseteq [-n, -m] \cup [m, n]$, and since $n - m = \T{M}$,
we have that $Z(2k_\l + 1), Z(2k_r + 1) = \T{\sqrt{\log M}}$ with probability $\geq 1 - \eps$, assuming  $M \geq M_0(\eps)$,
where we use the easily checked fact that the tails of $Z(2k + 1)$ have a Gaussian lower bound (see for instance Lemma 6.2 in \cite{geomarea}),  and therefore the maximum of $\T{M}$ of them is at least of size $\T{\log M}$ with large probability.
   However due to the one-point tail behavior of $\FS$, see \rlem{lem:fsfact}, the independence of $V,Z$ shows that
    \begin{align*}
        \P(V(2k_\l + 1) \geq Z(2k_\l + 1)) \leq \exp\big(-c\,(\log M)^{3/4}\big) \leq 5\eps\,,
    \end{align*}
for some %
absolute constant $c>0$. By symmetry, the very same bound applies to $2k_r + 1$.  So by a union bound
    \begin{align*}
        \P(V(2k_\l + 1) < Z(2k_\l + 1), V(2k_r + 1) < Z(2k_r+ 1)) \geq 1 - 10\eps,
    \end{align*}
    which implies $\P(\MT'')\geq 1-10\eps$ as desired. This ends the proof that $\P(\MT) \geq 1 - 20\eps$. 

    On the event $\MT $ we know that a domain $[\sig_\l, \sig_r]$ exists such that $X(\sig_\l) = Y(\sig_\l)$, $X(\sig_r) = Y(\sig_r)$ and such that $[\sig_\l,\sig_r] \supseteq [-M/3, M/3]$. Using that  $X$ are $Y$ are independent, a generalization of the strong Gibbs property of Lemma \ref{sabg}, where the domain $[\tau_{\ell}, \tau_r]$ is measurable with respect to the randomness of \emph{both} $X$ and $Y$ on $\R \setminus [\sig_\l, \sig_r],$ allows us to resample both $X$ and $Y$ inside $[\sig_\l, \sig_r]$ so that they agree everywhere on $[-M/3, M/3]$ on $\MT $. 
This gives us the
total variation distance  bound   \begin{align*}
        \tv{X|_{[-M/3, M/3]}, Y|_{[-M/3, M/3]}} \leq \P(\MT^\c) \leq 20\eps.
    \end{align*}
    Fix $L > 0$.
The above implies that for any fixed $\eps > 0$ and for all $M$ sufficiently large,
    \begin{align*}
        \tv{X|_{[-L, L]}, Y|_{[-L, L]}} \leq 20\eps.
    \end{align*}
    Recall that the definition of $Y$ depends on $M$, but \rlem{lem:fsfact} shows that $Y|_{[-L, L]}$ converges weakly as $M\to\infty$ to $\FS|_{[-L, L]}$ for any $L>0$. Convergence in total variation is stronger than weak convergence.   
Therefore, using  a triangle  inequality for a metric compatible with the weak topology,  and sending $\eps \to 0$ shows that $X|_{[-L, L]}$ has the same law as $\FS|_{[-L, L]}$ for any $L>0$.  Since $L$ is arbitrary this proves that $\mu=\FS$.
\end{proof}

\subsection{Existence}
In this section, we accomplish the existence part of Theorem \ref{thmoneline} and construct a Gibbs measure with given values of $L$ and $R.$ 
We start with the case of  finite values of $L,R$. 

\begin{ppn}[Existence for finite $L,R$] %
    \label{lem:existence}
    For any $(L,R)\in\R^2$ with $L+R<0$, there exists $\nu_{L,R}\in\cG$ satisfying  \eqref{linearlimla} a.s.
\end{ppn}
The proof requires several steps. 
We first note that 
it suffices to show that for any $K\in(0,\infty)$,     there exists  $\mu_K\in\cG$ with $L(\mu_K)=R(\mu_K) = -2K$.
Indeed, if $h\in\R$, and $\mu_K\in\cG$ has limits $L(\mu_K)=R(\mu_K) = -2K$ a.s., then by shifting 
as in the proof of Lemma \ref{lem:lpr} one has  $\theta_h\mu_K\in\cG$, with limits $R(\theta_h\mu_K)=R(\mu_K)-2h$
and $L(\theta_h\mu_K)=L(\mu_K)+2h$.

    \begin{figure}[ht]
        \centering{
                \input{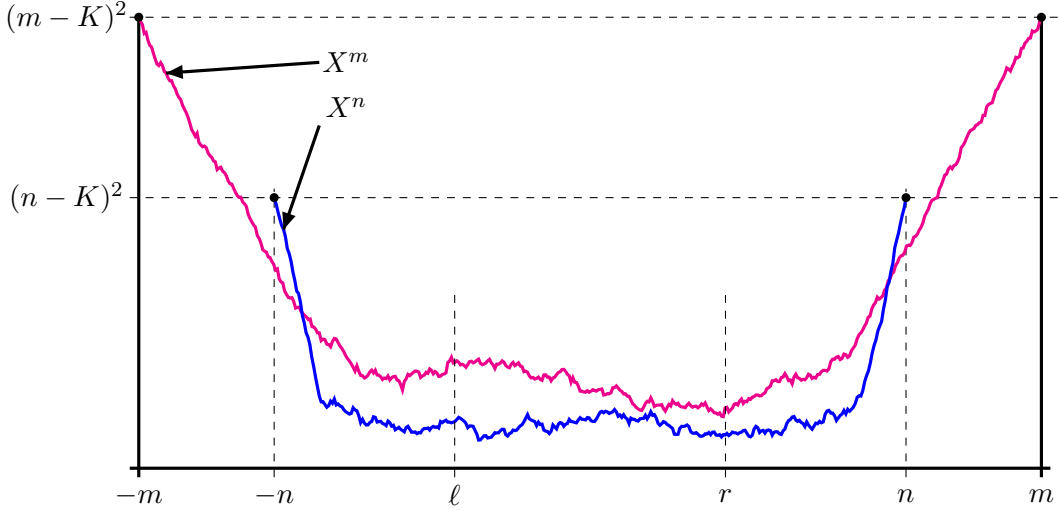}\qquad\qquad
}        \caption{Two independent copies of $X^{n}, X^{m}$ defined above. We resample inside the domain where they meet.}
        \label{fig:consistent-family}
    \end{figure}
    
    The strategy of proof is as follows. Fix $\l < r$, and a large integer $n$. 
    Consider $X^{n} \sim \L^{(n - K)^2}_{n}$, the area-tilted line on $[-n,n]$ with boundary $(n - K)^2$, and let $\mu^{n}_{\l, r}$ be the distribution of the restriction $X^{n}|_{[\l, r]}$ to the finite interval $[\ell,r]\subset [-n,n]$, see Figure \ref{fig:consistent-family}.
 A key step in the proof consists in showing 
 that for each $\ell<r$, there exists a probability measure $\mu_{\ell,r}$ on $\Omega_{\ell,r}^+$ such that  one has the weak convergence
 \begin{align}\label{eq:weakcon}
 \mu^{n}_{\ell,r}\;\Longrightarrow\; \mu_{\ell,r}\,,\qquad n\to\infty.
\end{align}
We pause to note that once this is achieved it is not too hard to see that $\{\mu_{\ell,r}, \;\ell<r\}$, must be a consistent family 
and that the global measure $\mu\in\calP^+$ which they define,  by Kolmogorov's consistency theorem, is a Gibbs measure. 
Indeed, the consistency follows from the consistency of $\{\mu^{n}_{\ell,r}, \;\ell<r\}$ for each fixed $n$. Moreover, since $\mu^{n}_{\ell,r}$ has the Gibbs property for every subdomain  $[a,b]\subset[\l, r]$, it follows that the same holds for the limit $ \mu_{\ell,r}$. This follows from the fact that according to our choice of (local) weak topology, a weak limit of Gibbs measures is a Gibbs measure, see e.g.\ Theorem 4.17 in \cite{georgii2011gibbs} or Theorem 6.26 in \cite{friedli2017statistical} for similar statements. 
Finally, the latter property for each $\ell<r$, together with the consistency, implies that the measure $\mu$ must have the Gibbs property on every domain, since being Gibbs is a condition that involves  finite domain marginals only.

Once the Gibbs measure $\mu$ has been constructed as above, it remains to check %
that it has the correct limiting behavior, namely that $\mu$-almost surely
\begin{align}\label{eq:LRlim}
 L(\mu)=R(\mu)=-2K.
\end{align}

We turn to  the details of the proof. For the first part of the proof, namely the convergence \eqref{eq:weakcon}, we will actually show the stronger fact that  $\mu^{n}_{\ell,r}$, $n\in\bbN$, forms a Cauchy sequence of probability measures on $\Omega^+_{\ell,r}$ for the total variation norm, see Lemma \ref{lem:cauchytv} below. Lemma \ref{lem:limit2K} will then establish the identity \eqref{eq:LRlim}.

We start %
with a technical estimate which quantifies the normal fluctuations around the mean for a Brownian bridge conditioned to avoid a parabola.  
Recall the notation $\B^{x,y}_n[h]$ for the Brownian bridge on $[-n,n]$ with boundary $x,y$ and conditioned to stay above the function $h$, and recall that $\p(t)=t^2$. The estimate shows that the effect of the conditioning does not significantly alter the fluctuations at sufficiently large times.
 We refer to Figure \ref{fig:linear-portion} for an illustration of the quantities in the lemma below.

\begin{lem}%
    \label{lem:closei}
    For any fixed $K > 0$ and $\beta>0$, there exist positive constants $\delta_0,c,C$ such that, if $\del\le \del_0$, $y^+ \in [-2(K + \del)n, -2(K - \del)n]$, and $y^-\leq -\beta n$,
    then $Y^n\sim  \B^{y^-,y^+}_n\![-\p]$ satisfies 
    \begin{align*}
        \P[Y^n(s) - \f{sy^{+}}{n} \in {\left[- u{\sqrt s},  +u{\sqrt s}\,\right]}
] \geq 1 - C\Exp{-cu}\,,
    \end{align*}
    for all $s,u\geq C$, and for all $n\geq C(u+s)$. 

\end{lem}
\begin{figure}[ht]
    \centering
\begin{tikzpicture}[scale=1]
\colorlet{colora}{magenta}
\colorlet{colorb}{teal}
\colorlet{colorc}{blue}
\draw[dashed] (-5.2,0) -- (5.2,0);
\draw[colorb,thick] (-5,-5) parabola bend (0,0) (5,-5);
\draw (-5.3,-2) node {$y^-$};
\draw (5.3,-2) node {$y^+$};
\fill[colora] (-5,-2) circle[radius=0.05];
\fill[colora] (5,-2) circle[radius=0.05];
\fill[colorc] (0,0) circle[radius=0.05];
\draw[very thick,colorc] (5,-2) -- (0,0);
\draw[thick,dashed,colorc] (0,-1) -- (2.5,-1);
\draw (-0.3,-0.7) node {$sy^+/n$};
\fill[colorc] (2.5,-1) circle[radius=0.05];
\draw[very thick,colora] (-5,-2) -- (-4.97958,-1.93835) -- (-4.95915,-1.9164) -- (-4.93873,-1.90439) -- (-4.9183,-1.94844) -- (-4.89788,-1.94604) -- (-4.87745,-1.94568) -- (-4.85703,-1.93734) -- (-4.8366,-1.94452) -- (-4.81618,-1.93466) -- (-4.79575,-1.93718) -- (-4.77533,-1.96334) -- (-4.7549,-1.92592) -- (-4.73448,-1.88914) -- (-4.71405,-1.82922) -- (-4.69363,-1.81774) -- (-4.6732,-1.81928) -- (-4.65278,-1.8246) -- (-4.63235,-1.85821) -- (-4.61193,-1.81863) -- (-4.5915,-1.8392) -- (-4.57108,-1.86154) -- (-4.55065,-1.8552) -- (-4.53023,-1.80043) -- (-4.5098,-1.78227) -- (-4.48938,-1.8003) -- (-4.46895,-1.80886) -- (-4.44853,-1.77878) -- (-4.4281,-1.77156) -- (-4.40768,-1.78166) -- (-4.38725,-1.77594) -- (-4.36683,-1.74225) -- (-4.34641,-1.67384) -- (-4.32598,-1.69866) -- (-4.30556,-1.7051) -- (-4.28513,-1.71627) -- (-4.26471,-1.77316) -- (-4.24428,-1.78594) -- (-4.22386,-1.8011) -- (-4.20343,-1.75435) -- (-4.18301,-1.74423) -- (-4.16258,-1.77673) -- (-4.14216,-1.74347) -- (-4.12173,-1.73934) -- (-4.10131,-1.77472) -- (-4.08088,-1.7764) -- (-4.06046,-1.77759) -- (-4.04003,-1.7867) -- (-4.01961,-1.76958) -- (-3.99918,-1.81771) -- (-3.97876,-1.80803) -- (-3.95833,-1.76172) -- (-3.93791,-1.72073) -- (-3.91748,-1.67305) -- (-3.89706,-1.62582) -- (-3.87663,-1.56881) -- (-3.85621,-1.58724) -- (-3.83578,-1.54893) -- (-3.81536,-1.58856) -- (-3.79493,-1.59121) -- (-3.77451,-1.63107) -- (-3.75408,-1.58523) -- (-3.73366,-1.53194) -- (-3.71324,-1.52329) -- (-3.69281,-1.45824) -- (-3.67239,-1.46431) -- (-3.65196,-1.44645) -- (-3.63154,-1.42868) -- (-3.61111,-1.44584) -- (-3.59069,-1.4195) -- (-3.57026,-1.35058) -- (-3.54984,-1.36757) -- (-3.52941,-1.329) -- (-3.50899,-1.32112) -- (-3.48856,-1.25819) -- (-3.46814,-1.23308) -- (-3.44771,-1.20277) -- (-3.42729,-1.17381) -- (-3.40686,-1.1565) -- (-3.38644,-1.18633) -- (-3.36601,-1.13567) -- (-3.34559,-1.16478) -- (-3.32516,-1.14709) -- (-3.30474,-1.11881) -- (-3.28431,-1.13239) -- (-3.26389,-1.11739) -- (-3.24346,-1.07324) -- (-3.22304,-1.06385) -- (-3.20261,-1.07914) -- (-3.18219,-1.15185) -- (-3.16176,-1.1673) -- (-3.14134,-1.17134) -- (-3.12092,-1.17294) -- (-3.10049,-1.18697) -- (-3.08007,-1.16228) -- (-3.05964,-1.16509) -- (-3.03922,-1.15433) -- (-3.01879,-1.12103) -- (-2.99837,-1.13206) -- (-2.97794,-1.12678) -- (-2.95752,-1.05113) -- (-2.93709,-1.0135) -- (-2.91667,-1.03839) -- (-2.89624,-1.02786) -- (-2.87582,-0.99892) -- (-2.85539,-0.95123) -- (-2.83497,-0.9353) -- (-2.81454,-0.93891) -- (-2.79412,-0.94794) -- (-2.77369,-0.91725) -- (-2.75327,-0.89647) -- (-2.73284,-0.8996) -- (-2.71242,-0.8931) -- (-2.69199,-0.89301) -- (-2.67157,-0.86009) -- (-2.65114,-0.86807) -- (-2.63072,-0.84204) -- (-2.61029,-0.8059) -- (-2.58987,-0.78072) -- (-2.56944,-0.79358) -- (-2.54902,-0.74367) -- (-2.52859,-0.7524) -- (-2.50817,-0.72916) -- (-2.48775,-0.73991) -- (-2.46732,-0.77853) -- (-2.4469,-0.75985) -- (-2.42647,-0.73959) -- (-2.40605,-0.72994) -- (-2.38562,-0.71017) -- (-2.3652,-0.71701) -- (-2.34477,-0.72604) -- (-2.32435,-0.7184) -- (-2.30392,-0.65132) -- (-2.2835,-0.61668) -- (-2.26307,-0.62069) -- (-2.24265,-0.59356) -- (-2.22222,-0.59642) -- (-2.2018,-0.5984) -- (-2.18137,-0.5625) -- (-2.16095,-0.56669) -- (-2.14052,-0.60346) -- (-2.1201,-0.61728) -- (-2.09967,-0.60216) -- (-2.07925,-0.58601) -- (-2.05882,-0.53814) -- (-2.0384,-0.54699) -- (-2.01797,-0.51407) -- (-1.99755,-0.49033) -- (-1.97712,-0.48331) -- (-1.9567,-0.44806) -- (-1.93627,-0.42314) -- (-1.91585,-0.43723) -- (-1.89542,-0.39554) -- (-1.875,-0.42001) -- (-1.85458,-0.40717) -- (-1.83415,-0.39601) -- (-1.81373,-0.37279) -- (-1.7933,-0.38271) -- (-1.77288,-0.37247) -- (-1.75245,-0.36524) -- (-1.73203,-0.32543) -- (-1.7116,-0.33747) -- (-1.69118,-0.32422) -- (-1.67075,-0.32057) -- (-1.65033,-0.34642) -- (-1.6299,-0.25792) -- (-1.60948,-0.24955) -- (-1.58905,-0.28159) -- (-1.56863,-0.23001) -- (-1.5482,-0.25013) -- (-1.52778,-0.22842) -- (-1.50735,-0.19288) -- (-1.48693,-0.20104) -- (-1.4665,-0.20261) -- (-1.44608,-0.20705) -- (-1.42565,-0.19945) -- (-1.40523,-0.23446) -- (-1.3848,-0.25748) -- (-1.36438,-0.24853) -- (-1.34395,-0.20582) -- (-1.32353,-0.17541) -- (-1.3031,-0.11403) -- (-1.28268,-0.08329) -- (-1.26225,-0.0874) -- (-1.24183,-0.08012) -- (-1.22141,-0.1273) -- (-1.20098,-0.14272) -- (-1.18056,-0.15932) -- (-1.16013,-0.13846) -- (-1.13971,-0.10703) -- (-1.11928,-0.11056) -- (-1.09886,-0.11461) -- (-1.07843,-0.14575) -- (-1.05801,-0.13548) -- (-1.03758,-0.11446) -- (-1.01716,-0.08801) -- (-0.99673,-0.097) -- (-0.97631,-0.02637) -- (-0.95588,0.02034) -- (-0.93546,0.02944) -- (-0.91503,0.02621) -- (-0.89461,0.02915) -- (-0.87418,0.04575) -- (-0.85376,0.05738) -- (-0.83333,0.08333) -- (-0.83333,0.08333) -- (-0.81197,0.10092) -- (-0.7906,0.13548) -- (-0.76923,0.1628) -- (-0.74786,0.15267) -- (-0.7265,0.15722) -- (-0.70513,0.13891) -- (-0.68376,0.11503) -- (-0.66239,0.16458) -- (-0.64103,0.14367) -- (-0.61966,0.19474) -- (-0.59829,0.17621) -- (-0.57692,0.17557) -- (-0.55556,0.1993) -- (-0.53419,0.19126) -- (-0.51282,0.19408) -- (-0.49145,0.17819) -- (-0.47009,0.15569) -- (-0.44872,0.14262) -- (-0.42735,0.11558) -- (-0.40598,0.1557) -- (-0.38462,0.14226) -- (-0.36325,0.18427) -- (-0.34188,0.09641) -- (-0.32051,0.08937) -- (-0.29915,0.09826) -- (-0.27778,0.08184) -- (-0.25641,0.10625) -- (-0.23504,0.11674) -- (-0.21368,0.10017) -- (-0.19231,0.1046) -- (-0.17094,0.17832) -- (-0.14957,0.17551) -- (-0.12821,0.17405) -- (-0.10684,0.1509) -- (-0.08547,0.1449) -- (-0.0641,0.09734) -- (-0.04274,0.08258) -- (-0.02137,0.14567) -- (0,0.16667) -- (0,0.16667) -- (0.02137,0.17844) -- (0.04274,0.20334) -- (0.0641,0.2503) -- (0.08547,0.25844) -- (0.10684,0.30946) -- (0.12821,0.32996) -- (0.14957,0.32347) -- (0.17094,0.36237) -- (0.19231,0.28105) -- (0.21368,0.25359) -- (0.23504,0.25104) -- (0.25641,0.21395) -- (0.27778,0.19962) -- (0.29915,0.17133) -- (0.32051,0.14303) -- (0.34188,0.21361) -- (0.36325,0.22159) -- (0.38462,0.20934) -- (0.40598,0.23722) -- (0.42735,0.21896) -- (0.44872,0.23377) -- (0.47009,0.23249) -- (0.49145,0.26839) -- (0.51282,0.2143) -- (0.53419,0.22512) -- (0.55556,0.21674) -- (0.57692,0.17098) -- (0.59829,0.19104) -- (0.61966,0.15387) -- (0.64103,0.19019) -- (0.66239,0.19324) -- (0.68376,0.20141) -- (0.70513,0.19127) -- (0.7265,0.2053) -- (0.74786,0.13789) -- (0.76923,0.1285) -- (0.7906,0.09861) -- (0.81197,0.09584) -- (0.83333,0.08333) -- (0.83333,0.08333) -- (0.85376,0.09104) -- (0.87418,0.05682) -- (0.89461,0.09961) -- (0.91503,0.15495) -- (0.93546,0.1368) -- (0.95588,0.12435) -- (0.97631,0.1206) -- (0.99673,0.14464) -- (1.01716,0.18606) -- (1.03758,0.13776) -- (1.05801,0.10712) -- (1.07843,0.06413) -- (1.09886,0.02037) -- (1.11928,-0.03775) -- (1.13971,-0.03766) -- (1.16013,-0.08127) -- (1.18056,-0.0481) -- (1.20098,-0.03282) -- (1.22141,-0.07344) -- (1.24183,-0.11345) -- (1.26225,-0.1409) -- (1.28268,-0.11048) -- (1.3031,-0.14518) -- (1.32353,-0.19778) -- (1.34395,-0.20823) -- (1.36438,-0.20658) -- (1.3848,-0.23463) -- (1.40523,-0.19723) -- (1.42565,-0.20022) -- (1.44608,-0.19727) -- (1.4665,-0.23199) -- (1.48693,-0.24288) -- (1.50735,-0.25709) -- (1.52778,-0.33168) -- (1.5482,-0.33179) -- (1.56863,-0.39162) -- (1.58905,-0.39995) -- (1.60948,-0.44962) -- (1.6299,-0.39862) -- (1.65033,-0.39727) -- (1.67075,-0.38085) -- (1.69118,-0.36333) -- (1.7116,-0.41049) -- (1.73203,-0.48425) -- (1.75245,-0.50091) -- (1.77288,-0.48236) -- (1.7933,-0.48744) -- (1.81373,-0.47277) -- (1.83415,-0.45428) -- (1.85458,-0.5058) -- (1.875,-0.50559) -- (1.89542,-0.4576) -- (1.91585,-0.48331) -- (1.93627,-0.55039) -- (1.9567,-0.56493) -- (1.97712,-0.58603) -- (1.99755,-0.58166) -- (2.01797,-0.61504) -- (2.0384,-0.61985) -- (2.05882,-0.69074) -- (2.07925,-0.67259) -- (2.09967,-0.63407) -- (2.1201,-0.60479) -- (2.14052,-0.65673) -- (2.16095,-0.65537) -- (2.18137,-0.74866) -- (2.2018,-0.74685) -- (2.22222,-0.75523) -- (2.24265,-0.79025) -- (2.26307,-0.78648) -- (2.2835,-0.78886) -- (2.30392,-0.76849) -- (2.32435,-0.77196) -- (2.34477,-0.77247) -- (2.3652,-0.8355) -- (2.38562,-0.87499) -- (2.40605,-0.89433) -- (2.42647,-0.89139) -- (2.4469,-0.89738) -- (2.46732,-0.86037) -- (2.48775,-0.89209) -- (2.50817,-0.90534) -- (2.52859,-0.90438) -- (2.54902,-0.91755) -- (2.56944,-0.90988) -- (2.58987,-0.93303) -- (2.61029,-0.89901) -- (2.63072,-0.91563) -- (2.65114,-0.89107) -- (2.67157,-0.89998) -- (2.69199,-0.86541) -- (2.71242,-0.91434) -- (2.73284,-0.90775) -- (2.75327,-0.92862) -- (2.77369,-0.9303) -- (2.79412,-0.90466) -- (2.81454,-0.9464) -- (2.83497,-0.99272) -- (2.85539,-1.02905) -- (2.87582,-1.02903) -- (2.89624,-1.09041) -- (2.91667,-1.09406) -- (2.93709,-1.13735) -- (2.95752,-1.17589) -- (2.97794,-1.19171) -- (2.99837,-1.18597) -- (3.01879,-1.20934) -- (3.03922,-1.22859) -- (3.05964,-1.23479) -- (3.08007,-1.24766) -- (3.10049,-1.32917) -- (3.12092,-1.34502) -- (3.14134,-1.34602) -- (3.16176,-1.31825) -- (3.18219,-1.26414) -- (3.20261,-1.25077) -- (3.22304,-1.25467) -- (3.24346,-1.28854) -- (3.26389,-1.29179) -- (3.28431,-1.27546) -- (3.30474,-1.24579) -- (3.32516,-1.27823) -- (3.34559,-1.2783) -- (3.36601,-1.26201) -- (3.38644,-1.21637) -- (3.40686,-1.21776) -- (3.42729,-1.2421) -- (3.44771,-1.27038) -- (3.46814,-1.2888) -- (3.48856,-1.31545) -- (3.50899,-1.29951) -- (3.52941,-1.27715) -- (3.54984,-1.2653) -- (3.57026,-1.24054) -- (3.59069,-1.26353) -- (3.61111,-1.26241) -- (3.63154,-1.28338) -- (3.65196,-1.29189) -- (3.67239,-1.26666) -- (3.69281,-1.27801) -- (3.71324,-1.34178) -- (3.73366,-1.35837) -- (3.75408,-1.38798) -- (3.77451,-1.43995) -- (3.79493,-1.47051) -- (3.81536,-1.4506) -- (3.83578,-1.41072) -- (3.85621,-1.43387) -- (3.87663,-1.41412) -- (3.89706,-1.38996) -- (3.91748,-1.43332) -- (3.93791,-1.44397) -- (3.95833,-1.45475) -- (3.97876,-1.45932) -- (3.99918,-1.51448) -- (4.01961,-1.53314) -- (4.04003,-1.5114) -- (4.06046,-1.48447) -- (4.08088,-1.44202) -- (4.10131,-1.46264) -- (4.12173,-1.41995) -- (4.14216,-1.41275) -- (4.16258,-1.46377) -- (4.18301,-1.45334) -- (4.20343,-1.47516) -- (4.22386,-1.47183) -- (4.24428,-1.48654) -- (4.26471,-1.46368) -- (4.28513,-1.4687) -- (4.30556,-1.48922) -- (4.32598,-1.47387) -- (4.34641,-1.4793) -- (4.36683,-1.50394) -- (4.38725,-1.47998) -- (4.40768,-1.50488) -- (4.4281,-1.53602) -- (4.44853,-1.56393) -- (4.46895,-1.61795) -- (4.48938,-1.64498) -- (4.5098,-1.60267) -- (4.53023,-1.64057) -- (4.55065,-1.6517) -- (4.57108,-1.69472) -- (4.5915,-1.72286) -- (4.61193,-1.68975) -- (4.63235,-1.67531) -- (4.65278,-1.70442) -- (4.6732,-1.71553) -- (4.69363,-1.75039) -- (4.71405,-1.78356) -- (4.73448,-1.81998) -- (4.7549,-1.85535) -- (4.77533,-1.85525) -- (4.79575,-1.8476) -- (4.81618,-1.84342) -- (4.8366,-1.85833) -- (4.85703,-1.90711) -- (4.87745,-1.91819) -- (4.89788,-1.95375) -- (4.9183,-1.93718) -- (4.93873,-1.93044) -- (4.95915,-1.91859) -- (4.97958,-1.93849) -- (5,-2);
\draw[very thick,colorb] (2.5,-1.6) -- (2.5,-0.4);
\draw[very thick,colorb] (2.43,-1.6) -- (2.57,-1.6);
\draw[very thick,colorb] (2.43,-0.4) -- (2.57,-0.4);
\draw (-5,0.3) node {$-n$};
\draw[thick] (-5,-0.1) -- (-5,0.1);
\draw (2.5,0.3) node {$s$};
\draw[thick] (2.5,-0.1) -- (2.5,0.1);
\draw (5,0.3) node {$n$};
\draw[thick] (5,-0.1) -- (5,0.1);
\draw[very thick] (1,-2) -- (2.5,-1.42);
\fill[very thick] (2.5,-1.42) -- (2.35627,-1.55908) -- (2.3001,-1.41379);
\draw (0.4,-2) node {$2u\sqrt{s}$};
\end{tikzpicture}
    \caption{Illustration of the Brownian trajectory quantified in Lemma \ref{lem:closei}.}
    \label{fig:linear-portion}
\end{figure}
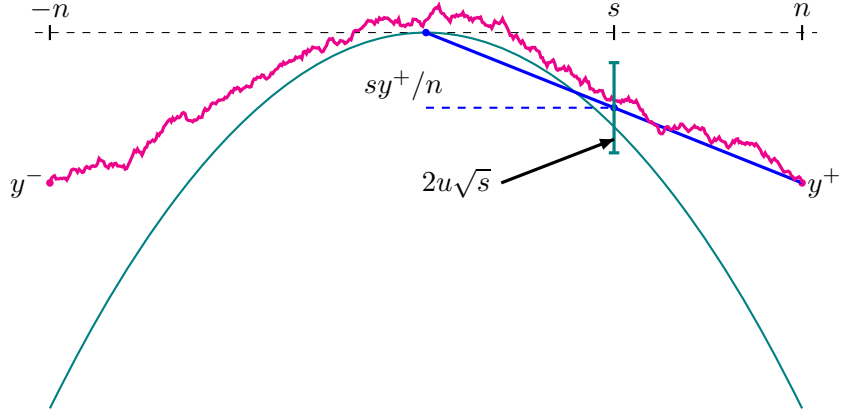

\begin{proof}

    \def\l{\ell}     We divide the argument into two parts.

\smallskip

\noindent   $(1)$ \textit{Lower bound: } By monotonicity, we may replace $Y^n|_{[0,n]}$ by $X\sim \B^{0, y^{+}}_{0, n}$. Since $X(t)= ty^+/n + X'(t)$ where $X'\sim \B^{0}_{0, n}$, a standard estimate for the Brownian bridge shows that 
            \begin{align*}
                \P[X(s) \geq \f{sy^{+}}{n} - u{\sqrt s}] \geq 1 - C\Exp{-cu^2},
            \end{align*}
       as soon as e.g.\ $n \geq 2s$.

\smallskip

\noindent            $(2)$ \textit{Upper bound: }  For the upper bound, we first  control the value of $Y^n(0)$.
                Observing that the distribution of the latter is monotone in the values $y^-, y^+$, and using %
                $y_2 = -2(K-\del)n\geq y^{-}, y_{1} = -\beta n\geq y^{-}$ in \rlem{lem:tailbounds}, the assumptions $\del\le \del_0$, $K-\del_0>0$ and $\beta>0$ show  %
                that $Y^n(0) \leq u$ with probability at least $1 - C\Exp{-cu}$, for all $u>0$ and  $n \geq 1$.
            Thus, using  monotonicity with respect to  $Y^n(0)$, we may replace $Y^n(s)$ by $X(s)$ where $X \sim \B^{u, y^{+}}_{0, n}[-\p]$.
Let $X'\sim \B^{u, y^{+}}_{0, n}$, that is the same as $X$ but without the conditioning, %
and let $Z \sim \B^{u, -2(K + \del_0)n}_{0, n}$.
Since $y^{+} \geq -2(K + \del_0)n$, monotonicity of Brownian bridges in the boundary conditions implies that %
            \begin{align*}
                \P[X' \succ %
                 -\p] \geq \P[Z \succ %
                  -\p].
            \end{align*}
            Now note that the line between $(0, u)$ and $(n, -2(K + \del_0)n)$ takes the form:
            \begin{align*}
                t \mapsto u- \f{2(K + \del_0)n + u}{n}\,t . %
            \end{align*}
            The distance between this and $-\p$ is minimized at 
            \begin{align*}
                t^* = K + \del_0 + \f{u}{2n}.
            \end{align*}
            Assume $n\ge u/2\del_0$, so that $t^* \in [K + \del_0, K + 2\del_0]$.
The actual distance at this $t^*$ is then
            \begin{align*}
                d^* & = u - \f{1}{4}\Rnd{2(K + \del_0) + \f{u}{n}}^2 \geq u - (K + 2\del_0)^2.
            \end{align*}
 We may assume also that $u\geq    100(K + 2\del_0)^2$, so that $d^*\geq .99u$ and one has  
            \begin{align*}
                \P[Z \succ  -\p] \geq %
               1/2,
            \end{align*}
  by an application of Corollary \ref{paraavoidancecor}, as soon as $u$ is also larger than some absolute constant.
This allows us to replace probabilities about $X \sim \B^{u, y^{+}}_{0, n}[-\p]$ by
probabilities about the unconditional $X'\sim \B^{u, y^{+}}_{0, n}$ up to a %
factor of 2.
 Note that $X'-w \sim   \B^{0}_{0, n}$ where  $w$ is the line
            \begin{align*}
               w(t)=u + \f{y^{+} - u}{n} \,t.
            \end{align*}
            Since $w(s)=\f{sy^{+}}{n} + u\Rnd{1 - \f{s}{n}}\leq \f{sy^{+}}{n} + u$,
              a simple estimate for the Gaussian fluctuations of $X'(s)$ allows us to conclude
            \begin{align*}
                \P[Y^n(s) \geq \f{sy^{+}}{n} + \sqrt{s}u] &\leq 2 \,\P[X'(s) \geq w(s) %
                + \tfrac12\sqrt{s}u] 
                                                                                 \leq C\Exp{-cu^2}\,,
            \end{align*}
as soon as $s\geq 4$ and $n\geq s$.
\end{proof} %

We now turn to the proof that $\mu^n_{[\l, r]}$ 
form a Cauchy sequence in total variation, for each $\ell<r$.

\begin{lem}[Cauchy in total variation] %
    \label{lem:cauchytv}
    Fix $\ell < r$ and let $X^n\sim \L^{(n - K)^2}_{n}$ %
    Then for $n < m$, 
$\tv{X^n|_{[\l, r]}, X^m|_{[\l, r]}} = o(1)$, as  $n\to\infty$.
\end{lem}
\begin{proof}
    Assume $n < m$ and let $\al \in (0, 1/2)$ %
    be a parameter to be taken sufficiently small later. Switching to \PAR\,, see Lemma \ref{lem:girsanov},  
    we replace $X^n$ by the path $Y^n\sim \B^{-2Kn + K^2}_n[-\p]$. 
Since 
    \begin{align*}
        Y^n(\pm n) & = -2Kn + K^2 \in [-2(K + \del)n, -2(K - \del)n],   \\
        Y^m(\pm m) & = -2Km + K^2 \in [-2(K + \del)m, -2(K - \del)m]
    \end{align*}
    if $n$ is large enough given the constants $K,\delta$, we may use \rlem{lem:closei} for both $Y^n$ and $Y^m$,  with e.g.\  $s = n^\al$ and  $u = C_1\log n$ for some constant $C_1$. 
Applying \rlem{lem:closei} to $Y^n$ we have
    \begin{align*}
        \P[{Y^n( n^\al) \in -2Kn^\al + \Box{ K^2n^{\al - 1} - un^{\al/2} ,  K^2n^{\al - 1} + u n^{\al/2}}}] \geq 1 - C\Exp{-cu}.
    \end{align*}
    By symmetry, the same estimate holds for $Y^n( -n^\al)$. Thus, choosing the constant $C_1$ appropriately, we may estimate, for all $n$ large enough
    \begin{align}\label{eq:tvp1}
        \P[{Y^n(\pm n^\al) \in-2Kn^\al +  \Box{ - 2u n^{\al/2} ,  + 2u n^{\al/2}}}] \geq 1 - 1/n.
    \end{align}
    Similarly, for $Y^m$, $m>n$, with the same choices $s= n^\al$ and $u = C_1\log n$ in the application of  \rlem{lem:closei}  we have
    \begin{align}\label{eq:tvp2}
        \P[{Y^m(\pm n^\al) \in -2Kn^\al + \Box{ - 2un^{\al/2},  + 2un^{\al/2} }}] \geq 1 - 1/n.
    \end{align}

    Now choose $y_n, y_m \in I_{K,n}:=[-2Kn^\al - 2un^{\al/2}, -2Kn^\al + 2un^{\al/2}]$.
We want to show that uniformly over $y_n, y_m\in I_{K,n}$, conditioned on $Y^n(n^\al) = y_n, Y^m(n^\al) = y_m$ one has a small  total variation distance between $Y^n(n^{2\al}), Y^m(n^{2\al})$. With slight abuse of notation, we write this claim as follows
\begin{align}\label{eq:tvclab}
 \sup_{y_n,y_m\in I_{K,n}}\tv{Y^n(n^{2\al}), Y^m(n^{2\al})\,|\,Y^n(n^\al) = y_n,Y^m(n^\al) = y_m}
 =o(1)\,. \end{align}
Let us first check that this is the case for the path without the parabolic barrier. Namely, 
    let $Z^n \sim \B^{y_n, -2Kn + K^2}_{n^\al, n}$.
Observe that %
    \begin{align*}
        Z^n(n^{2\al}) \sim \Nor{z_n}{\f{(n - n^{2\al})(n^{2\al} - n^\al)}{n - n^\al}}\,,\qquad 
        Z^m(n^{2\al}) \sim \Nor{z_m}{\f{(m - n^{2\al})(n^{2\al} - n^\al)}{m - n^\al}} ,
    \end{align*}
    where
    \begin{align*}
        z_n &:= -2K\, \f{n(n^{2\al} - n^\al)}{n - n^\al} + K^2 \,\f{n^{2\al} - n^\al}{n - n^\al} + y_n\Rnd{1 - \f{n^{2\al} - n^\al}{n - n^\al}}, \\
        z_m &:= -2K \,\f{m(n^{2\al} - n^\al)}{m - n^\al} + K^2\, \f{n^{2\al} - n^\al}{m - n^\al} + y_m\Rnd{1 - \f{n^{2\al} - n^\al}{m - n^\al}},
    \end{align*}
    so that, uniformly in $y_n,y_m\in I_{K,n}$,
        \begin{align*}
        |z_n - z_m| = O\left({n^{3\al - 1} + n^{\al/2} \log n}\right).
    \end{align*}

We now recall the  following well known result on the total variation distance between Gaussians,  which can be verified by straightforward calculus and triangle inequality.
\begin{lem}[TV distance between Gaussians]\label{lem:tvgau}
    Let $X_1 \sim \Nor{\mu_1}{\sig_1}, X_2 \sim \Nor{\mu_2}{\sig_2}$. Then
    \begin{align*}
    \tv{X_1, X_2} \leq %
    C\,\left(\f{\Abs{\mu_1 - \mu_2}}{\max(\sig_1, \sig_2)} +\frac{|\sigma^2_1-\sigma^2_2|}{\max(\sig^2_1, \sig^2_2)}\right),
    \end{align*}
    for some absolute constant $C>0$.
\end{lem} 
Using the above    %
    \begin{align*}
        \tv{Z^n(n^{2\al}), Z^m(n^{2\al})} = o(1)
    \end{align*}
    if $\al$ is small enough.
Therefore we can focus on showing %
that 
\[
\tv{Y^n(n^{2\al}), Z^n(n^{2\al})\,|\,Y^n(n^\al) = y_n} = o(1),
\] 
uniformly in $y_n \in I_{K,n}$, 
and similarly for $Y^m, Z^m$.
To establish this, observe that it is enough to show that $\P(Z^n \gtr -\p) = 1 - o(1)$ because, conditionally on this  event, $Y^n$ and $Z^n$ can be coupled to be equal at time $n^{2\al}$.
Note that the line joining $(n^\al, -2Kn^\al - 2un^{\al/2})$ and $(n, -2Kn + K^2)$ grows farther from $\p$ as we move along the line from $t= n^\al$ to $t = n$ (if $n$ is large enough compared to $K$).
The distance at $t = n^{\al}$ is 
    \begin{align*}
        -2Kn^\al - 2un^{\al/2} + n^{2\al}
    \end{align*}
    so that by \rlem{paraavoidancecor}, and using monotonicity, the claim $\P(Z^n \gtr -\p) = 1 - o(1)$ is established, uniformly in $y_n\in I_{K,n}$. The same argument applies to $Z^m$. Summarizing, we have proved \eqref{eq:tvclab}. By symmetry, the same argument applies to times $-n^{2\al}$. Therefore, using also 
    \begin{align}\label{eq:estatv}
        \tv{Y^n(n^{2\al}), Y^m(n^{2\al})} = o(1).
    \end{align}
The same bound applies to time $-n^{2\al}$.
  Thanks to \eqref{eq:tvp1} and \eqref{eq:tvp2}, we may then show that there exists a coupling of $(Y^n(\pm n^{2\al}), Y^m(\pm n^{2\al})$ such that $Y^n(n^{2\al})=Y^m(n^{2\al})$ and $Y^n(-n^{2\al})=Y^m(-n^{2\al})$ with probability $1-o(1)$. Thus, %
  using the Gibbs property to resample on the interval $[-n^{2\al},n^{2\al}]$, 
  we
  obtain a coupling of $Y^n,Y^m$ such that with probability $1-o(1)$ one has $Y^n(t)= Y^m(t)$ for all $t\in[-n^{2\al},n^{2\al}]$. Since eventually $[\ell,r]\subset [-n^{2\al},n^{2\al}]$, this ends the proof.

\end{proof} %

As we discussed after \eqref{eq:weakcon}, the above proves the existence of a measure $\mu=\mu_K\in\cG$, by consistency of the limiting marginals $\mu_{\ell,r}$.
To prove Proposition \ref{lem:existence} it remains to prove the next lemma.
\begin{lem} %
    \label{lem:limit2K}
    For any $K>0$, if $\mu_K$ is defined as above and $X \sim \mu_K$, then $\mu$- almost surely\[
    \lim_{t \to \pm\infty} \f{X(t) - t^2}{|t|} = -2K.\]
\end{lem}
\begin{proof}
    We already know that $\mu=\mu_K\in\cG$ %
    and therefore the limit above exists a.s.\ by Lemma \ref{lem:limexists}.
    We therefore only need to show that the limit along increasing positive integers is indeed $-2K$.

    If   $X^n\sim \L^{(n - K)^2}_{n}$, by Lemma \ref{lem:cauchytv} we know that
$\tv{X^n|_{[-s, s]}, \mu|_{[-s, s]}} < s^{-2}$, for all $s\in\bbN$ and integers $n\ge n_0(s)$. Moreover, using the \PAR, for all $s$ large enough, and $u=C_1\log s$ for some absolute constant $C_1$, \rlem{lem:closei} shows 
    \begin{align}\label{eq:reas1}
        \P[X^n(s) - s^2 \in \Box{-2Ks - u{\sqrt s}, -2Ks + u{\sqrt s}}] \geq 1 - s^{-2}.
    \end{align}
    Therefore, for $X\sim\mu$ we have
    \begin{align}\label{eq:reas2}
        \P[X(s) - s^2 \in \Box{-2Ks - u{\sqrt s}, -2Ks + u{\sqrt s}}] \geq 1 - 2s^{-2}.
    \end{align}
    Since $2s^{-2}$ is summable, these events fail only for finitely many $s\in\bbN$, almost surely.
Therefore, a.s., the right limit $R$ exists along $\N$ and is $-2K$.
The same argument applies to the left limit $L$. 
\end{proof} %

We finally address the existence of the measures $\nu_{L,R}$ in the case where one (but not both) of $L,R$ equals $-\infty$. In the next lemma  this is accomplished %
by passing to the limit in the previous construction. 

\begin{lem}%
    \label{lem:existence2}
For any $x\in\R$,  there is a measure $\mu\in\cG$ with $L(\mu) = -\infty, R(\mu)=x$ and a  measure $\mu\in\cG$ with $L(\mu) = x, R(\mu)=-\infty$. 
\end{lem}
\begin{proof}
By symmetry, it suffices to show the statement with $L(\mu) = -\infty, R(\mu)=x$ only. Moreover, by the shifting argument already used in the  proof of Lemma \ref{lem:lpr}, it suffices to  show the statement with e.g.\ $x=-2$. 
    Let $\mu_{K,1}$ denote the %
    measure constructed in the proof of Proposition \ref{lem:existence}, with $L(\mu_{K,1}) = -2K$ and $R(\mu_{K,1}) = -2$. %
It is clear from our construction that  $\mu_{K,1}$, $K>0$,  are stochastically ordered in the sense that if $K_1 > K_2$ then $\mu_{K_1,1} \preceq \mu_{K_2,1}$. %
As a consequence of a more general statement about monotone line ensembles, see e.g.\  \cite[Theorem 1.5]{CIW19} and the proof of \cite[Theorem 3.7]{geomarea}, this implies that the weak limit $\mu := \lim_{K \to \infty} \mu_{K,1}$ exists.
The fact that $\mu$ is Gibbs follows from the Gibbs property of $\mu_{K,1}$ established in Proposition \ref{lem:existence}, since the limit of Gibbs measures is a Gibbs measure, see the discussion after \eqref{eq:weakcon}.

Thus $\mu\in\cG$ and by Lemma \ref{lem:limexists} the limiting values $L(\mu)$ and $R(\mu)$ exist a.s. 
Let $X \sim \mu$, and let $Y \sim \mu_{K,1}$.
Then $X$ and $Y$ may be coupled so that $X \preceq Y$ a.s., and
    \begin{align*}
        \lim_{t \to -\infty} (X(t) - t^2)/|t| \leq \lim_{t \to -\infty} (Y(t) - t^2)/|t| = -2K.
    \end{align*}
    Sending $K \to \infty$ shows that $L(\mu) = -\infty$. Similarly,  $R(\mu_{K,1})=-2$ for all $K>0$ implies $R(\mu)\leq -2$. To prove the lower bound $R(\mu)\geq -2$ a.s.\ it is sufficient to observe that, from the proof of the lower bound in Lemma \ref{lem:closei}, and reasoning as in \eqref{eq:reas2},  one obtains that uniformly in $K,$ $X\sim\mu_{K,1}$, and hence   $X\sim\mu$, 
    satisfy
      \begin{align*}
        \P[X(s) - s^2 \ge -2s - u{\sqrt s}] \geq 1 - 2s^{-2}.
    \end{align*}
\end{proof}

\subsection{Uniqueness}
We start by showing  uniqueness for %
$(L,R)\in\R^2$ with $L+R<0$.
\begin{lem} %
    \label{lem:uniqueness}
    For each $(L,R)\in\R^2$ with $L+R<0$, there is 
     at most one Gibbs state $\mu$ such that \eqref{linearlim} holds.
\end{lem}
\begin{proof}
From the shifting argument in Lemma \ref{lem:lpr}, it suffices to show that 
if $\mu,\nu\in\cG$ satisfy $L(\mu)=L(\nu)= R(\mu)=R(\nu)=-2K$ for some $K>0$,  then $\mu=\nu$.
Define $X\sim\mu$ and $Y\sim \nu$ on the same space independently and switching to $\PAR$ set $\bar X = X-\p$ and $\bar Y=Y-\p$.  
Take $\del>0$ and consider  the events
    \begin{align*}
        A_m := \left\{\bar X(t), \bar Y(t) \in [-2(K + \del)|t|, -2(K - \del)|t|],\ \forall t \in [-m, m]^{\c}\right\}\,,\qquad m\in\bbN.
    \end{align*}
    The condition in the statement implies that $\lim_{m \to \infty} \P(A_m) = 1$, and therefore for any $\eps>0$, we may take $M$ %
     large enough so that $\P(A_M) \geq 1 - \eps$.
In particular, according to the independent coupling of $(X,Y)$,
 \begin{align}\label{eq:giv1}
        \P[F] \geq 1 - \eps\,,\qquad F:=\{\bar X(\pm M), \bar Y(\pm M) \in [-2(K + \del)M, -2(K - \del)M]\}.
    \end{align}
 For a constant $C_1$ to be fixed below, take $s\geq C_1$ and $u=C_1\log s$ and consider  the events 
    \begin{align*}
        &E^+_X := \Bra{\bar X(s) \in \Box{\f{s\bar X(M)}{M} - u{\sqrt s}, \f{s\bar X(M)}{M} + u{\sqrt s}}}, \\
        &E^-_X := \Bra{\bar X(-s) \in \Box{\f{s\bar X(-M)}{M} - u{\sqrt s}, \f{s\bar X(-M)}{M} + u{\sqrt s}}}, 
    \end{align*}
    and define $E^\pm_Y$ similarly. Conditionally on the event $F$ in \eqref{eq:giv1}, switching to $\PAR,$ for a suitably large choice of the constant $C_1$ and assuming $\del\le 1/C_1$,
    an application of   \rlem{lem:closei} and a union bound shows that the event $E:=E^+_X \cap E^-_X \cap E^+_Y \cap E^-_Y$ has probability at least $1 - 1/s$     for all $M\geq C_1 s$. %
 Thus, \eqref{eq:giv1} yields the estimate
    \begin{align}\label{eq:giv2}
       \P(E) \geq (1 - \eps)(1 - 1/s)\,.
    \end{align}
Conditionally on $\bar X(\pm M), \bar X(\pm s)$, $\bar X|_{[-M, -s]}$ and $\bar X|_{[s, M]}$ are independent and similarly for $\bar Y$.
We now show that $\tv{\bar X(s^2), \bar Y(s^2)}$ is small, which by independence will imply that the total variation distance $\tv{(\bar X(s^2), \bar X(-s^2)), (\bar Y(s^2), \bar Y(-s^2))}$ is also small.

    Note that given $\bar X(s), \bar X(M)$, we have $\bar X|_{[s, M]} \sim \B^{\bar X(s), \bar X(M)}_{s, M}[-\p]$.
Define the unconditional version $X' \sim \B^{\bar X(s), \bar X(M)}_{s, M}$ and similarly define $Y'$ for $\bar Y$.
The parabola avoidance probability for $X'$ is least when $\bar X(s), \bar X(M)$ are both at their lowest endpoints.
Assuming $E\cap F$ %
these endpoints satisfy %
the inequalities
    \begin{align*}
 \bar X(M) \geq -2(K + \del)M =: R^-\,,\qquad         \bar X(s) %
 \geq -2(K + \del)s - u{\sqrt s} =: L^-\,.
    \end{align*}
    By choosing $M\geq C\,s$ where $C$ is a constant depending on $K$, one may ensure that the slope of the line joining $(s, L^-)$ and $(M, R^-)$ is at least $-2(K + 2\del)$.
Since the slope of $t \mapsto -t^2$ at $s$ is $-2s$, the distance between this line and the parabola only increases after $s$, if $s \geq K + 2\del$, which we may assume. %
Therefore the minimum distance is at $s$, where the distance is 
    \begin{align*}
        d^* := s^2 - 2s(K + \del) - u{\sqrt s} \geq \tfrac12\,s^2,
    \end{align*}
    if $s$ is large enough depending on $K$, provided $\del$ is sufficiently small depending on $K$.
Using Corollary \ref{paraavoidancecor}, this shows that the probability that $X'$ avoids the parabola is at least $1 - \Exp{-\T{s^3}}$.
Thus it suffices to consider the unconditional paths $X', Y',$ and consider $\tv{X'(s^2), Y'(s^2)}$.
Given $\bar X(s), \bar X(M)$ and $\bar Y(s), \bar Y(M)$ we have
    \begin{align*}
        X'(s^2) & \sim \Nor{\bar X(s) + \f{s^2 - s}{M - s}(\bar X(M) - \bar X(s))}{\sig^2} \\
        Y'(s^2) & \sim \Nor{\bar Y(s) + \f{s^2 - s}{M - s}(\bar Y(M) - \bar Y(s))}{\sig^2} \\
        \sig^2 & := \f{(s^2 - s)(M - s^2)}{M - s} \approx s^2 \quad \text{if $M \gg s^2$}
    \end{align*}
    Using Lemma \ref{lem:tvgau}, 
   we have
    \begin{align*}
        \tv{X'(s^2), Y'(s^2)} \leq \f{\Abs{\E[X'(s^2)] - \E[Y'(s^2)]}}{2\sig}
    \end{align*}
 We may take $M$ so large that %
$(s^2 - s)/(M - s) \leq \be s^2 / M$ where $\be > 0$ is an absolute constant.
The difference of expectations above is now bounded by
    \begin{align*}
        \Abs{\E[X'(s^2)] - \E[Y'(s^2)]} &\leq \Abs{\bar X(s) - \bar Y(s)} + \be \f{s^2}{M}\Abs{\bar X(M) - \bar Y(M)} + \be\f{s^2}{M}\Abs{\bar X(s) -\bar Y(s)} \\
                                        &\leq 4\del s + 2u{\sqrt s} + 4\be s^2 \del + 4\be s^3 \del / M
    \end{align*}
    so that the total variation distance is bounded by
    \begin{align*}
        \tv{X'(s^2), Y'(s^2)} \leq Cs\del + 2s^{-1/2}u.
    \end{align*}
    where $C$ depends only on $K$ and $s \geq C_1$ for a constant $C_1$ depending on $K$.
Observe that we have not fixed $\del$ yet except that $\del <  \del_0$ for some constant $\del_0=\del_0(K)$ depending on $K$ as in Lemma \ref{lem:closei}.
So now choose $\del$ small enough depending on $s, u$, giving that
    \begin{align*}
        \tv{X'(s^2), Y'(s^2)} \leq 3s^{-1/2}u = O(s^{-1/3})\,.
    \end{align*}
    By a triangle inequality, uniformly in $E\cap F$, %
    \begin{align*}
        \tv{\bar X(s^2), \bar Y(s^2) \,|\, \bar X(s), \bar Y(s), \bar X(M), \bar Y(M)} =  %
        O(s^{-1/3}).
    \end{align*}
 By independence, uniformly in $E\cap F$,
    \begin{align*}
        \tv{\bar X(\pm s^2),\bar Y(\pm s^2) \,|\, \bar X(\pm s),\bar  Y(\pm s), \bar X(\pm M), \bar Y(\pm M)} 
        =O(s^{-1/3}).
    \end{align*}
    Recalling \eqref{eq:giv1}, and choosing e.g.\ $\eps = s^{-1/3}$ there, 
we obtain
    \begin{align*}
        \tv{\bar X(\pm s^2), \bar Y(\pm s^2)} \leq  \P(E^{\c}) +\P(F^{\c})+ O(s^{-1/3}) = O(s^{-1/3}) .
    \end{align*}
Thus, switching back to the primal representation, there exists a coupling such that $X(\pm s^2)= Y(\pm s^2)$ with probability $1-O(s^{-1/3})$. On this favorable event we may resample with an identity coupling inside $[-s^2,s^2]$.
Therefore, for any fixed $b>0$,
sending $s\to\infty$ one has 
$\tv{X|_{[-b, b]}, Y|_{[-b, b]}} =0$. Since $b$ can be taken arbitrarily large, we obtain the desired uniqueness statement. 
\end{proof}

\begin{lem}%
    \label{lem:uniqueness2}
    For any $x\in\R$, there exists at most one Gibbs state $\mu$ %
    with limits $L=-\infty$ and $R=x$. 
\end{lem}
\begin{proof}
By the shifting argument from Lemma \ref{lem:lpr} it is sufficient to show that if $\mu, \nu\in\cG$ with $L(\mu) = L(\nu) = -\infty$ and $R(\mu) = R(\nu) = -2$, then $\mu = \nu$.
 The proof follows from the same arguments already used in Proposition \ref{ppn:fscouple} and %
 Lemma \ref{lem:uniqueness} respectively, and so we will be brief to avoid repetitions. Let $X\sim \mu$ and $Y\sim \nu$. As in the proofs above, %
 it is sufficient to  show that there exists a coupling of $X,Y$ such that $X(+s)=Y(+s)$ and $X(-s)=Y(-s)$ with probability $1-o(1)$, for arbitrarily large $s$. For negative times, we may use the arguments from  Proposition \ref{ppn:fscouple} to show that as $s\to\infty$, one can couple the paths $X,Y$ so that $X(- s)=Y(- s)$, with probability $1-o(1)$. More precisely, to couple $X(-s),Y(- s)$, one first observes that $X(0),Y(0)$ can be controlled by using Lemma \ref{lem:tailbounds}. Once the heights at zero are fixed at two possibly large constants, one can restrict to the intervals $[-m,0]$ for large $m$ and then use the argument from Proposition \ref{ppn:fscouple} which guarantees the existence of a coupling  such that $X(-s)=Y(- s)$ with probability $1-o(1)$ as $s\to\infty$.
For positive times,  one can use the same argument from Lemma \ref{lem:uniqueness}, with very minor modification, to obtain a coupling such that $X(+s)=Y(+ s)$, with probability $1-o(1)$.  We omit the details.

\end{proof}

\subsection{Proof of Theorem \ref{thmoneline} completed} \label{sec:completed}
We can now summarize our results and complete the proof of the theorem. As we mentioned at the beginning of Section \ref{sec:proofofth1} it remained to prove that 
for any $(L,R)\in\calT$ there exists a unique $\nu_{L,R}\in\cG$ such that \eqref{linearlim} holds, and that the measure is extremal, that is $\nu_{L,R}\in\cG_{\rm ext}$. 
When $L=R=-\infty$, the existence and uniqueness has been obtained in Section \ref{sec:fsuni}. When $(L, R)\in\R^2$, $L+R<0$, the existence is in Lemma \ref{lem:existence} and the uniqueness in \rlem{lem:uniqueness}.  When $L=-\infty$ and $R\in\R$, or $L\in\R$ and $L=-\infty$, we have existence by Lemma \ref{lem:existence2} and uniqueness by Lemma \ref{lem:uniqueness2}. Finally, concerning extremality, let us note that for any $(L,R)\in\calT$, the unique measure $\nu_{L,R}$ obtained above must be extremal in $\cG$. This can be seen by contradiction. Indeed, assume that $\nu_{L,R}$ is given by a nontrivial mixture of extremal Gibbs measures. By Lemma \ref{lem:deterministic} and \rlem{lem:lpr} each of the extremal measures has its own pair of constants $(L,R)\in\calT$ and by the uniqueness results established above we know that these pairs are all distinct. This forces the distribution of the limits defined in \eqref{linearlim} to be nontrivial for our measure $\nu_{L,R}$, whereas we know that $\nu_{L,R}$ has deterministic limits by construction. \\

Given the above preparation, we are now in a position to dive in to the proof of Theorem \ref{thmmanylines}.
\section{Multi-line Gibbs measures}
\label{sec5}
Here we prove Theorem \ref{thmmanylines} and Theorem \ref{thmmanylinesconf}. Most of the section is devoted to the proof of Theorem \ref{thmmanylines}. The proof of Theorem \ref{thmmanylinesconf} will be a rather quick consequence of some technical estimates established during the proof of Theorem \ref{thmmanylines} and some confinement statements that were already obtained in \cite{geomarea}. 

To guide the reader, we provide  %
a brief roadmap of the section. 
\begin{itemize}

\item
We  start with Lemma  \ref{lem:limexists123} which proves the counterpart of Lemma \ref{lem:limexists} and Lemma \ref{lem:deterministic}, showing that the top line for any Gibbs measure $\mu\in\cG_\infty$ has asymptotic left and right linear behavior with slopes $L$ and $R$ respectively on subtracting $\p.$ Further, %
we record a multi-line counterpart of Lemma \ref{lem:lpr} stating that $L+R$ must be strictly smaller than zero which turns out to follow immediately from the one-line argument and monotonicity considerations. 

\item 
We then move on to the construction of infinite measures $\mu_K$ where $L=R=-2K.$ This proceeds by considering a certain pre-limiting proxy that we term $\rho_{K,T}$, defined on $[-T,T]$, which exhibits an approximate monotonicity in $T$. The latter  allows us to conclude %
weak convergence to a measure $\mu_K\in\cG_\infty$ as $T\to\infty$. We then show that the measure $\mu_K$ constructed in this way indeed exhibits the desired asymptotic behavior. 

\item Next, to show uniqueness, we take another measure $\mu'_K\in\cG_\infty$ with the same asymptotic behavior, and consider its restriction to $[-T,T]$ that we term  $\rho'_{K,T}.$
We then show that  $\rho_{K,T}$ and  $\rho'_{K,T}$ have  the same limit as $T\to\infty$. The proof of the latter statement requires two steps. First, by our choice of the boundary conditions of $\rho_{K,T}$, it turns out that one can naturally couple $\rho'_{K,T}$ and $\rho_{K,T}$ in such a way  that with high probability the former dominates the latter. The second step, which entails most of the work, is to show that one can construct a reverse coupling such that, with high probability, the latter dominates the former, at least in a mesoscopically large domain, i.e. $[-T^c, T^c]$ for some constant $c\in(0,1)$.

\end{itemize}

The following subsection analyzes the asymptotic properties of the top path, in analogy with 
Section \ref{sec:asymp}.
\subsection{Growth at infinity for Gibbs measures.}

\newcommand\scP{\mathcal{P}}
We will first prove a statement analogous to 
Lemma \ref{lem:limexists} and Lemma \ref{lem:deterministic}.
\begin{lem} %
    \label{lem:limexists123}
   For any $\mu\in  \cG_{\infty}$, %
the following two %
limits for the top path $X^1$ exist $\mu$-almost surely as extended real random variables: %
    \begin{align*}
        L %
        := \lim_{t \to -\infty} \f{X^1(t) - t^2}{|t|}, \qquad 
        R %
         :=\lim_{t \to \infty} \f{X^1(t) - t^2}{t}.
    \end{align*}
    Further, if $\mu\in  \cG_{\infty, \rm{ext}},$ then $L$ and $R$ are deterministic.
\end{lem}
Fortunately, monotonicity will allow us to essentially reduce the problem to a one-line case for which we will reuse the arguments from the proof of \rlem{lem:limexists}.

\begin{proof}
As in the proof of Lemma \ref{lem:limexists}, we will show that $\mu$-almost surely
  \begin{align}\label{Rrbar}
\overline R:=\limsup_{t\to \infty} \,(X^1(t)-t^2)/t
=\liminf_{t\to \infty} \,(X^1(t)-t^2)/t
=:\underline R.
\end{align}
Towards this it is sufficient to establish \eqref{eq:implic2}.
As in the proof of Lemma \ref{lem:limexists}, fixing some $b \in \R, \del > 0,$ and considering the events
    \begin{align*}
        \Once_k &= \{ \exists\ t \in [k, k + 1) : (X^1(t) - t^2)/t \geq -2b \}, \quad k \in \N \\
        \All_{\ell, r} &= \{\forall\ t \in [\ell, r] : (X^1(t) - t^2)/t \geq -2(b + \del) \}, \quad \ell < r \in \R,
    \end{align*}
it suffices to show that for any $\ell, r$, we have $\Once_k \implies \All_{\ell, r}$ with high probability if $k$ is large enough. 
 Recall the definition of %
 $\tau = \tau_k$ given by
    \begin{align*}
        \tau = \sup \{t \in [k, k + 1) : (X^1(t) - t^2)/t \geq -2b \}, \quad \text{on } \Once_k,
    \end{align*}
    and {$ \tau= k$ on $\Once_k^\c$}. Note now that monotonicity implies that given $X^{1}(\tau)=-2b,$ the probability of $\All_{\ell, r}$ is higher in the multi-line case than in the single line case.
Therefore, as in  \eqref{eq:oka} one has
    \begin{align}\label{mon345}
        \P(\All_{\ell, r} \ |\  \Once_k) \geq 1 - C\Exp{-c\,\ell \del^2}.
    \end{align}
The remainder of the proof of \eqref{Rrbar} is verbatim the proof of Lemma \ref{lem:limexists}. We omit the details to avoid repetition.  
To prove the last assertion, as in Lemma \ref{lem:deterministic}, if $\mu \in \cG_{\infty, \rm{ext}}$ then $\mu $ is tail trivial. Since $L$ and $R$ are tail measurable random variables one has that the limits $L,R$ must be deterministic for all extremal Gibbs measures. 
\end{proof} %

The next lemma is analogous to Lemma \ref{lem:lpr}.
 \begin{lem}
    \label{lem:lpr1234}
    Let $\mu\in\cG_{\infty, \rm ext}$ and let $L,R$ be the constants from Lemma \ref{lem:limexists123}.
Then $(L,R)\in\calT$, with $\calT=\{(L,R)\in[-\infty,\infty)^2: \; L + R <0\}$ as in \eqref{lineartilts}.
\end{lem}
The proof of Lemma \ref{lem:lpr1234} is the same as that of  Lemma \ref{lem:lpr} along with an application of monotonicity. Indeed, the proof of Lemma \ref{lem:lpr} proceeded by showing that in the one-line case, if $(L,R) \notin \calT,$ then the marginal tail at a single point $\P(X^1(0)>t)$ remains uniformly bounded away from zero as $t$ goes to infinity, thereby contradicting tightness. In the presence of multiple lines, by monotonicity the marginal tail of the top line is at least as large as that of the single line, which allows one to reach the same contradiction.  %

\subsection{Existence} \label{thirdstep}
We now set out to prove the existence part of Theorem \ref{thmmanylines}. As in the one-line case, we first consider %
$(L,R)\in \R^2$ with $L+R<0.$ When the limits are both finite, by shifting we will take $L=R=-2K$ where $0<K <\infty.$ The arguments for the case where $L=R=-\infty$ or when only one of them is finite will involve minor modifications and will only be sketched to avoid lengthy repetitions.  

For reasons that will become apparent soon, we introduce the line ensemble $\rho_{K,T}$, defined as  the $\lambda$-tilted LE on $[-T,T]$, with infinitely many lines, with strength parameter $a=2$,  and with boundary conditions \begin{gather}
X^1(-T)= X^1(T) = T^2-2KT - T^b\\
X^j(-T)= X^j(T) = 0,\quad j=2,3,\dots
\end{gather}   
where $b\in(1/2,1)$ is a fixed constant. Actually,  $b=0.6$ will be sufficient and hence we will work with this specific choice below. %
We begin by showing %
weak convergence of $\rho_{K,T}$.

\begin{lem}\label{lem:rkt}
For every $K>0$, there exists %
$\mu_K\in\cG_{\infty}$ such that $\rho_{K,T}$ converges weakly to $\mu_K$ as $T\to\infty$.  
\end{lem}

\begin{proof}
The proof %
involves two components, tightness and convergence of finite dimensional distribution. The former has been discussed extensively in \cite{CIW19}. In fact, were it not for the presence of the top path with growing boundary conditions, tightness of $\rho_{K,T}$ would follow immediately from \cite[Th.\,1.4]{CIW19}, and it is not hard to check that the same argument applies provided that one has a confinement statement for the top path. To this end, one has that for any $\e>0,S>0$ there exists $C=C(S,\e)>0$ such that
\begin{gather}\label{eq:tightx1}
\inf_{T\ge S}\,\P[\max_{s\in[-S,S]}X^1(s)\le C]\ge 1-\e\,.
\end{gather}   
The argument in \eqref{toplinecouple456} and \eqref{toplinecouple4567} below actually yields an estimate that is more precise than \eqref{eq:tightx1}, and thus we skip the details of the proof of this. 
Once the control \eqref{eq:tightx1} is available, a minor modification of the proof of \cite[Th.\,1.4]{CIW19}  provides tightness of $\{\rho_{K,T}, T>0\}$ in $\calP^+_\infty$. Moreover, the same theorem also implies that any limit point of the sequence $\rho_{K,T}$, as $T\to\infty$ is a Gibbs measure in the sense of Definition \ref{def:laBG} and is asymptotically pinned to zero in the sense of Definition \ref{def:asympin}.

We now move on to the proof of the 
 convergence of finite dimensional distributions.
To this end, fix $m\in\bbN$, let $\bbS=(s_1,\dots,s_m)\in \bbR^m$, $\mathbb{I}=(i_1,\dots,i_m)\in\{1,\dots,n\}^m$, and let $\mathbb{T}=(t_1,\dots,t_m)\in \bbR_+^m$. Consider the event
\begin{equation}\label{eq:monpis}
E=E(\bbS,\bbI,\bbT)=\left\{\uX\in\Omega:\; X^{i_j}(s_j) > t_j\,,\; j=1,\dots,m\right\}.
\end{equation}
Thanks to the tightness of the measures $\rho_{K,T}$, standard measure theoretic arguments imply that %
it suffices to show that for each $n\in\bbN$, and for each choice of $\bbS,\bbI,\bbT$, 
$\rho_{K,T}(E)$ converges when $T\to\infty$.
This will be achieved by showing that %
for any such $E,$ the sequence of numbers  $\rho_{K,T}(E)$ is approximately monotone, in a sense made precise shortly, and hence has a limit. 
 
Let $T_i=2^i$ and let $(X_i^{1}, X_i^{2}, \ldots)$ denote the corresponding LE sampled from $\rho_{K,T_i}$.  We will now compare the marginal of $\rho_{K,T_{i+1}}$ restricted to $[-T_i, T_i]$ to $\rho_{K,T_{i}}.$ Towards this, consider the marginal boundary induced on $[-T_i,T_i]$ by the measure $\rho_{K,T_{i+1}}.$
Certainly, by definition, for all $j\ge 2$ we have $$X_{i+1}^{j}(\pm T_i)\ge 0= X_{i}^{j}(\pm T_i).$$
We will now seek to show that with high probability we also have $$X_{i+1}^{1}(\pm T_i)\ge X_{i}^{1}(\pm T_i).$$
By definition, 
\begin{align*}
X_{i+1}^{1}(\pm T_{i+1})=T_{i+1}^2-2KT_{i+1}-T_{i+1}^{0.6}, \;\;\;\qquad %
X_{i}^{1}(\pm T_{i})=T_{i}^2-2KT_{i}-T_{i}^{0.6}.
\end{align*}
At this point we exploit %
the correction term $T_{i}^{0.6}$ to obtain %
the desired monotonicity with high probability. 
Using the fact that $X_{i+1}^{1}(\cdot)$ is stochastically larger than in the case of the single line with the same starting and ending points, using Lemma \ref{lem:closei} (note the lemma was stated in $\PAR$ language), we get (plugging in $n=T_{i+1,}$ $y^{\pm}=-2KT_{i+1}-T^{0.6}_{i+1}$ and $s=T_i$)
    \begin{align*}
        \P[X_{i+1}^1(\pm T_i) \in T_i^2-2KT_i-\frac{T^{0.6}_{i+1}}{2}+[-uT_i^{1/2}, \infty)] \geq 1 - C\Exp{-cu}\,.
    \end{align*}
    Thus plugging in $u=0.1T_i^{0.1}$, and using $\frac12{T^{0.6}_{i+1}}+uT_i^{1/2}\le T_i^{0.6}$,  it follows that $$\P(X_{i+1}^1(\pm T_i)\ge X_{i}^1(\pm T_i))\ge 1-C\exp(-cT_i^{0.1})=:1-a_i.$$ 
    Denote the event on the LHS above by $A_i.$ Thus, on $A_i$ we have the ordering of the boundary data $X_{i+1}^{j}(\pm T_i)\ge X_{i}^j(\pm T_i)$, for all $j\ge 1$.  
 Note that the $T^{0.6}$ term allowed us to conclude that it would beat the diffusive fluctuation of the Brownian bridge allowing the desired monotonicity relations between the boundary conditions.    
Now, conditioned on $A_i,$ %
using the monotone coupling, it follows that 
    $\rho_{K,T_{i+1}}(E\mid A_i)\ge \rho_{K,T_{i}}(E),$
    and hence by averaging, %
    $$\rho_{K,T_{i}}(E)\le \rho_{K,T_{i+1}}(E)+a_i.$$ Letting $p_i:=\rho_{K,T_{i}}(E)$ we infer %
    the approximate monotonicity relation $p_{i+1}\ge p_{i}-a_i.$ Since $\sum_{i\ge 1} a_i < \infty,$ it immediately follows that the $p_{i}\,$s have a limit as $i\to\infty$.
 To prove that the entire sequence $\rho_{K, T}(E)$ converges as $T\to\infty$, it remains to note that for any $T \in [2^{i+1},2^{i+2}],$ by the same argument as above,  $$\rho_{K,T_{i}}(E)\le \rho_{K,T}(E)+a_i$$
which is enough to conclude the proof.
\end{proof}
Having constructed $\mu_{K}$ we next show that it indeed has the desired asymptotic properties.

\begin{lem}\label{correct behavior} For any $K>0$, $\mu_K$-almost surely,
$$\lim_{t\to \pm\infty} \frac{X^1(t)-t^2}{|t|}=-2K
.$$
\end{lem}

\begin{proof}
Let us first observe that under $\rho_{K,T}$ the lines indexed $2$ onwards  are confined, uniformly in $T$. By monotonicity, we can ignore  the top line, and thus $(X^2,X^3,\dots)$ is stochastically dominated by the $\lambda$-tilted line ensemble with zero boundary and overall strength $2\lambda$. Therefore,  \cite[Theorem 3.1]{geomarea} implies  that for suitable constants $C,c>0$, 
\begin{equation}\label{strbound}
\P\left[\max_{u\in[-S,S]}X^{2}(u) \ge y \right]\le CS\exp(-cy^{3/2}),
\end{equation}
for all $S\geq 1$ and $y>0$. These estimates are uniform in $T$, and thus they carry over to the limiting measure $\mu_K$. Moreover, note that by construction, the estimate is uniform in $K$ as well. 

Using \eqref{strbound}, and a union bound over the sequence of $S_i=s2^i$,  it follows that for any $\delta>0$ and  for any $s$ greater than a fixed constant, with probability at least $1-\exp(-s^{\delta}),$ for all $t\geq 0,$ one has 
\begin{equation}\label{floorbound}
X^2(t)\le t^{\delta}+s^{\delta}.
\end{equation}
Taking $\delta$ small enough, this will essentially allow us to ignore $X^2$ and treat $X^1$ as a single line allowing usage of one-line estimates.  Switching to $\PAR,$ we have that under $\rho_{K,T}$, conditionally on $X^2$, the random path $Z^1:=-\p+X^1$ has the law of the Brownian bridge $\B^{-2KT-T^{0.6}}_T$ conditioned to avoid $-\p+X^2$.
We claim that \begin{equation}\label{toplinecouple456}
\big|Z^1(\pm s)+2Ks\big|\le 
s^{2\delta}{\sqrt s},
\end{equation}
 with probability at least $1-\exp(-s^{c})$, for some constant $c>0$, uniformly in $T$. 
The lower bound on $Z^1(\pm s)+2Ks$, by monotonicity, can be reduced to the single line case, and therefore follows from  Lemma \ref{lem:closei}. Hence we will focus on the upper bound. 
For this, using the same argument as in the tail bounds from Lemma \ref{lem:tailbounds}, we see that $Z^1(0) \le s^{2\delta}$ with probability at least $1-\exp(-s^\delta)$. This uses Lemma \ref{lem:brownianav2} which shows that starting from $s^{2\delta},$ a Brownian bridge to $-2KT-T^{0.6}$ avoids $-\p+X^2$ with probability at least $1-\exp(-s^\delta)$, provided $X^2$ satisfies \eqref{floorbound} and $\delta$ is small enough. Hence it suffices to prove \eqref{toplinecouple456} for the Brownian bridge in which case it follows from standard Brownian estimates. 
Putting everything together, and switching back to the primal representation, we get, for any  $s>0$,
\begin{equation}\label{toplinecouple4567}
\big|X^1(\pm s)-s^2+2Ks\big|\le s^{2\delta}{\sqrt s},
\end{equation}
with probability $1-C\exp(-s^c)$, for some constants $C,c>0$. This estimates holds for the top path $X^1$ under $\rho_{K,T}$, uniformly in $T$, and thus continues to hold for $X^1$ under $\mu_K$. This, along with an application of the Borel-Cantelli lemma %
finishes the proof. 
\end{proof}

\subsection{The key technical coming down estimate}
Before addressing the problem of uniqueness, we present a  key technical device that will be needed. 
In words, this states that if all the lines start at $T^{2}-2KT$ at $\pm T,$ there exists a macroscopic domain $[-\eta T, \eta T]$ with $\eta=\eta(\lambda)$, such that by then the second line has descended to a height which is an arbitrarily  small polynomial in $T$.
\begin{ppn}\label{ppn:refinedbnd}  Let $\uX$ be the $\lambda$-tilted $n$-line ensemble on $[-T,T]$ with boundary conditions $x_T=T^2 - 2KT$ for all lines, that is $\uX
 \sim\L_{n;T}^{\ux_T}$ where $\ux_T\equiv x_T$. For any $\delta>0$, 
   there exist  constants $\eta=\eta(\lambda)\in(0,1)$, and $C,c>0$ depending on $\lambda$ and $\delta$, %
   such that, with $T_1=\eta T$, uniformly in $n\in\bbN$ and $K>0$,  %
    \begin{gather}
      \label{lowerline1}    \P[X^{2}(\pm T_1) \le {T_1^{\del}}]\geq 1 -C\Exp{-T^{c}}.
    \end{gather}
\end{ppn}
The proof of Proposition \ref{ppn:refinedbnd} 
constitutes a significant part of the technical component of this work. It relies on an elaborate scheme providing control on the $k$-th path by induction over $k$. 
This is accomplished in Section \ref{sec:inductive}. In the present section we  assume its validity to finish the proof of Theorem \ref{thmmanylines} and Theorem \ref{thmmanylinesconf}.

\subsection{Uniqueness}
It will be useful to first obtain confinement estimates for the second line onwards that apply to any Gibbs state, and hence we begin by proving 
the following statement. %
\begin{lem}\label{confine345} 
The following applies to any Gibbs state $\mu'_K\in\cG_\infty$ whose top path has left and right limits $L=R=-2K$, with $K>0$. 
There exists a constant $C$ independent of $K$ such that for all integers $ k\ge 1$, for all $s\in\bbR$, if $\uX\sim \mu'_K$,
then 
\begin{equation}\label{eq:tight3}
\bbE\left[X^{k+1}(s)\right]\le C\lambda^{-k/3}, 
\end{equation}
and such that, for all $S>0$ and all $k\ge 1$
\begin{equation}\label{eq:tight3a}
\sup_{s\in\R}\,\bbE\left[\max_{u\in[-S,S]}X^{k+1}(s+u) \right]\le C\lambda^{-k/3}[1+\log(1+S\lambda^{2k/3})].
\end{equation}
Moreover, there exist constants $c,\alpha>0$ independent of $K$ such that for all $k\ge 1$,
\begin{equation}\label{eq:stretchedexp10002}
\P\left[\max_{s\in [-S,S]}X^{k+1}(s)> t\lambda^{-k/3}\right]\le CS \exp{\left(-c\,t^{\alpha}\right)}\,, \qquad \forall\, S\ge 1,\quad \forall \,t>0\,.
\end{equation}
\end{lem}
In the language introduced in \cite{geomarea} we say that the the paths indexed by $j\geq 2$ are uniformly confined, see \cite[Definition 5.1]{geomarea}. In fact, the arguments from \cite[Theorem 3.1]{geomarea} can be used to upgrade the exponent $\alpha>0$ to $3/2$ but the stated weaker bound will suffice for us.
\begin{proof}[Proof of Lemma \ref{confine345}]
We first invoke Proposition \ref{ppn:refinedbnd} which implies that for any $T$, $X^2(\pm T)$ is bounded by $T^{\delta}$ with arbitrarily large probability. To see this, note that since $L=R=-2K,$ for any $\e>0,$ for all large $T$, $X^{1}(\pm T)\le T^2-KT$ with probability at least $1-\e.$ Moreover, since $\mu'_K$ is asymptotically pinned to zero, see Definition \ref{def:asympin},  we also have that $X^{n+1}(t)\le \e$ for all $t\in[-T,T]$ with probability at least $1-\e$, for all $n$ sufficiently large. Conditionally on these events, by monotonicity we can resample the first $n$ paths with the same boundary height $x_T:=T^2-KT$ and with a floor at $\e$, so that the uniform (in $n$) bound of Proposition \ref{ppn:refinedbnd}, applied with $\delta/2$ instead of $\delta$, then implies that with probability at least $1-3\e,$ for all large $T,$
$X^{2}(\pm \eta T)\le T^{\delta/2}+\e\le (\eta T)^\delta$ for some constant $\eta=\eta(\lambda)\in(0,1)$. Replacing $\eta T$ by $T$ finishes the argument. Moreover, by a Borel-Cantelli argument, the above reasoning can be strengthened to obtain that the events $E_T:=\{X^2(\pm T)\le T^{\delta}\}$ hold eventually almost surely, say for $T\in\N$. 
Next, %
on the event $E_T$, by monotonicity we can remove the top line $X^1$ and work with the $\lambda$-LE with strength $a=2\lambda$ and boundary data all equal to $T^\delta$ on $[-T,T]$. 
At this point a straightforward modification of the proof of Proposition 7.2 in \cite{geomarea} yields the following estimates:
there exists an absolute constant $C>0$, such that 
\begin{gather}\label{eq:tight2y}
\max_{|s|\le T/2} \E\left[X^{k+1}(s)\ind_{E_T}\right]\le C\lambda^{-k/3},\\
\E\left[\max_{|s|\le S} X^{k+1}(s) \ind_{E_T}\right]\le C\lambda^{-k/3}\log(1+\lambda^{2k/3}S). \label{eq:tight2y2}
\end{gather}
for all $S\in[0,T/2]$, and for all $k\geq 1$. Note that the parameter $s$ is restricted to the bulk $[-T/2,T/2].$ This is crucial to have enough room for the paths to come down from the original height $T^\delta$ at $\pm T$, to a bounded height. 

While Proposition 7.2 in \cite{geomarea} is stated for bounded boundary data $L$, the argument in its proof, namely Lemma 7.3 there, applied with $L=T^\delta$, is actually sufficient to guarantee the statements  \eqref{eq:tight2y}-\eqref{eq:tight2y2} for our case with boundary data $O(T^\delta)$, as soon as $\delta<1/2$. Since $T$ can be taken to be arbitrarily large, and $\ind_{E_T}\to 1$ a.s., by passing to the limit $T\to\infty$, the uniform confinement bounds in \eqref{eq:tight3} and \eqref{eq:tight3a} follow by Fatou lemma.


Further, once the estimates \eqref{eq:tight2y}-\eqref{eq:tight2y2} are available, it follows from the same argument in the proof of  Lemma 5.6 in \cite{geomarea} that 
there exist constants  $\alpha,c>0$ and $C>0$\begin{equation}\label{eq:stretchedexp1001}
\P\left[X^{k+1}(s)> \lambda^{-k/3}y\right]\le C\,e^{-c\,y^\alpha}\,, \qquad \forall\, s\in [-T/2,T/2],\quad \forall \,y\in(0,T^c)\,,
\end{equation}
for any integer $k\ge 1$ whenever $y^{10}\lambda^{-2k/3} < T/2$.
By adjusting the value of the constant $c$ we can ensure that \eqref{eq:stretchedexp1001} holds for all $y\in(0,T^c)$, for all $k\geq 1$.
Again by sending $T$ to infinity, we get \eqref{eq:stretchedexp1001} for all $y$ and $s.$  %
Finally, 
a simple union bound allows one to get \eqref{eq:stretchedexp10002} from  \eqref{eq:stretchedexp1001}, see for instance \cite[Lemma 5.3]{geomarea} where a similar argument appears. 
\end{proof}

As a consequence, we have the following multi-line analog of Lemma \ref{lem:closei}.  
\begin{lem} %
For any Gibbs state $\mu'_K\in\cG_\infty$ whose top path has  limits $L=R=-2K$, with $K>0$, 
   \label{onelinelocation}
    \begin{align*}
        \P\left(X^1(s) \in s^2-2Ks + [-s^{0.6},s^{0.6}]\right) \geq 1 - C\Exp{-s^{c}}\,,
    \end{align*}
    for all $s \geq C$, where $c,C>0$ are constants. %
\end{lem}

\begin{proof}Since the lower bound can be inherited from the one-line case in Lemma \ref{lem:closei} using monotonicity, we move on to the upper bound. Thanks to \eqref{eq:stretchedexp10002}, the bound stated in \eqref{floorbound} on $X^2$ continues to hold even for $\mu'_K,$ %
with probability at least $1-\exp(-s^{c})$.  Further, since by hypothesis on $\mu'_{K},$ the limits $R=L=-2K,$ one has that that given $\e>0,$ for all large $T,$ 
with probability at least $1-\e,$
\begin{equation}\label{eq:tcond}
(X^{1}(\pm T)-T^2)/T
\in -2K+[-\e,\e].
\end{equation}

Conditioning on $X^2$ satisfying \eqref{floorbound} and further on $X^1({\pm T})$ as in \eqref{eq:tcond}, one can now follow the arguments of Lemma \ref{lem:tailbounds} to get that with probability at least $1-\exp(-s^{\delta}),$ we have $X^{1}(0) \le s^{3\delta}.$ 
We now
condition further on $X^{1}(0)$ at a value less than $s^{3\delta},$ and switching to $\PAR,$ and letting $Z^1=-\p+X^1,$ we see that $Z^{1}$ on $[-T,0]$ is upper bounded by a Brownian Bridge between $-(2K-\e)T$ and $s^{3\delta}$ conditioned to avoid $-\p+X^2.$ However by Lemma \ref{lem:brownianav2}, the avoidance constraint is satisfied by a Brownian bridge with the same endpoints with probability at least $1-\exp(-s^{\delta}).$ Now just by Brownian bridge estimates, we get, for a suitably small constant $c>0$,   
$$\P\left(X^1(s) \in s^2-(2K-\e)s + [-s^{0.6},s^{0.6}]\right)\ge 1-\exp(-s^{c}).$$
Now sending $T \to \infty$, and then $\e \to 0$, the proof is complete.
\end{proof}

We will now prove that $\mu'_K=\mu_K$. 
The proof will proceed by analyzing LEs on the domain $[-T,T]$ and then passing to the limit. 
Using the previous lemma, it is clear that one can couple $\uX\sim \rho'_{K,T}$, where $ \rho'_{K,T}$ denotes the restriction of $\mu'_K$ to $[-T,T]$, and $\uY\sim \rho_{K,T}$ on $[-T,T]$ such that with probability at least $1-\exp(-T^c),$  we have
$X\succeq Y.$ The proof will then be based on
constructing a different coupling which reverses this order with large probability.

\subsection{Reverse coupling}\label{sec:revco}
Let $ \rho'_{K,T}$ denote the restriction of $\mu'_K$ to $[-T,T]$, where $\mu'_K\in\cG_\infty$ is any Gibbs measure whose top path has left and right limits $L=R=-2K$, with $K>0$. 
\begin{lem}\label{reversecoupling}
There exists  a constant $c>0$ and a coupling of $\uX\sim \rho'_{K,T}$ and $\uY\sim \rho_{K,T}$
such that with probability at least $1- \varphi(T)$
one has  %
\begin{gather}\label{eq:revco}
X^j(s)-Y^j(s) %
\le  \varphi(T)\,,\qquad \forall \, j=1,2,\dots %
\,,\quad \forall \,s\in[-T^c,T^c],
\end{gather}
where $\varphi(T) = \exp{\left[-(\log T)^c\right]}$. 
 \end{lem}
 As a corollary of Lemma \ref{reversecoupling} and Lemma \ref{lem:rkt}, using standard measure theoretic arguments, see e.g.\ the proof of \cite[Theorem 3.7]{geomarea},
one concludes that also $ \rho'_{K,T}$ converges to $\mu_K$ as 
$T\to\infty$,  which implies that $\mu'_K=\mu_K$, finishing the proof of uniqueness. 

\begin{proof}[Proof of Lemma \ref{reversecoupling}]
We are going to modify the argument in the proof of Lemma 7.4 in \cite{geomarea}  to handle the presence of the top line with parabolic boundary data.  {First, we show that it is sufficient to prove \eqref{eq:revco}} where $Y$ is replaced by a truncated version $\uY_k$  where all lines below the $k$-th path are removed with $k=\lfloor \left(\log T\right)^{\gamma}\rfloor$,  where $\gamma>0$ is a small constant to be fixed later. That is we let $Y_{k}$ denote the $k$-line ensemble  $[-T,T]$ with boundary condition
\begin{gather}
Y_k^1(-T)= Y_k^1(T) =T^2-2KT-T^{0.6},\\
Y^j_k(-T)= Y^j_k(T) = 0,\quad j=2,\dots,k\,,\qquad k=\lfloor \left(\log T\right)^{\gamma}\rfloor.
\end{gather}   
To see that this indeed suffices, observe that with high probability for the ensemble $\rho'_{K,T}$,  
 the $(k+1)$-th line $X^{k+1}$ does not exceed the  height $\lambda^{-k/6}$ in $[-T,T]$. 
  Namely, consider the event
\begin{equation}\label{eq:ftf1}
F_T = \left\{\max_{s\in[-T,T]}X^{k+1}(s) \le \lambda^{-k/6}
\right\}\,.
\end{equation}
Using \eqref{eq:tight2y2} and Markov's inequality, our choice of $k$ implies that for all $T$ large enough one has
\begin{equation}\label{eq:ft2}
\P(F_T)\ge 1- \lambda^{-k/7}.
\end{equation}
Since one can take the constant $c$ in the definition of $\varphi$ so that $\lambda^{-k/7}\ll \varphi(T)$, thanks to the  control on $X^{k+1}$ in \eqref{eq:ft2} and the obvious fact that $Y^{k+1}\succeq 0$, we can reduce to proving \eqref{eq:revco} only for $j\le k$ and hence by monotonicity we may work with the truncated ensemble $\uY_k.$

 Next, we observe that we can replace $\uX|_{[-T,T]}$ with law $\rho'_{K,T}$ by the paths with law $\rho'_{K,T}$ conditioned on the pinning $X^1(\pm T)= T^2-2KT+T^{0.6}$ on the top path of $\uX$ at the endpoints of the interval. Indeed, it follows from Lemma \ref{onelinelocation} and monotonicity that this is sufficient to prove \eqref{eq:revco} up to a small probability error of size $\exp(-T^c)$ for some constant $c>0$. By slight abuse of notation we continue to denote by $\uX$ this line ensemble. 
  
 Let $\hat T=T^\zeta$ where $\zeta\in(0,1)$ is a suitably small constant to be fixed later. 
 Consider the discrete time steps $s_j=-\hat T+2j-1$, $j=1,\dots,j_{\max}$,
where {$j_{\max}=\lfloor \hat T/2\rfloor$}. Similarly, let $u_j = \hat T -2j+1$, $j=1,\dots,j_{\max}$. 
Set $v:=k^b$, with $b>0$ to be fixed later, and consider the indexes 
\begin{align}\label{eq:coupp2} 
\ell_* &= \inf\{j\in\{1,\dots,j_{\max}\}:   Y_{k}^i(s_j)\ge v\lambda^{-(i-1)/3}\,,\;\forall i=2,\dots,k\}\,, \\
r_*  & = \inf\{j\in\{1,\dots,j_{\max}\}:   Y_{k}^i(u_j)\ge v\lambda^{-(i-1)/3}\,,\;\forall i=2,\dots,k\}\,.
\end{align}
Accordingly, we define the random times $\tau_\ell=s_{\ell_*}$ and $\tau_r=u_{r_*}$, and  
consider the events 
\begin{gather}\label{eq:coupp3} 
 B_T=\{\tau_\ell<-\hat T/4\,,\;\tau_r>\hat T/4 \}\,,
 \\
 D_T=\left\{\max_{s\in[-T,\tau_\ell]\cap[\tau_r,T]}X^{2}(s) \le T^\delta\right\}
\,,\quad D'_T=\left\{\max_{s\in[-T,\tau_\ell]\cap[\tau_r,T]}Y^{2}_k(s) \le T^\delta\right\}\,,\\
  G_T:=\cap_{i=2}^kG^{\,i}_T\,,
\qquad G^{\,i}_T := \left\{X^i(\tau_\ell) \le v\lambda^{-(i-1)/3}\right\}\cap \left\{X^i(\tau_r) \le v\lambda^{-(i-1)/3}\right\}\,,
\end{gather}
where $\delta \in(0,1)$ is a small constant to be chosen later. 
We will later show that for a suitable choice of the parameters involved, under the independent coupling of $\uX$ and $\uY_k$, one has 
\begin{equation}\label{eq:claimgood1}
\P(E^{\c}_T)\ll \varphi(T)\,,\quad E_T:=B_T\cap D_T\cap D'_T\cap G_T\cap F_T.
\end{equation}
For the moment we assume the validity of \eqref{eq:claimgood1} and proceed with the definition of our coupling. 
Let us recall the notation of the external sigma algebra from \eqref{eq:Bext}
\begin{gather*}
\calB_{n; \ell , r}^{\sfe} 
= \sigma\lb X^i (t ):\, \text{{either $t{\notin} (\ell , r)$ {or} $i >n$}}\rb,
\end{gather*}
and let us define the analogous sigma algebra for the ensemble $\uY_k$ by $\calS_{n; \ell , r}^{\sfe}.$ For our purposes, we  need a variant of the above which only considers the lines indexed two or greater. Thus let
\begin{gather*}
{\calB}_{n; \ell , r}^{(1),\sfe} 
= \sigma\lb X^i (t ):\, \text{{either $2\le i \le n \,\,\text{and}\,\, t{\notin} (\ell , r)$ {or} $i >n$}}\rb,
\end{gather*}
and analogously for the ensemble $\uY_k$, let the corresponding notation be $\calS_{n; \ell , r}^{(1),\sfe}.$
Note that $\tau_{\ell}, \tau_r$ are measurable with respect to $\calS_{k; \tau_\ell , \tau_r}^{(1), \sfe}$, but independent of $X$, and that $E_T$ is in the $\sig$-algebra generated by ${\calB}_{k; \tau_\ell , \tau_r}^{(1),\sfe} ,\calS_{k; \tau_\ell , \tau_r}^{(1), \sfe}$. %

We will  prove that, conditionally on %
$\calB_{k; \tau_\ell , \tau_r}^{(1),\sfe}$, $\calS_{k; \tau_\ell , \tau_r}^{(1), \sfe}$, uniformly on the event $E_T$, the total variation distance between the %
two random vectors $(X^1(\tau_\ell),X^1(\tau_r))$ and $(Y^1_k(\tau_\ell),Y^1_k(\tau_r))$ is much smaller than $\varphi(T)$. Equivalently, that there exists a coupling such  that 
\begin{equation}\label{eq:coup}
\P_{E_T}(H^{\c}_T)\ll \varphi(T)\,,\qquad H_T :=  \{X^1(\tau_\ell)= Y^1_k(\tau_\ell)\}\cap\{X^1(\tau_r)= Y^1_k(\tau_r)\},
\end{equation}
where the notation $\P_{E_T}$ denotes the supremum of the probability over all realizations of
$\calB_{k; \tau_\ell , \tau_r}^{(1),\sfe}$ and $\calS_{k; \tau_\ell , \tau_r}^{(1), \sfe}$ satisfying the event $E_T$. 
Let us for a moment assume the validity of \eqref{eq:coup} and proceed with the proof of the lemma. 

We construct our coupling of $\uX$ and $\uY_k$ as follows. Consider two independent samples of the lines $\{X^j,\,j\geq 2\}$ and the $(k-1)$-line ensemble $\{Y^j_k,\,j=2,\dots,k\}$. If the event $E_T$ does not occur, we conclude by sampling $X^1,Y^1_k$ independently. Otherwise, conditionally on $\calB_{k; \tau_\ell , \tau_r}^{(1),\sfe}$, $\calS_{k; \tau_\ell , \tau_r}^{(1), \sfe}$ we sample $(X^1(\tau_\ell),X^1(\tau_r))$ and $(Y^1_k(\tau_\ell),Y^1_k(\tau_r))$ with the optimal coupling, which satisfies the event $H_T$ with large probability by \eqref{eq:coup}. %

By definition of the events $E_T,H_T$, on $E_T\cap H_T$, at the stopping domain $[\tau_\ell,\tau_r]\supset[-\hat T/4,\hat T/4],$ the boundary values of $\uY_{k}$ are higher than or equal to the boundary values of the top $k$ lines of $X$ and the $(k+1)$-th line of $\uX$ does not exceed height $\lambda^{-k/6}$. Thus, on this event we may resample the lines  \[
\{(X^{i}(s),Y^{i}_{k}(s)),\; s\in[\tau_\ell,\tau_r]\}\] under a monotone coupling,
 which by construction guarantees that 
\begin{align}\label{eq:coupp31} 
X^i(s)\le Y^i_{k}(s) + \lambda^{-k/6}, \qquad s\in[-\hat T/4,\hat T/4],\;\;\;i=1,\dots, k,
\end{align}
with probability one. This follows from the fact that the resampling of the top $k$ lines of $X$, on the event $F_T$, using stochastic domination, can be performed with an effective floor at height $\lambda^{-k/6}$, and therefore one  can compare with the lines $Y^i_{k}$, $i=1,\dots,k$ with floor at zero, with an overall shift by $ \lambda^{-k/6}$.  
{ Note that above we crucially used the fact that 
$[\tau_\ell,\tau_r]$ is a stopping domain for $Y_k$ and independent of $X$, which allows us to use the Gibbs property to simultaneously resample $X$ and $Y_k$ on $[\tau_{\ell},\tau_r]$, given the respective boundary data.}
From \eqref{eq:coupp31}, using also $\lambda^{-k/6} \le \varphi(T)$, on the event $E_T\cap H_T$ we have the desired bound \eqref{eq:revco} by adjusting the value of the constant $c$.
This concludes the proof, assuming the validity of  \eqref{eq:claimgood1} and \eqref{eq:coup}.

We turn to the proof of  \eqref{eq:coup}.  To this end, it is convenient to introduce the following 
 events, representing a stronger version of the events $D_T,D'_T$ defined above:
\begin{gather*}
A_T = \left\{\max_{s\in[-T,T]}X^{2}(s) \le T^\delta
\right\}\,, \qquad A'_T = \left\{\max_{s\in[-T,T]}Y^{2}_k(s) \le T^\delta
\right\}\,.
\end{gather*}
Suppose that all lines $X^j, j\geq 2$ are sampled, %
and, independently sample the lines $Y^j_k, j\ge 2$. We claim that conditionally on these realizations, uniformly on the event 
\[
E'_T:=B_T\cap A_T\cap A'_T,
\] 
we can couple the two vectors $(X^1(\tau_\ell),X^1(\tau_r))$ and $(Y^1_k(\tau_\ell),Y^1_k(\tau_r))$ in such a way that the event $H_T$ has probability 
 at least $1-T^{-c}$. 
Note that this is sufficient to prove \eqref{eq:coup}. Indeed, 
using Lemma \ref{confine345}, see \eqref{eq:stretchedexp10002}, and monotonicity, we know that 
\[\P(A'_T) \ge \P( A_T)\ge 1- \exp( -T^c)\,,\]
and moreover one clearly has $A_T\cap A'_T\subset D_T\cap D'_T$. 
Since $\exp( -T^c)\le T^{-c}\ll \varphi(T)$, the TV bound %
claimed in \eqref{eq:coup} is then obtained from the above claim about coupling on the event $E'_T$ by using the fact that TV distance can only decrease on further averaging.  

Thus we can now focus on proving the asserted estimate on the event $E'_T$. Here all lines $X^j$, $j\ge 2$ and $Y^j_k$, $2\le j\le k$ are sampled and satisfy $E'_T$, the top paths are pinned at \[X^1(\pm T)= T^2-2KT+T^{0.6}, \qquad Y^1_k(\pm T)= T^2-2KT-T^{0.6},\] and we want a coupling of the top paths at times of order $\widehat T=T^\zeta$.  
To prove the claim, we will actually prove the stronger fact  that 
uniformly on the given realization of the lower lines as above, 
the law of $X^{1}$ on $[-T, \tau_{\ell}] \cup [\tau_{r}, T]$  is very close in TV distance to the parabola $\p$ plus two independent Brownian bridges on $[-T, \tau_{\ell}]$ and $[ \tau_{r}, T]$ respectively, where the values at $\pm T$ are $T^2-2KT+T^{0.6}$ while the values at $\tau\in\{\tau_{\ell}, \tau_r\}$ are two independent Gaussians $(W_1,W_2),$ of mean $\tau^2-2K|\tau|$ up to a small polynomial correction in $T$ and with variance $|\tau |(1-\frac{|\tau|}{T}).$
The same argument also shows that the above conditional TV bound holds for the law of $Y_k^{1}$ on $[-T, \tau_{\ell}] \cup [\tau_{r}, T]$. %
The small discrepancies in the boundary conditions will allow us to obtain that $Y_k^{1}(\tau)=X^{1}(\tau)$ at $\tau\in\{\tau_{\ell}, \tau_r\}$ with high probability.

To prove this, we  condition on $X^1(0)$ as well, which makes $X^1(\tau_\ell),X^1(\tau_r)$ independent.
Since we are on the event $A_T,$ we can ignore $X^2$ and raise the floor to $T^\delta$ instead, and apply Lemma \ref{lem:tailbounds}  to conclude that up to a probability of $\exp(-T^{c}),$ one has $X^1(0) \le T^{3\delta}$.
We will now switch to $\PAR,$ and let $Z^1=-\p + X^1$ denote the corresponding Brownian bridge conditioned to avoid $-\p+X^2$. Thus we may assume $Z^1(0) \le T^{3\delta}.$
Given the above, by arguments as in Lemma \ref{onelinelocation}, 
\begin{equation}\label{toplinecouple}
\big|Z^1(\pm \kappa)+2K\kappa)\big|\le \kappa^{1/2}T^{2\delta},
\end{equation}
where $\kappa= T^{10\delta}$ with probability at least $1-\exp(-T^{c}).$ Therefore,  $Z^1\mid_{[\kappa, T]}$ is $\exp(-T^{c})$ close in total variation norm to a Brownian bridge with the same endpoints, since the constraint of avoiding $-\p+X^2$ is satisfied with probability $1-\exp(-T^{c})$ by Lemma \ref{lem:brownianav2}. On the event $B_T$ we have $\tau_r=\T{T^\zeta}$ and $|\tau_\ell|=\T{T^\zeta}$. This in particular implies that at $\tau \in \{\tau_\ell,\tau_r\}$, conditionally on $Z^1(0) \le T^{3\delta}$, the random variables $Z^1(\tau_\ell),Z^1(\tau_r)$, up to a TV error of $\exp(-T^{c})$, are independent Gaussian variables with variance $(|\tau|-\kappa)(1-\frac{|\tau|-\kappa}{T-\kappa})$ and means $-2K |\tau| + O(T^{7\delta})+O(\frac{T^{0.6}T^\zeta}{T}).$ The last two error terms appear due to the conditioning on the values at $\kappa$ as well as the discrepancy of the starting points at $\pm T.$  Note that the exact same conclusion holds for $Y^1_k$ as well. In conclusion, the triangle inequality and Lemma \ref{lem:tvgau}, show that as long as, say,  $7\delta <\zeta/4 <0.1,$  we can couple $X^{1}(\tau)$ and $Y^{1} (\tau)$ with failure probability $O(T^{-c}).$ This ends the proof of  \eqref{eq:coup}.

The last bit of the proof concerns the estimate \eqref{eq:claimgood1}.
Since $D_T\cap D'_T\supset A_T\cap A'_T$ we already know that $D_T\cap D'_T$ occurs with probability at least $1- \exp(-T^c)$. Similarly, $F_T$ has been already dealt with in \eqref{eq:ft2}. It remains to show that $B_T\cap G_T$ has large probability. 

Using the independence of the random variables $(\tau_\ell,\tau_r)$ and the line ensemble $\uX$, by Lemma \ref{confine345}, see \eqref{eq:stretchedexp10002}, and a union bound one has
 \begin{align}\label{eq:couppe4} 
\P(G_T^{\c}) \le %
2 k \exp(-ck^{b\alpha }).
\end{align}
Taking e.g.\ $b=1/\alpha$ and choosing the constant $\gamma $ in the definition of $k$ sufficiently small,
\[
\P(G_T^{\c}) \ll \varphi(T).
\]
It remains to estimate $\P( B_T^{\c}) $.
We note that this event concerns the $k$-line ensemble $\uY_k$ in the interval $[-\hat T, \hat T]$. 
As in the proof of Proposition \ref{ppn:fscouple},
we will use that $\uY_k$ is stochastically higher than the field obtained by pinning at zero all lines $Y^{j}_k$, $j=2,\dots,k$ at the endpoints of the intervals $I_j$ defined by
\[
I_j = [-\hat T+2(j-1), -\hat T+ 2j]\,,\quad j=1,\dots,j_{\max}.
\] 
The intervals have length $2$ and the points $s_j$ are the midpoints of $I_j$. We call $\uZ_k$ the new $k$-line ensemble obtained in this way.  
Let $\tau'_\ell$ be the random time defined as $\tau_\ell$ with $\uZ_k$ in place of $\uY_k$. Thus by monotonicity one has $\P(\tau_\ell<-\hat T/4)\ge\P(\tau'_\ell<-\hat T/4)$. We want to prove that 
\begin{align}\label{eq:tota} 
\P(\tau'_\ell<-\hat T/4) \ge 1 - \exp({- T^c})\,.
\end{align}
By symmetry this would imply $\P(B_T^{\c}) \le 2\exp({- T^c}) \ll \varphi(T)$ and it is thus sufficient to end the proof. 
We will first ensure that the first curve remains high enough on the interval $[-\hat T,-\hat T/4]$ and does not interact with the bottom curves. Towards this we define the event 
\begin{equation}\label{eq:fat1}
V_T = \left\{\min_{s\in[-T,-\hat T/4]}Z^{1}(s) \ge \hat T
\right\}\,. %
\end{equation}
By stochastic domination, we can ignore the bottom lines and use a one-line estimate, and by Lemma \ref{onelinelocation} we can restrict to the case where  $Z^{1}(s)$ is lower bounded on $[-T,0]$ by $\p+$ a Brownian bridge starting and ending at $-2KT-T^{0.6}$ and $0$ respectively.
Straightforward Brownian estimates then show that
\begin{equation}\label{eq:fat2}
\P(V_T)\geq 1-\exp({- T^c})\,.
\end{equation}
Consider now only the intervals $I_j$, $j=1,\dots,j'_{\max}$, where $j'_{\max} =\lfloor \hat T/4\rfloor$, such that  $I_T:=\cup_{j=1}^{j'_{\max}}I_j\subset [-\hat T,-\hat T/2]$. 
On the event $V_T$ by monotonicity we may replace the lines $Z^i_k(s)$, $s\in I_T$, $i=2,\dots,k$ by a $(k-1)$-line ensemble $\bar Z_k=\{\bar Z^i_k,\; i=2,\dots,k\}$, on the interval $I_T$, obtained by imposing a ceiling at height $\hat T$.  
Now the line ensemble $\bar Z_k$ is a product over the intervals $I_j$. Tail estimates for such an ensemble, albeit without the ceiling appear in the proof of Lemma 6.2 of \cite{geomarea}. 

Nonetheless, the argument works as is, even in presence of the rather high ceiling and yields that the probability $p_k$ of the event $Y_{k}^i(s_j)\ge v\lambda^{-(i-1)/3}\,,\;\forall i=2,\dots,k$, for any given $j$ satisfies 
\[
p_k \geq e^{-kv^2}.
\] Since $\tau'_\ell\ge -\hat T/4$ implies that these events fail for all $j=1,\dots,j'_{\rm max}$, the product structure shows that 
 \begin{equation}\label{eq:fat3}
\P[\tau'_\ell\ge-\hat T/4]\le (1-p_k)^{ j'_{\max}}\leq 
\exp{\left(-c\,\hat Tp_k \right)}\leq  \exp{\left(-c\,\hat T e^{-k^{1+2b}}\right)}\,,
\end{equation}
where we use $v=k^b$ and $j'_{\max}\sim \hat T/4$. Taking as above $b=1/\alpha$ and taking the constant $\gamma $ in the definition of $k$ sufficiently small, for some new constant $c>0$, and $T$ large enough one has \eqref{eq:tota}, which ends the proof. 
\end{proof}

\subsection{Proof of Theorem \ref{thmmanylines}}
Collecting the results presented so far in this section  allows us to end the proof of Theorem \ref{thmmanylines}.
Indeed, the existence and uniqueness of the Gibbs measure $\mu_K\in\cG_\infty$ with left and right limit $-2K$ for the top path, has been established in the case $K\in(0,\infty)$. As usual, by shifting, this also shows existence and uniqueness of the measures $\nu_{\infty,L,R}\in\cG_\infty$ with left and right limit $L,R$ respectively, for all $(L,R)\in\R^2$ with $L+R<0$. As in the one-line case, uniqueness also implies extremality. Thus, it remains only to check existence and uniqueness in the cases where one or both limits are equal to $-\infty$. %
Quick adaptations of the above arguments together with results from \cite{geomarea} allow us to treat these remaining cases, as we briefly discuss below. 

\bigskip

\noindent
$\bullet$ \textbf{$L=R= -\infty.$}
We only have to address the uniqueness part here, since the zero boundary measure $\mu^0$  is already known to be a Gibbs measure measure with the requisite asymptotic properties, see e.g.\ \cite[Theorem 1.4]{CIW19}. For the uniqueness we will recycle some of the argument  presented above.  Let $\uX\sim\nu\in\cG_\infty$ denote a $\lambda$-tilted LE with left and right limit for the top path equal to $-\infty$. 
By monotonicity and Lemma \ref{confine345},
we have that 
$\uX':=\{X^j, \,j\ge 2\}$ is 
uniformly confined in the sense of \cite[Definition 5.1]{geomarea}. 
Thus, for any fixed $T>0$, 
on the interval $[-10 T, 10T]$ we can pretend that $X^{2}(s)\le T^\delta$, for any fixed $\delta>0$. Now by raising the floor to $T^\delta,$ the  analogous argument in the one-line case of Proposition \ref{ppn:fscouple} can be followed verbatim to show that with high probability on $[-T/10,T/10]$ all the lines, not just the second line onwards, are below $T^{2\delta}.$ This then by an application of \cite[Proposition 7.2]{geomarea} as explained in the proof of Lemma \ref{confine345} above, allows us to argue that the entire LE with law $\nu$ is uniformly confined in the sense of \cite[Definition 5.1]{geomarea}. Thus, by appealing to the uniqueness established in \cite[Theorem 3.7]{geomarea}, one has  $\nu=\mu^0$.
\\

\noindent
$\bullet$ \textbf{$L= -\infty, R=x.$}
As usual, by shifting one can reduce to the case $L= -\infty, R=-2.$
We proceed as above except that we have to be somewhat mindful of the different behaviors on the positive and the negative sides.  For the construction part we argue as follows. 
We define the sequence of measures $\rho_{1, \infty, T}$ which is an LE on $[-T,T]$ with boundary conditions  $X^i(T)=0$ for all $i\ge 1$,  $X^{j}(-T)=0$ for all $j\ge 2$ and $X^1(-T)=T^{2}-2T-T^{0.6}.$ 
The same arguments as in Lemma \ref{lem:rkt} then show that the sequence of measure $\rho_{1, \infty, T}$ is approximately monotone and hence weakly converges to a limiting measure, say, $\mu_{1,\infty}.$ 
By monotonicity, Proposition \ref{ppn:refinedbnd} shows that the second line of this limiting ensemble is bounded by $T^\delta$ on $[-\eta T, \eta T]$ for some $\eta=\eta(\lambda)$  with high probability. We may then
apply the one-point tail estimate of Lemma \ref{lem:tailbounds} at the origin with the parabola $-\p$ replaced by the shift $-\p+T^\delta$. This allows us  to separate the left and right parts of the top path,   and to apply the previous arguments  on each side separately, to prove the desired asymptotic behavior. This settles the existence.
Finally, to argue for uniqueness, for any other $\mu$ with the same asymptotic properties, pick $T$ such that with probability, say, $1-\delta,$ one has $(X^{1}(T)-T^{2})/T\in -2+[-\delta,\delta]$ and $(X^{1}(-T)-T^{2})/T\le -K$ for some large $K.$ 
Again the arguments above imply that second curve onwards is uniformly confined on both sides. Finally, the origin has a tight distribution conditioning on which and arguing as in the $L=R=-\infty$ case implies that all the lines of $\mu$ are uniformly confined on $(-\infty,0]$ while the same is true only for $X^2$ onwards on the right side $[0,\infty)$. Applying the argument in Lemma \ref{onelinelocation}, then proves that under $\mu$, $X^{1}(T)\in T^2-2T+[-T^{0.6}, T^{0.6}]$ with high probability.  Thus, conditioned on such boundary data at $-T$ and $T$, clearly it dominates $\rho_{1,\infty, T}.$ On the other hand the exact same argument as in Lemma \ref{reversecoupling}, allows us to reverse couple on a meso-scopic interval $[-T^c,T^c]$ for some $c>0$,  finishing the proof.

\subsection{Proof of Theorem \ref{thmmanylinesconf}}
Suppose $\mu\in\cG_{\infty,\rm ext}$. Then by Theorem \ref{thmmanylines} we may assume that the top path $X^1$ under $\mu$ has left and right limits $(L,R)\in\calT$. We may further suppose that $(L,R)\in\R^2$ with $L+R<0$, since the degenerate cases can be obtained from this non-degenerate one by monotonicity considerations. Now, by shifting we may further suppose that $L=R=-2K$ for some $K>0$. At this point, by Theorem \ref{thmmanylines} we know that $\mu=\mu_K$ is the weak limit of the LE $\rho_{K,T}$ defined in Lemma \ref{lem:rkt}. Upon ignoring the top line of $\rho_{K,T}$ and passing to the limit $T\to\infty$ we see that the second line onwards of the LE $\mu$ is stochoastically dominated by the zero boundary LE $\mu^0$, that is if $\uX\sim \mu$ and $\uY\sim \mu^0$ we can couple them in such a way that  $X^{k+1}\preceq Y^k$ for all $k\geq 1$. Therefore, the confinement estimates stated in  Theorem \ref{thmmanylinesconf} are an immediate consequence of the estimates established in \cite{geomarea} for the zero boundary LE $\mu^0$, see Corollary 2.8 and Theorem 3.1 there.

\section{Inductive framework}
\label{sec:inductive}
 \def\fl{\mathsf{Floor}}
    \def\sfl{\mathsf{\widehat{Floor}}}
    \def\hq{\mathsf{\widehat{Path}}}
    \def\q{\mathsf{Path}}
    \def\bxi{\bar{\xi}}
     \def\Err{\mathsf{Err}}
The goal of this section is to prove Proposition \ref{ppn:refinedbnd}.
We begin by recalling some of the elements of the proof idea.
The key observation being that, in the one-line case, a Gibbs measure with parameter $a>0$ growing at infinity must necessarily, up to a linear correction term, exhibit \emph{exactly} the parabolic growth $t\to \frac{a}{2}t^2$. 
The boundary conditions $x_T=T^2-KT$ imposed on the LE $\uX
 \sim\L_{n;T}^{\ux_T}$ appearing in  Proposition \ref{ppn:refinedbnd}, force the second line to grow at most like $T^2$, which by the previous comment, falls short significantly of competing with the area tilt factor, causing them to descend rapidly. 
 
 Roughly speaking, this is the mechanism leading to the strong confinement properties for 
 the second line onwards.
 Note however that rigorously establishing this picture involves some work since the one-line results do not directly apply to any line in the multiline ensemble which feels the push up from the line below.  To address this we revise and strengthen our one-line results to accommodate the presence of certain floors. This allows us to set up an inductive framework: we prove an upper bound on the $k^{th}$ line given an upper bound on the $(k+1)$-th line, and the upper bound of the $k^{th}$ curve serves as a floor in the one-line set up to prove an upper bound on the $(k-1)$-th  curve, and so on.

The proof of this is somewhat technical and has two broad parts. One part considers paths with a large index, referred to as heavy paths. %
For heavy paths, by adapting some ideas from \cite{geomarea} and by using the coming down estimates established in Lemma \ref{lem:doublesidedcomingdown} it is not too difficult to show that they descend rapidly to the desired height. 
The paths with smaller index, referred to as light paths, require a more delicate inductive argument, which constitutes the core of the present section.

    \subsection{Heavy vs. light paths}\label{heavypaths}
    The threshold between heavy and light will not be very delicate and we do this by defining $n_0=n_0(T)$ to be the largest number $k$ such that $T^{-15} \leq T\lam^{-(k - 1)/2}$. We call heavy the paths with index $k\ge n_0+1$, and light the paths with index $k\le n_0$. The choice $T^{-15}$ is rather arbitrary and any inverse polynomial would have sufficed. 
        We also stipulate the following convention for the rest of this section. Given events $E_T$ indexed by the parameter $T>0$, we say that $E_T$ occurs {\em with high probability} (w.h.p.) if
 for some constant $c>0$,  for all sufficiently large $T$,
  \begin{align}\label{eq:whp}
\P[E_T]\ge 1-  \Exp{-T^c}.      
    \end{align}
    \begin{lem}[Heavy paths]\label{lem:crudebnd}
        Let $\uX
 \sim\L_{n;T}^{\ux_T}$, where $\ux_T\equiv x_T=T^2-2KT$. %
      For any $\delta>0$, %
      uniformly in $n>n_0(T)$, 
        \begin{gather}
          \label{lowerline}    X^{n_0 + 1}(t) \leq T^\del\,, \;\text{ for all } t \in {[-T+T^{-10},T-T^{-10}]}, \qquad {\rm w.h.p.} 
           \qquad \end{gather}
\end{lem}
 The proof of Lemma \ref{lem:crudebnd} is given in Section \ref{sec:pfheavy}. Next, we  record the statement concerning  light paths. 
      Here we take $\uX
 \sim\L_{n_0;T}^{\ux'_T}$ to be a system of $n_{0}$  lines on the interval $[-T, T]$, referred to as the light paths, with boundary conditions $\ux'_T$ slightly higher than $x_T=T^2 - 2KT$.  
    We use an inductive framework to show that there is a subdomain $[-\eta T,\eta T]$ %
    such that at its boundary all but the top path $X^1$ come down to a height of size $O(T^\del)$. 
      We introduce %
      some notation. Though they might seem somewhat ill motivated at this point, they will be useful in stating our results. 
      The reader may also wish to refer to Figure \ref{fig:inductive-step} for an illustration of these quantities.
     
    \begin{itemize}
    \item $\p_k$ will denote the parabola $t \mapsto \lam^{k - 1} t^2$.
    \item The starting location of the $k${th} line will be $(x'_T)_k:=S_k$, where 
     \begin{align}\label{startinglocation}
     S_k := T^2 - 2KT + 4(n_{0} - k + 1) T^\del.
     \end{align}
    \item $\xi_k$ is the $x$-tangency location %
    from the point $(T, -\p_k(T) + S_k%
    )$ to $-\p_k$, and we define 
    \[\bxi_k = \tfrac{1}{2}(\xi_k + \xi_{k + 1}).
    \]
    \item %
    $\fl_k,\, k < n_{0}$ is the path defined as follows: 
          \begin{align}\label{shiftedfloor123}
            \fl_k(t) := [4(n_0 - k) + 2] T^\del +
            \begin{cases}
                                                          \ell_1(t), & t \leq -\bxi_k,             \\
                                                          -\p_k(t),  & -\bxi_k \leq t \leq \bxi_k, \\
                                                          \ell_2(t), & \bxi_k \leq t,
                                                        \end{cases}
          \end{align}
          where the straight lines $\ell_1, \ell_2$ are defined by  
          \begin{align*}
            \ell_1(t): \text{ line joining } & (-T, -\p_k(-T) + T^2 - 2KT) \text{ to } (-\bxi_k, -\p_k(-\bxi_k)), \\
            \ell_2(t): \text{ line  joining } & (T, -\p_k(T) + T^2 - 2KT) \text{ to } (\bxi_k, -\p_k(\bxi_k)).
          \end{align*}
    {Finally, let $\fl_{n_{0}} := -\p_{n_{0}}$. }

    \end{itemize}

    The path $\fl_k$ will play the role of an effective floor for the $k$-th path in the \PAR \;representation, that is $\fl_k + \p_k$ will act as the floor for the path $X^k$.  
   
   Before we state our results, we record some facts about %
   $\xi_{k}$ and $\bxi_{k}$. Note that the $x$-tangency location from the point $(T, R)$ to $-\p_k$  is decreasing in $R$. Since
    \begin{align}
      S_k \geq T^2 - 2KT = \g T^2, \quad \text{where } \; %
      \g := 1 - 2K/T \label{eq:gdef},
    \end{align}
    $\xi_k$ is upper bounded by the tangency location from $(T, -\p_k(T) + \g T^2)$. Therefore by \rlem{lem:tangency}
    \begin{align}
      \xi_k \leq T \left(1 - \sqrt{\g\lam^{-(k - 1)}}\right)
      \label{eq:xiubnd}
    \end{align}
    Also note that since $n_0=O(\log T)$, for sufficiently large $T$ we have that $S_k \leq T^2$, so that
    \begin{align}
      \xi_k \geq T\left(1 - \sqrt{\lam^{-(k - 1)}}\right).
      \label{eq:xilbnd}
    \end{align}
    A straightforward algebra then shows that, for $\be := \f{1}{2}(1 + \sqrt{1/\lam}) < 1$,
    \begin{align}
      T\Rnd{1 - \be \sqrt{\lam^{-(k - 1)}}} \leq \bxi_k \leq T\Rnd{1 - \be \sqrt{\g \lam^{-(k - 1)}}}.
      \label{eq:bxibnd}
    \end{align}

With the above notational apparatus, we now state the main result in the analysis of the light paths. 
    
            \begin{lem}[Light paths]
        \label{lem:ensinduclem}
        Let $\uX = (X^{1}, \ldots, X^{n_{0}})\sim\L_{n_0;T}^{\ux'_T}$, with $(x'_T)_k=S_k$, be the system of light paths defined above. 
        Then %
        \begin{align}
          X^{k} &\preceq \fl_{k - 1} + \p_{k - 1} \,,\quad \forall k=2,\dots,n_0\,,\qquad {\rm w.h.p.} \label{eq:ubelowfl}
        \end{align}
    \end{lem}

Let us now show that Proposition \ref{ppn:refinedbnd} follows directly from Lemma \ref{lem:crudebnd} and  Lemma \ref{lem:ensinduclem}.
The proof of the two lemmas is given in Sections   \ref{sec:pfheavy} and \ref{sec:pflight} respectively.
 
\subsection{Proof of Proposition \ref{ppn:refinedbnd}}
          
          Let $\uX \sim\L_{n;T}^{\ux_T}$, where $\ux_T\equiv x_T=T^2-2KT$. 
          Raising the floor at height $x_T$, simple monotonicity considerations and an application of  the upper tail estimates on the maximum of an infinite LE, see e.g., Corollary 5.5 from \cite{geomarea}, we obtain that, uniformly in $n$, \[\displaystyle{\max_{t\in[- T,T]}\,X^1(t)} \le x_T + T^\delta\,, \quad {\rm w.h.p.}\] Therefore, setting $S := T - T^{-10}$, 
                     \rlem{lem:crudebnd} implies that uniformly in $n$, w.h.p. 
\[
X^{1}(\pm S) \leq x_T + T^{\del} \;\;\text{ and } \;\;X^{n_0 + 1}(t) \leq T^{\del},\quad  \forall\;t \in [-S, S].
\]
         If we condition on this event,  by monotonicity, on the interval $[-S,S]$, the LE $(X^{1}, \ldots, X^{n_{0}})$ is dominated by $CT^\delta + (\widehat X^{1}, \ldots, \widehat X^{n_{0}})$, where $C$ is a suitable constant and $\widehat X$ is  the $n_{0}$-line ensemble on $[-S, S]$ with all lines having boundary $x_S=S^2-2KS$ and floor at zero. By adjusting the value of the constant $\delta>0$, and by renaming $S$ to $T$, we have  thus reduced the proof to establishing that w.h.p. one has $X^{2}(t)\le T^\delta$ for all $t\in[-\eta T,\eta T]$, where $\eta=\eta(\lambda)>0$ is a constant. Here $X^2$ denotes the second line of the ensemble $\uX\sim\L_{n_0;T}^{\ux_T}$ with $n_0$ lines, where $\ux_T\equiv x_T=T^2-2KT$. 
We further note that, by monotonicity, we may raise the boundary condition to $\ux'_T$, where $(x'_T)_k=S_k= T^2 - 2KT + 4(n_{0} - k + 1) T^\del$, and thus we are in the setting of Lemma \ref{lem:ensinduclem}. The latter implies that w.h.p.\
\[
X^{2}(t)\preceq  
            \fl_{1}(t) + \p_{1}(t) %
            \leq 4n_{0} T^{\del}\,,\qquad \forall\,t\in[-\bxi_{1}, \bxi_{1}].
\]
Since, by \eqref{eq:bxibnd}, $\bxi_{1} \ge (1-\beta)T $, one has $[-\bxi_{1}, \bxi_{1}]\supset[-\eta T,\eta T]$, where $\eta = (1-\beta)$. Since  $n_0=O(\log T)$, by adjusting the value of the constant $\delta>0$, this completes the proof of Proposition \ref{ppn:refinedbnd}.

\subsection{Proof of Lemma \ref{lem:ensinduclem}}\label{sec:pflight}
Let $\uX\sim \L_{n_0;T}^{\ux'_T} $ be as in the statement of Lemma  \ref{lem:ensinduclem}. We proceed by induction over $k$. When proving an upper bound on the $k$-th line $X^k$, by monotonicity we may remove all lines above it. Thus assuming inductively that $X^{k+1} \preceq \p_k +  \fl_{k }$, by monotonicity we may restrict to a single line with floor at $\p_k +  \fl_{k }$. By using the \PAR \; representation from Lemma \ref{lem:girsanov}, we then  consider the  bound $Y^k+\p_k\preceq \fl_{k-1 }+\p_{k-1}$, where $Y^k \sim \Bav{-\p_k(T) + S_k}{ T}{\fl_k}$. 
The proof of Lemma  \ref{lem:ensinduclem} is complete once we establish the base case $k=n_0$ and the inductive step from $k+1$ to $k$ discussed above. Thus,  using a union bound over the $n_0=O(\log T)$ paths,  Lemma \ref{lem:ensinduclem} follows from the next %
lemma.

    \begin{lem}\label{lem:induclema}
       Let $Y^k \sim \Bav{-\p_k(T) + S_k}{ T}{\fl_k}$. Then, for all $ k=2,\dots,n_0$,
        \begin{gather}
          Y^k +\p_k\preceq \fl_{k-1 } +\p_{k-1} \quad {\rm w.h.p.}
          \label{eq:topr}
        \end{gather}
  \end{lem}
We now embark on the proof of Lemma \ref{lem:induclema}.
 {We will establish three auxiliary results, appearing in the next three lemmas, that will allow us to prove the desired statement. The first is a technical lemma which quantifies %
 the following observation already implicit in our proofs for the one-line case.
        If $Y \sim \Bav{H}{T}{-\p}$, then  $Y$  follows the the boundary of the convex hull of the region delimited by  $(-T, H), (T, H)$ and $-\p$, up to Brownian fluctuations.
        When $H < 0$, this consists of the tangent lines from $(\pm T, H)$ to $-\p$, and the stretch of $-\p$ in between the tangency locations, namely the function denoted $\q$ in \eqref{erdefpathb} below, see also Figure \ref{fig:stay-close} for an illustration. }
   Define the 
   error function  
   \begin{align}\label{erdefn234}
          \Err(t) := T^\del  (T - |t|)^{1/2} + T^{-10}\,, \qquad  t\in[- T,T]\,.
        \end{align}
    \begin{lem}[Typical path]
        \label{lem:stayclose}
{For any $a\ge1,$}
        consider $Y \sim \Bav{H}{T}{-a\p}$ {where $H \in  [-a \p(T), 0]$}. Let $\pm \xi$ be the $x$-tangency locations from $(\pm T, H)$ to $-a\p$ respectively.
        Let $\q$ denote the %
        function %
        \begin{align}\label{erdefpathb}
          \q(t) := \begin{cases}
                     \ell_1(t), & t < -\xi,       \\
                     -a\p(t),   & -\xi < t < \xi, \\
                     \ell_2(t), & \xi < t,
                   \end{cases}
        \end{align}
        where $\ell_1, \ell_2$ are the tangent lines from $(-T, H)$ to $(-\xi, -a\p(-\xi))$ and from $(T, H)$ to $(\xi, -a\p(\xi))$ respectively. 
        Then,  %
        \begin{align}
          \label{eq:errclose}
          \Abs{Y(t) - \q(t)} \leq \Err(t),\quad \forall \, t\in[- T,T]\,,\qquad {\rm w.h.p.}
        \end{align}
\end{lem}

We stress that the above estimates hold for any fixed choice of the parameter $\del>0$ and that the constants hidden in the w.h.p.\ statements are independent of  $a\ge 1$, which may itself depend on $T$ in the applications below.  
        \begin{proof}
         Observe that $Y$ is stochastically dominated by the Brownian excursion on $[-T, T]$ conditioned to stay above the straight line $\ell_2$. %
       By an application of     \rlem{lem:brownianav2}
 we then conclude that
        \begin{align*}
          Y(t) \leq {\ell_2(t) + } \Err(t),\quad \forall\, t \in [\xi, T]\,,\qquad {\rm w.h.p.}
        \end{align*}
        The same bound holds on $[-T,-\xi]$ with $\ell_1$ replacing $\ell_2$. Hence, a union bound proves the upper bound in \eqref{eq:errclose} restricted to $|t|\ge \xi$. 
        To prove it also for $|t|\le \xi$, we condition on the high probability event 
        \[Y(\pm \xi)\le \q(\pm \xi) + \Err( \xi) = -a\p(\xi)+ \Err( \xi) \,,
        \] and focus on the behavior in the interval $[-\xi, \xi]$. In fact, by the same argument as above we may even assume $Y(\pm \xi)\le \q(\pm \xi) + \Err( \xi)/2$, and thus
            going back to the area-tilted model, by monotonicity, we may dominate $Y + a\p$ on this interval by $Z \sim \L^{2a;z}_{ \xi}$, with boundary values $z:=\Err(\xi)/2$. %
            By neglecting the area tilt, at this point one can again use Brownian excursion estimates and conclude, by another application of \rlem{lem:brownianav2},
            that $Z(t)\leq \Err(t)$, for all $t\in[-\xi,\xi]$ w.h.p. This proves the  upper bound in \eqref{eq:errclose}.
            
        For the lower bound, observe that it is sufficient to prove it only for $|t| \geq \xi$. For $t\in[\xi,T]$, 
         $X$ stochastically dominates $V$ where $V \sim \B^{\ell_2(\xi), \ell_2(T)}_{\xi, T}$, and so the bound follows from Brownian fluctuations, see again
\rlem{lem:brownianav2}.
\end{proof}
    
 We also need the following refinement of the above lemma, which holds under additional assumptions on  
 the tangency location $\xi$ guaranteeing that it is not too close to the endpoint $T$ and that it is well separated from the origin. Under these assumptions, it gives a stronger bound on the height between the two tangency locations, which exploits the pull to the parabolic barrier to achieve a height at most $T^\delta$ in the bulk.  
     \begin{lem}[Typical path 2]
        \label{lem:stayclose2}
In the setting of Lemma \ref{lem:stayclose}, suppose further that 
 \begin{gather}
          T - \xi \geq \max\{T^{-20},a^{-2/3}\}\,,\qquad 
          \xi \geq \al T,\label{eq:condxibig}
        \end{gather}
        where $\al>0$ is a fixed constant. Then %
       \begin{align}
          \label{eq:errclose3}
          \Abs{Y(t) - \q(t)} \leq \Err'(t),\quad \forall \, t\in[- T,T]\,,\qquad {\rm w.h.p.}
        \end{align}
        where
        \begin{align}\label{errdefn1}
          \Err'(t) := \begin{cases}\Err(t),
& \text{ if } \; |t| \geq \xi - C T^{3\del/2}(T - \xi)^{3/4}a^{-1/6}, \\                  
                         T^\del,                & \text{otherwise}
                       \end{cases}
        \end{align}
        and $C$ is some absolute constant. See Figure \ref{fig:stay-close} for an illustration.
    \end{lem}

 \begin{figure}[ht]
        \centering
                \input{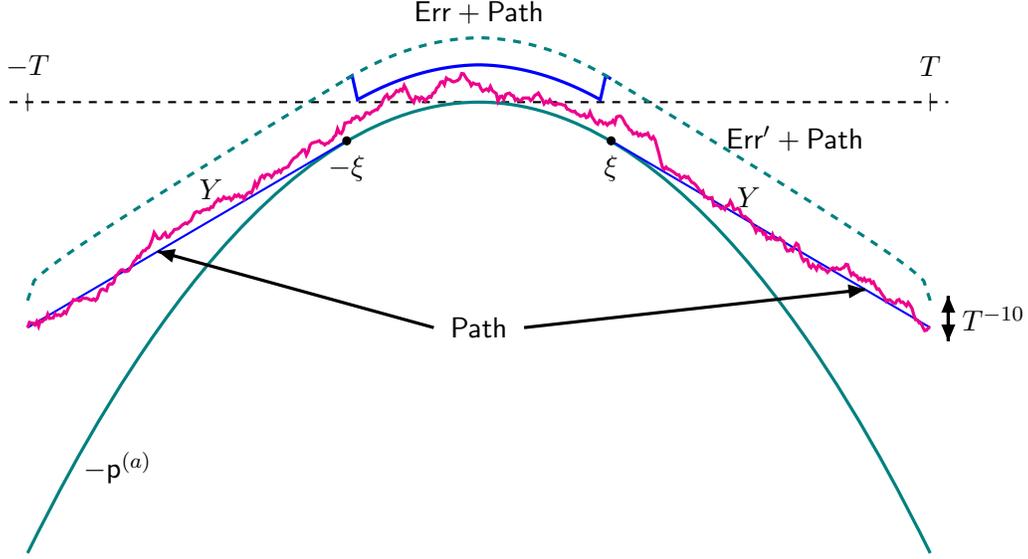}
        \caption{Illustration of the error functions $\Err$ and $\Err'$ defined in \eqref{erdefn234} and \eqref{errdefn1}.}
        \label{fig:stay-close}
    \end{figure}
    \begin{proof}
Thanks to the previous lemma we may restrict to $|t|\le \xi$, and we may assume that  
                \begin{align*}
              Y(\pm \xi) - \q(\pm \xi) \leq T^{\del}(T - \xi)^{1/2} + T^{-10} \leq 2T^{\del}(T - \xi)^{1/2},
            \end{align*}
            where we use the assumption $T - \xi \geq T^{-20}$ in the last inequality.
            Therefore, switching back to  the area-tilted model, by monotonicity, we dominate $Y + a\p$ on $[-\xi,\xi]$ by $Z \sim \L^{2a;z}_{\xi}$, with $z:=2T^{\del}(T - \xi)^{1/2}$.
             We now  apply \rlem{lem:doublesidedcomingdown} by plugging in $X = Z, H = 2T^{\del}(T - \xi)^{1/2}, T = \xi$. The
             hypotheses required by \rlem{lem:doublesidedcomingdown} are satisfied since by our assumptions \eqref{eq:condxibig} one has   $\xi \geq \al T %
             \geq 3H^{3/2}$ for all $T$ large enough, if $\del>0$ is sufficiently small, which can indeed assume without loss of generality.
             Since also $2T^{\del}(T - \xi)^{1/2} a^{1/3} \geq 2T^{\del}$, it follows that w.h.p. 
            \begin{align*}
              Z(t) \leq \begin{cases}
                         8T^{\del}(T - \xi)^{1/2} & t \in [-\xi, -\xi + S] \cup [\xi - S, \xi], \\
                         \xi^\del \leq T^{\del} & t \in [-\xi + S, \xi - S],
                       \end{cases}
            \end{align*}
            where $S = CT^{3\del/2}(T - \xi)^{3/4}a^{-1/6}$,
for an absolute constant $C$.

    \end{proof}
    \begin{rmk}\label{rem:conds}
   In the application to light paths one has the corresponding parabolas $-a\p$ where $a=\lambda^{k-1}$, $k=2,\dots,n_0$, $\xi=\xi_k$, and thus using \eqref{eq:xiubnd} and \eqref{eq:xilbnd}, the assumptions in \eqref{eq:condxibig} are easily seen to be satisfied for such cases. 
   Indeed, 
   \[T - \xi_{k} \geq T\sqrt{\g}\lam^{-(k - 1)/2} \geq T^{-20},\] as long as $T$ is large enough due to the assumption that $T\lam^{-(n_{0} - 1)/2} \geq T^{-15}$.
Moreover, %
\[T - \xi_{k} \geq T\sqrt{\g}\lam^{-(k - 1)/2} \geq \lam^{-2(k - 1)/3}\,,\] 
 for large enough  $T$ because $\lam \geq 1$. Finally,  \[\xi_{k} \geq T(1 - \lam^{-(k - 1)/2}) \geq T(1 - \lam^{1/2})\,,\] as long as $k \geq 2$.
\end{rmk}    
    We turn to the last auxiliary result for the proof of  Lemma \ref{lem:induclema}. Define the functions
         \begin{align}\label{shiftedpath123}
            \q_{k}(t) - [4(n_0 - k) + 2]T^{\del} := \widehat{\q}_{k}(t) := \begin{cases}
                                                                          \ell_{3}(t) & t \leq -\xi_{k}, \\
                                                                          -\p_{k}(t) & -\xi_{k} \leq t \leq \xi_{k}, \\
                                                                          \ell_{4}(t) & \xi_{k} \leq t,
                                                                       \end{cases}
          \end{align}
          where the lines $\ell_{3}, \ell_{4}$ are tangent lines to $-\p_{k}$ from $(\mp T, -\p_{k}(T) + T^{2} - 2KT + 2T^{\del})$ respectively, and $\pm\xi_k$ are the corresponding tangency locations, see \eqref{eq:xiubnd} and \eqref{eq:xilbnd}. Defining $\widehat{\q}_{k}(\cdot)$ to be a shifted version of ${\q}_{k}(\cdot)$ will be convenient for the arguments.

\begin{figure}[ht]
        \centering
        \input{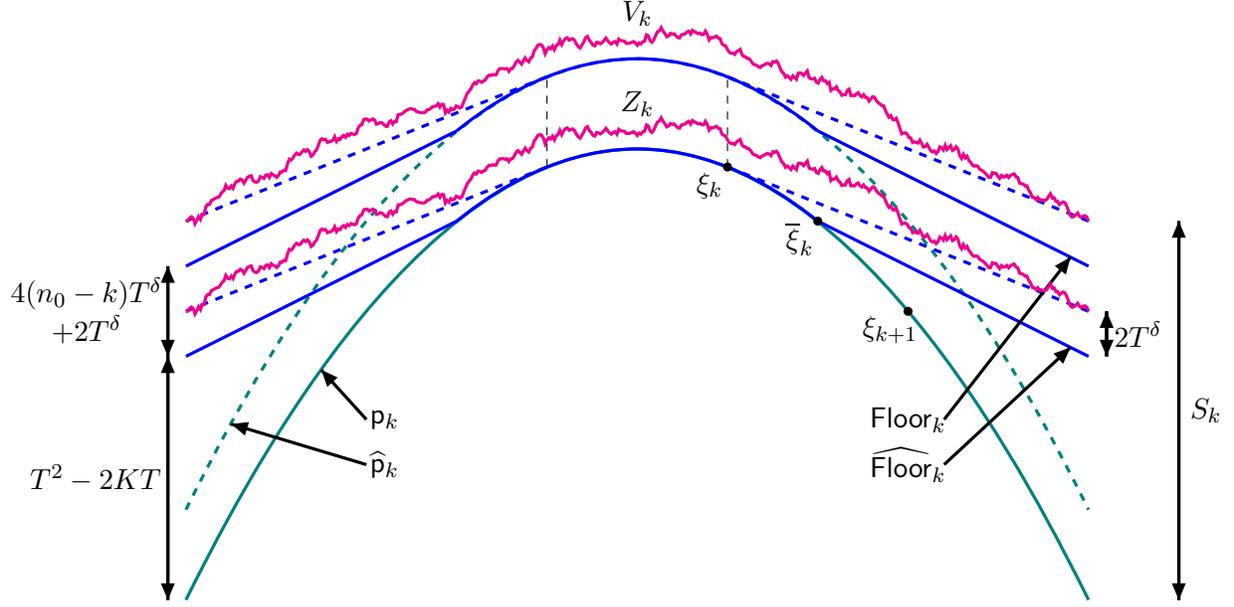}
        \caption{Important quantities in the statement and proof of \rlem{lem:ensinduclem}, in particular indicating the shifts involved in the definitions as in \eqref{shiftedpath123}.}
        \label{fig:inductive-step}
    \end{figure}
      \begin{lem}[Typical path 3]
        \label{lem:stayclose3}
In the setting of Lemma \ref{lem:induclema}, 
  \begin{align}
  \left|Y^{k} - \q_{k}\right|
   \preceq \Err\,,\qquad {\rm w.h.p.} 
  \label{eq:followpath}
        \end{align}
    \end{lem}
\begin{proof}
 We first show that Brownian motion conditioned to avoid the shifted parabola 
 \[\hat \p_k:=-\p_k + [4(n_0 - k) + 2]T^\del
 \]
  will avoid $\fl_k$ with high probability. For an illustration the reader may refer to Figure \ref{fig:inductive-step}.
        Since $\fl_k \gtr \hat \p_k$ by construction, this says that the constraints of avoiding $\hat \p_k$ and $\fl_k$ are essentially equivalent, which will be useful to establish the claim, since it allows us to use estimates already developed for Brownian bridges conditioned to avoid a parabola.

Let $V^k\sim\Bav{-\p_k(T) + S_k}{T}{\hat\p_{k}}$.  We claim that w.h.p. %
\begin{align}\label{eq:clyk}
V^k \succeq \fl_k.
\end{align}
        Shifting down by $ [4(n_0 - k) + 2]T^\del$, by the definition \eqref{shiftedfloor123},  this is equivalent to %
        $$Z^k \succeq \sfl_k\,,\qquad Z^k := V^k - [4(n_0 - k) + 2]T^\del\sim \Bav{-\p_k(T) + T^2 - 2KT + 2T^\del}{T}{-\p_{k}},$$
        where 
        \[
        \hat\fl_k(t)= \begin{cases}
                                                          \ell_1(t), & t \leq -\bxi_k,              \\
                                                          -\p_k(t),  & -\bxi_k \leq t \leq \bxi_k, \\
                                                          \ell_2(t), & \bxi_k \leq t.
                                                        \end{cases}
\]
   {
            We apply \rlem{lem:stayclose} to $Z^{k}$, so that $a=\lambda^{k-1}$, $H = Z^{k}(\pm T) = -\p_{k}(T) + T^{2} - 2KT + 2T^{\del}$, $\xi=\xi_k$. 
            Therefore, 
by \rlem{lem:stayclose}, %
            \begin{align}
              \label{eq:errbndz}
              \Abs{Z^{k} - \hq_{k}} \leq \Err, \qquad {\rm w.h.p.}
            \end{align}
                                  
 Our next step is to compute the distance between $\hq_k$ and $\sfl_k$ and show that it is bigger than $\Err$. Since $Z^k$ avoiding {$\sfl_k$} is the same as $Z^k$ avoiding the $\ell_1, \ell_2$ pieces in the definition of $\sfl_k$, using symmetry, we only need to compare the two functions in the interval $[\bxi_k, T]$. In this interval, 
the gap between $\hq_k$ and $\sfl_k$ is a line with vertical heights between $G_k$ and $2T^{\del}$ where
        \begin{align*}
          G_k &:= \hq_k(\bxi_k) - \sfl_k(\bxi_k), \\
              &= -2\lam^k \xi_k \bxi_k + \lam^k \xi_k^2 + \lam^k\bxi_k^2,    \\
              &= \lam^k(\bxi_k - \xi_k)^2 \geq \Rnd{\sqrt{\g} - \be}^2 T^2,
        \end{align*}
        where we used the bounds on $\xi_k, \bxi_k$ established in \eqref{eq:xiubnd}-\eqref{eq:bxibnd}. Thus, for large enough $T$ the gap is always more than $\Err(t)$ in $[\bxi_k, T]$ because at the right endpoint, $$\Err(T) = T^{-10} \leq \hq_{k}(T) - \sfl_{k}(T) = 2T^{\del}$$ and at the left endpoint 
        $$\Err(\bxi_{k}) = T^{\del}(T - \bxi_{k})^{0.5} + T^{-10} \leq T \leq (\sqrt{\g} - \be)^{2}T^{2}.$$}
       Therefore, \eqref{eq:errbndz} implies  that  $Z^k \succ \sfl_k$ w.h.p. which establishes the claimed bound \eqref{eq:clyk}.

We can now finish the proof of  the lemma. In the statement \eqref{eq:followpath} we have $Y^k \sim \Bav{-\p_k(T) + S_k}{T}{\fl_k}$, and therefore 
$$Y^k-[4(n_0 - k) + 2]T^\del\sim Z^k\left(\,\cdot \mid Z^k \gtr \sfl_k\right).$$ 
   Hence we may couple $Y^k$ and $V^k = Z^k + [4(n_0 - k) + 2]T^\del$ with high probability.
         In conclusion shifting back up by $[4(n_0 - k) + 2]T^\del$, \eqref{eq:errbndz} shows that 
         $\Abs{Y^{k} - \q_{k}} \leq \Err$ w.h.p. This proves \eqref{eq:followpath}.
\end{proof}
We are now ready to address  the proof of   Lemma \ref{lem:induclema}. 
 
     \subsection{Proof of Lemma \ref{lem:induclema}}
We need to prove 
  \begin{align}
 Y^k + \p_k \preceq \fl_{k - 1} + \p_{k - 1} \,,\qquad {\rm w.h.p.} \label{eq:floorav}
        \end{align}
                  By shifting everything down by $[4(n_0 - k) + 2]T^\del$, with the notation from the proof of Lemma \ref{lem:stayclose3}, the  inequality \eqref{eq:floorav} is equivalent to 
        \begin{align}\label{keybound123}
          Z^k + \p_k & \preceq \fl_{k - 1} - [4(n_0 - k) + 2]T^\del + \p_{k - 1} \\
          \notag                   &= \sfl_{k - 1} + [4(n_{0} - k + 1) + 2]T^\del - [4(n_{0}- k) + 2]T^\del + \p_{k - 1}  \\
          \notag                   &= \sfl_{k - 1} + 4T^\del + \p_{k - 1}.
        \end{align}
            From \rlem{lem:stayclose2} applied to $Z^{k}$, so that $H = Z^{k}(\pm T) = -\p_{k}(T) + T^2 - 2KT + 2T^{\del}$, we get that
            \begin{align}\label{thepath}
              Z^{k} \preceq \hq_{k} + \Err'\,,\qquad {\rm w.h.p.}
            \end{align}
          We refer to Remark \ref{rem:conds} for the validity of  the three conditions in \eqref{eq:condxibig}. 
          
Therefore, it is sufficient to prove that
        \begin{align}
          \hq_k + \Err' + \p_k \preceq \sfl_{k - 1} + 4T^\del + \p_{k - 1}. \label{eq:reqb}
        \end{align}
        Note that this a purely deterministic relation, which can be checked via elementary geometric considerations. Nonetheless, the proof requires a bit of work.    
We do this in three steps. 
   First, guided by \eqref{errdefn1}, we define 
   \begin{align}\label{plugin25}\xi'_k=\xi_k - C T^{3\del/2}(T - \xi_k)^{3/4}\lambda^{-(k-1)/6}.
\end{align}
We %
 verify \eqref{eq:reqb} sequentially on the intervals $[-\xi'_k,\xi'_k], [\xi'_k, \xi_k], [\xi_k,T]$ with their symmetric versions holding analogously.

We start with  the interval $[-\xi'_k,\xi'_k].$ Note that by \eqref{errdefn1}, { $\Err'(t) = T^\del$ on $[-\xi'_k, \xi'_k]$. Since $\sfl_{k - 1} + \p_{k - 1} \geq 0$ for all $t$ and $\hq_k + \p_k = 0$ on $[-\xi_k, \xi_k] \supseteq [-\xi'_k, \xi'_k]$, condition \eqref{eq:reqb} holds on this interval.}

Next, consider the interval $[\xi_k, T]$. On this interval, by \eqref{shiftedfloor123} we have $\sfl_{k - 1} = \ell_{2}$  and by \eqref{errdefn1} we have $\Err=\Err'$. Therefore, it is sufficient to check
\begin{align}\label{eq:checkfloor}
  \hq_{k} + \p_{k} - \p_{k - 1} + \Err \preceq \ell_{2} + 4T^{\del}.
\end{align}

This is accomplished in two steps: First we show that $\hq_{k} + \p_{k} - \p_{k - 1} \preceq \ell$ for a certain straight line $\ell$ due to the convexity of $\hq_{k} + \p_{k} - \p_{k - 1}$. In the second step we show that $\Err \preceq \ell_{2} - \ell + 4T^{\del}$.
{For the first step, observe that $$\hq_k + \p_k - \p_{k - 1} = \hq_{k} + (\lam^{k-1} - \lam^{k - 2})t^{2}$$ is a {convex parabola} since $\hq_{k}$ is linear on $[\xi_{k}, T]$, and $\lam^{k - 1} - \lam^{k - 2} > 0$.
    Thus we upper bound it on the interval $[\xi_k, T]$ by the straight line $\ell$ joining its endpoints. We now compute the endpoints of $\ell$.}
        At $T$, we have
        \begin{align*}
          \hq_k(T) + \p_k(T) - \p_{k - 1}(T) = T^2 - 2KT + 2T^\del - \lam^{k - 2}T^2
        \end{align*}
        and at $\xi_k$ we have
        \begin{align*}
          \hq_k(\xi_k) + \p_k(\xi_k) - \p_{k - 1}(\xi_k) = - \lam^{k - 2}\xi_k^2
        \end{align*}
        because $\hq_k(\xi_k) + \p_k(\xi_k) = 0$. %
        Thus, $\hq_k + \p_k - \p_{k - 1}$ is bounded on $[\xi_{k}, T]$ by the straight line $\ell$ satisfying:
        \begin{align}
          \label{eq:ellboundT}
          \ell(T) &= -\p_{k - 1}(T) + T^2 - 2KT + 2T^\del, \\
          \ell(\xi_{k}) &= -\p_{k - 1}(\xi_k).
        \end{align}   {
      Moving on to the second step, we wish to show that $\ell_{2} - \ell + 4T^{\del} \gtr \Err$. Since $\ell_{2} - \ell + 4T^{\del}$ is a line, we first lower bound its value at both endpoints $t = \xi_{k}$ and $t = T$, and then show that the line joining these lower bounds actually upper bounds $\Err$. As $\ell(\xi_{k}) = -\p_{k-1}(\xi_{k})$, to lower bound $\ell_{2} - \ell + 4T^{\del}$ at $\xi_{k}$, it is sufficient to lower bound $(\ell_{2} - \p_{k - 1})(\xi_{k})$. This is accomplished as follows: Let $\ell_{c}$ be the chord joining $(\bxi_{k - 1}, -\p_{k - 1}(\bxi_{k - 1}))$ and $(\xi_{k}, -\p_{k - 1}(\xi_{k}))$. Then
            \begin{align}
              \label{eq:goalslopediff}
              \ell_{2}(\xi_{k}) - \p_{k - 1}(\xi_{k}) = (\xi_{k} - \bxi_{k - 1}) \cdot (\text{absolute slope difference between $\ell_{2}$ and $\ell_{c}$}),
            \end{align}
            see Figure \ref{fig:inductive-step-2} for an illustration.
            We lower bound both factors next.

    \begin{figure}[ht]
        \centering
                \begin{tikzpicture}[scale=1.2]
\colorlet{colora}{magenta}
\colorlet{colorb}{teal}
\colorlet{colorc}{blue}
\draw[dashed] (-1.66667,0) -- (5.4,0);
\draw (5,-0.1) -- (5,0.1);
\draw (5,0.3) node {$T$};
\draw[very thick,colorb] (-1.66667,-0.55556) parabola bend (0,0) (5,-5);
\fill (3.5,-2.45) circle[radius=0.05];
\draw (3.3,-2.65) node {$\xi_k$};
\fill (2.25,-1.0125) circle[radius=0.05];
\draw (2.05,-1.2125) node {$\bar{\xi}_{k - 1}$};
\fill (1,-0.2) circle[radius=0.05];
\draw (0.8,-0.4) node {$\xi_{k - 1}$};
\fill (3,-1.8) circle[radius=0.05];
\draw (2.8,-2) node {$\xi_k'$};
\draw[very thick] (3,-1.8) -- (3.5,-2.45);
\draw (-1,-0.6) node {$-\mathsf{p}_{k - 1}$};
\draw[very thick] (2.5,-3) -- (3.25,-2.1125);
\fill[very thick] (3.25,-2.1125) -- (3.19059,-2.30347) -- (3.07161,-2.20293);
\draw (2.5,-3.2) node {$\ell_c$};
\draw[very thick,colorc] (5,-3) -- (2.25,-1.0125);
\draw (3.45,-1.6125) node {$\ell_2$};
\draw[dashed,colorb] (3.5,-2.45) -- (5,-2.5);
\draw[very thick,colorc,dashed] (5,-2) -- (2.25,-0.0125);
\draw (4,-0.5) node {$\widehat{\mathsf{Floor}_k} + 4T^\delta$};
\draw[very thick] (5.2,-5) -- (5.2,-3);
\fill[very thick] (5.2,-3) -- (5.27788,-3.18421) -- (5.12212,-3.18421);
\draw[very thick] (5.2,-3) -- (5.2,-5);
\fill[very thick] (5.2,-5) -- (5.12212,-4.81579) -- (5.27788,-4.81579);
\draw (6,-4) node {$T^2 - 2KT$};
\draw[very thick] (5.2,-3) -- (5.2,-2.5);
\fill[very thick] (5.2,-2.5) -- (5.27788,-2.68421) -- (5.12212,-2.68421);
\draw[very thick] (5.2,-2.5) -- (5.2,-3);
\fill[very thick] (5.2,-3) -- (5.12212,-2.81579) -- (5.27788,-2.81579);
\draw (5.6,-2.75) node {$2T^\delta$};
\draw[very thick] (5.2,-2.5) -- (5.2,-2);
\fill[very thick] (5.2,-2) -- (5.27788,-2.18421) -- (5.12212,-2.18421);
\draw[very thick] (5.2,-2) -- (5.2,-2.5);
\fill[very thick] (5.2,-2.5) -- (5.12212,-2.31579) -- (5.27788,-2.31579);
\draw (5.6,-2.25) node {$2T^\delta$};
\draw[very thick,colorc,dashed] (-1.66667,0.44444) parabola bend (0,1) (2.25,-0.0125);
\draw[colora,dashed,very thick] (5,-3) parabola bend (4.8,-3.5) (3.5,-2.45);
\draw[very thick] (5.5,-3.4) -- (4.8,-3.5);
\fill[very thick] (4.8,-3.5) -- (4.97135,-3.39685) -- (4.99338,-3.55105);
\draw (6.8,-3.4) node {$\widehat{\mathsf{Path}}_k + \mathsf{p}_k - \mathsf{p}_{k - 1}$};
\end{tikzpicture}
        \caption{Right side of the parabola $-\p_{k-1}$ along with quantities of interest in the proof of Lemma \ref{lem:induclema}. Objects illustrated include for instance, the chord $[\xi'_k, \xi_k]$, as well as $\hq_{k}+\p_k-\p_{k-1}$, the floor induced on $Y^{k-1}$ up to a shift.}
        \label{fig:inductive-step-2}
    \end{figure}
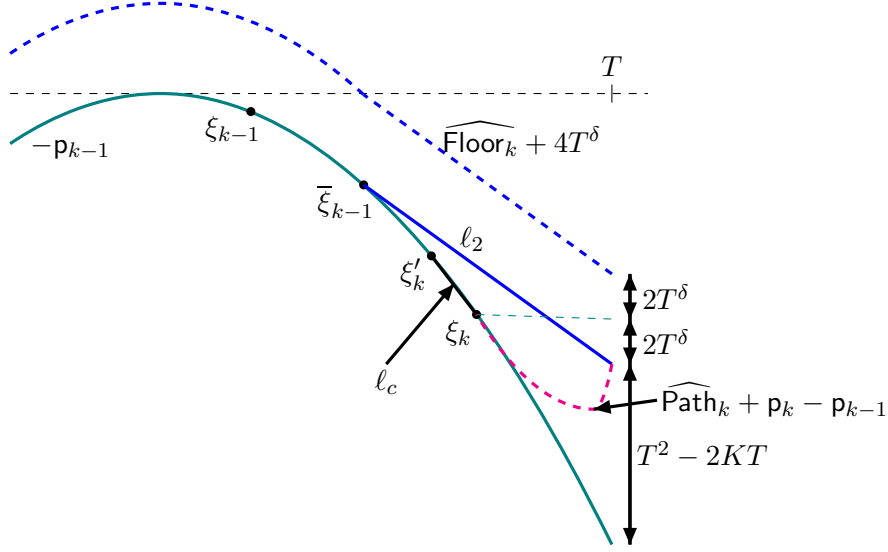
         Using \eqref{eq:xilbnd} and \eqref{eq:bxibnd} 
            \begin{align}
              \label{eq:xikmbxik}
              \xi_k - \bxi_{k - 1} \geq T(1 - \lam^{-(k - 1)/2}) - T(1 - \be\sqrt{\g} \lam^{-(k - 2)/2}) = T\lam^{-(k - 1)/2}(\be\sqrt{\g\lam} - 1).
            \end{align}
           The slope of the line $\l_2$ is
        \begin{align*}
          \f{T^2 - 2KT - \lam^{k - 2}T^{2} + \lam^{k - 2}\bxi_{k - 1}^2}{T - \bxi_{k - 1}}
          &\geq \f{T - 2K}{\be \sqrt{\lam}} \lam^{(k - 1)/2} - \lam^{k - 2}(T + \bxi_{k - 1})
        \end{align*}
        using the lower bound on $\bxi_{k - 1}$ from \eqref{eq:bxibnd}. The slope of the chord $\ell_c$ joining $(\bxi_{k - 1}, -\p_{k - 1}(\bxi_{k - 1}))$ and $(\xi_k, -\p_{k - 1}(\xi_k))$ is $-\lam^{k - 2}(\bxi_{k - 1} + \xi_k)$, so that the difference of slopes between $\ell_{2}$ and $\ell_{c}$ is at least
        \begin{align*}
          &\f{T - 2K}{\be \sqrt{\lam}} \lam^{(k - 1)/2} - \lam^{k - 2}(T + \bxi_{k - 1}) + \lam^{k - 2}(\bxi_{k - 1} + \xi_k) =
                                                                                                                         { \f{T - 2K}{\be \sqrt{\lam}} \lam^{(k - 1)/2} - \lam^{k - 2}(T - \xi_k)} \\
                                                                                                                       & \qquad \qquad\geq \f{T - 2K}{\be \sqrt{\lam}} \lam^{(k - 1)/2} - \lam^{k - 2}T\lam^{-(k - 1)/2} 
                                                                                                                     = T\lam^{(k - 1)/2}\Box{\f{\g}{\be\sqrt{\lam}} - \f{1}{\lam}} > 0,
        \end{align*}
        where we use \eqref{eq:gdef} for the last equality, as long as $T$ is sufficiently large.  Plugging all these estimates into \eqref{eq:goalslopediff} yields
        \begin{align}
          \label{eq:slopegap}
          \ell_{2}(\xi_{k}) - \p_{k - 1}(\xi_{k}) 
          \geq %
          \f{T^2}{\be\lam}\, (\g\sqrt{\lam} - \be)(\be\sqrt{\g\lam} - 1)
          = :\t \,T^2,
        \end{align}
        where $\t = \t(\g, \lam) > 0$ is a constant.

        Recall from definition \eqref{shiftedfloor123} that $\ell_{2}(T) = -\p_{k - 1}(T) + T^{2} - 2KT$ and from \eqref{eq:ellboundT} that $\ell(T) = -\p_{k - 1}(T) + T^{2} - 2KT + 2T^{\del}$. Therefore, in conclusion, $\l_2 + 4T^\del - \l$ is a line on $[\xi_{k}, T]$ with endpoints at least  $\t T^2 + 4T^\del$ and $2T^\del$, and therefore exceeds $\Err(t)$ throughout the interval. Therefore condition \eqref{eq:reqb} is also satisfied on $[\xi_k, T]$, and symmetrically on $[-T, -\xi_k]$.

    It remains to verify the condition  \eqref{eq:reqb} on the interval $[\xi'_k, \xi_k]$.
          For this, we will again check \eqref{eq:checkfloor}, but now on the interval $[\xi'_{k}, \xi_{k}]$. Recall that by \eqref{shiftedpath123}, on this interval, $\hq_{k} = -\p_{k}$, so that it is sufficient to check
        \begin{align*}
          -\p_{k - 1} + \Err\preceq \ell_{2} + 4T^{\del}.
        \end{align*}
        As long as $T$ is sufficiently large, %
        one has $\Err(t) \leq 2T^{\del + 1/2}$ for all $t \in [\xi'_{k}, \xi_{k}]$, so that it is sufficient to check the stronger statement
        \begin{align*}
          2T^{\del + 1/2} \preceq \ell_{2} - (-\p_{k - 1})
        \end{align*}
        throughout $[\xi'_{k}, \xi_{k}]$. %
        One can in fact show that the right hand side, i.e., the  ``gap'' between $\ell_{2}$ and $-\p_{k - 1}$ is at least $\T{T^{2}}$ throughout the said interval.

        To this end, we proceed as follows. We  lower bound the gap between $\ell_{2}$ and $\ell_{c}$ throughout the interval by $\t T^{2}/2$, and upper bound the gap between $\ell_{c}$ and $-\p_{k - 1}$ by $\t T^{2}/4$, see Figure \ref{fig:inductive-step-2} for an illustration. These two together will be sufficient to establish the desired claim.

        Since $\ell_{2} - \ell_{c}$ is straight line, to lower bound it by $\t T^{2}/2$, it is sufficient to show that it is at least $\t T^{2}/2$ at both endpoints, i.e., at $t = \xi'_{k}$ and $t = \xi_{k}$.
        We already know from \eqref{eq:slopegap} that the gap between $\ell_{2}$ and $-\p_{k - 1}$ is $\t T^{2}$ at $t = \xi_{k}$. This is also the gap between $\ell_{2}$ and $\ell_{c}$, because $\ell_{c}(\xi_{k}) = -\p_{k - 1}(\xi_{k})$. We now show that the gap is at least $\t T^{2}/2$ at $t = \xi'_{k}$. For that we will apply the exact same argument as in \eqref{eq:slopegap}, but with $\xi_{k}'$ instead of $\xi_{k}$. Observe that from \eqref{eq:slopegap}, to achieve the lower bound of at least $\t T^{2}/2$, it is sufficient to show that $\xi'_{k} - \bxi_{k - 1} \geq \f{1}{2}(\xi_{k} - \bxi_{k - 1})$ which is equivalent to showing that $\xi_{k} - \xi'_{k} \leq \f{1}{2}(\xi_{k} - \bxi_{k - 1})$. Recalling \eqref{plugin25}, the following inequality is then sufficient:
        \begin{align}
          CT^{3\del/2}&(T - \xi_{k})^{3/4}\lam^{-(k - 1)/6}  \leq CT^{3\del/2}T^{3/4}\lam^{-3(k - 1)/8}\lam^{-(k - 1)/6} \nonumber\\
                                                          &= CT^{3/4 + 3\del/2}\lam^{-\f{13}{24}(k - 1)} 
                                                          \leq \tfrac{1}{2}(\be\sqrt{\g\lam} - 1)%
                                                          T\lam^{-(k - 1)/2} 
                                                          \leq \tfrac{1}{2}(\xi_{k} - \bxi_{k - 1}),
          \label{eq:gapxi}
        \end{align}
        which holds as long as $T$ is sufficiently large and $\del$ is sufficiently small. This completes the proof of the claim that the gap between $\ell_{2}$ and $\ell_{c}$ is at least $\t T^{2}/2$ throughout the interval.

        The last remaining step is to upper bound the gap between $\ell_{c}$ and $-\p_{k - 1}$ by $\t T^{2}/4$. The gap is 0 at $t = \xi_{k}$, and the (absolute) slope of $-\p_{k - 1}$ is maximized at $t = \xi_{k}$. Therefore this gap is upper bounded by
        \begin{align}
          \label{eq:slopegap2}
          (\xi_{k} - \xi'_{k}) \cdot \Abs{(\text{slope of } \ell_{c}) - (\text{slope of $-\p_{k - 1}$ at $\xi_{k}$})}.
        \end{align}
            Note that the slope of $-\p_{k - 1}$ at $\xi_k$ is $-2\lam^{k - 2}\xi_k$. Hence, the gap between $\ell_{c}$ and $-\p_{k - 1}$ is upper bounded by
             \begin{align*}
              (\xi_k - \xi'_k) & \Box{2\lam^{k - 2}\xi_k - \lam^{k - 2}(\xi_k + \bxi_{k - 1})} 
                               = \lam^{k - 2}\Rnd{\xi_{k} - \bxi_{k - 1}}\Rnd{\xi_{k} - \xi'_{k}} \\
                               &\leq C\lam^{k - 2} T^{\f{3}{4} + \f{3\del}{2}}\lam^{-\f{13}{24}(k - 1)} \cdot \Box{T(1 - \sqrt{\g}\lam^{-(k - 1)/2}) - T(1 - \be\lam^{-(k - 2)/2})} \\
                               &\leq C T^{\f{7}{4} + \f{3\del}{2}} \Rnd{\be\sqrt{\lam} - \g}\lam^{k - 2 - \f{13}{24}(k - 1) - (k - 1)/2}\leq \t T^{2}/4
            \end{align*}
            as long as $T$ is large enough where in the first inequality we used the calculations in \eqref{eq:gapxi} and the bounds in     \eqref{eq:xiubnd} and \eqref{eq:bxibnd}. This establishes our claim, and therefore completes the proof.}

\smallskip

Next, we turn to the control on the heavy paths with index $n>n_0$. 
\subsection{Proof of Lemma \ref{lem:crudebnd}}\label{sec:pfheavy}
    We first prove an upper bound for a single area-tilted line with a high strength parameter $a$.  Iterating this as $a$ ranges from $\lambda^n$ to $\lambda^{n_0}$ will yield the desired estimate \eqref{lowerline}.

    \newcommand\Env{\mathsf{Envelope}}
    \begin{lem}[High parameter one-line envelope]
        \label{lem:highenv}
      {Let $X \sim \L^{2a;H}_{T}$} where $a = \lam^{k - 1}$ for $k > n_0$. 
      Then, for any $\delta<1/20$, $H \le 2T^2$, one has $X \preceq \Env$ 
         with probability at least
        \begin{align*}
          1 - \Exp{-\T{\lam^{k\del/2}}},
        \end{align*}
      where the symbol $\T{\cdot}$ hides an absolute constant, and we define
        \begin{align*}
          \Env(s) := \begin{cases}
                      H + \Delta_2,& \quad s \in [-T, -T + \Delta_1], \\
                      \Delta_2,& \quad s \in [-T + \Delta_1, T - \Delta_1], \\
                      H + \Delta_2,& \quad s \in [T - \Delta_1, T], \\
                    \end{cases}
        \end{align*}
        where $\Delta_2 = T^\del \lam^{-(k - n_0)/5}, \Delta_1 = T\lam^{-k/2}$.
    \end{lem}
    \begin{proof}  Switching to the $\PAR$, let $Y := X - \p_{k}$ so that
            \begin{align*}
          {
          Y \sim \Bav{-\lam^{k - 1} T^2 + H}{T}{-\p_{k}}.
          }
        \end{align*}
        As in \rlem{lem:stayclose}, let $\xi$ be the $x$-location of tangency on the tangent line $\ell$ from $(T, -\lam^{k - 1}T^2 + H)$ to $-\p_{k}$. {\rlem{lem:tangency} (with $a = \lam^{k - 1}$ and $\al = a - 2$)} shows that $ \xi= T(1 - O(\lam^{-k/2}))$. Monotonicity and Brownian excursion fluctuation arguments, as {in \rlem{lem:stayclose}}, show that
        {$Y|_{[\xi, T]}$ is bounded by the shifted line $ T^{1/2} \lam^{-k(1 - \del)/4} + \ell$} with probability at least
        \begin{align}\label{1est}
          1 - \Exp{-\T{\f{T \lam^{-k(1 - \del)/2}}{T\lam^{-(k -1)/2}}}} = 1 - \Exp{-\T{ \lam^{k\del/2}}}.
        \end{align}
        {
        Since by assumption we have $\lam^{n_{0}} \geq T^{32}$, it follows that
        \begin{align*}
          \f{T^{1/2}}{\lam^{k(1 - \del)/4}}
          &= \f{T^{1/2}}{\lam^{n_0(1 - \del)/4}} \lam^{-(k - n_0)(1 - \del)/4} 
          \leq \f{T^{1/2}}{T^{8(1 - \del)}} \lam^{-(k - n_0)(1 - \del)/4} \\
          &= T^{\f{1}{2} - 8(1 - \del)}\cdot %
          \lam^{-(k - n_0)/5} \leq T^{-7}\lam^{-(k - n_0)/5},
        \end{align*}
        using the bounds $8(1 - \del) \geq 7 + \f{1}{2}$ and $1 - \del \geq 4/5$, both of which can be satisfied with sufficiently small $\del$.
        This shows that with the same probability,
        \begin{align}
          X(t) &\leq  H + T^{-7} \lam^{-(k - n_{0})/5}, \quad \forall t \in [-T, -\xi] \cup [\xi, T]\ \ \text{and}, \nonumber\\
          X(\pm \xi) &\leq T^{-7} \lam^{-(k - n_{0})/5}. \label{eq:tangencybnd}
        \end{align}
        } To bound $X|_{[-\xi, \xi]}$, we condition on the bound \eqref{eq:tangencybnd} on $X(\pm \xi)$, and observe that in that case, $X|_{[-\xi, \xi]}$ is stochastically bounded by a single area-tilted line with zero boundary conditions and %
        with floor $T^{-7} \lam^{-(k - n_{0})/5}$. Therefore, with this conditioning, using  \eqref{eq:fsmax} we have
        {
        \begin{align}\label{2est}
          &\P[X|_{[-\xi, \xi]} \preceq T^\del \lam^{-(k - n_{0})/5}\ \Big|\ X(\pm \xi) \leq T^{-7} \lam^{-(k - n_{0})/5}] \nonumber 
          \\
          &\qquad \geq 1 - 2T\lam^{2k/3}\Exp{-\T{(T^\del\lam^{-(k - n_{0})/5})^{3/2} \lam^{k/2}}} \nonumber \\
                                                                                                            &\qquad \geq 1 - 2T\lam^{2k/3}\Exp{-\T{T^{3\del/2} \lam^{k/5}}} 
                                                                                                            \geq 1 - \Exp{-\T{T^{3\del/2}\lam^{k/5}}}.
        \end{align}
        }
        Combining \eqref{1est} and \eqref{2est} %
        finishes the proof.
    \end{proof}
\begin{figure}[ht]
        \centering
                \input{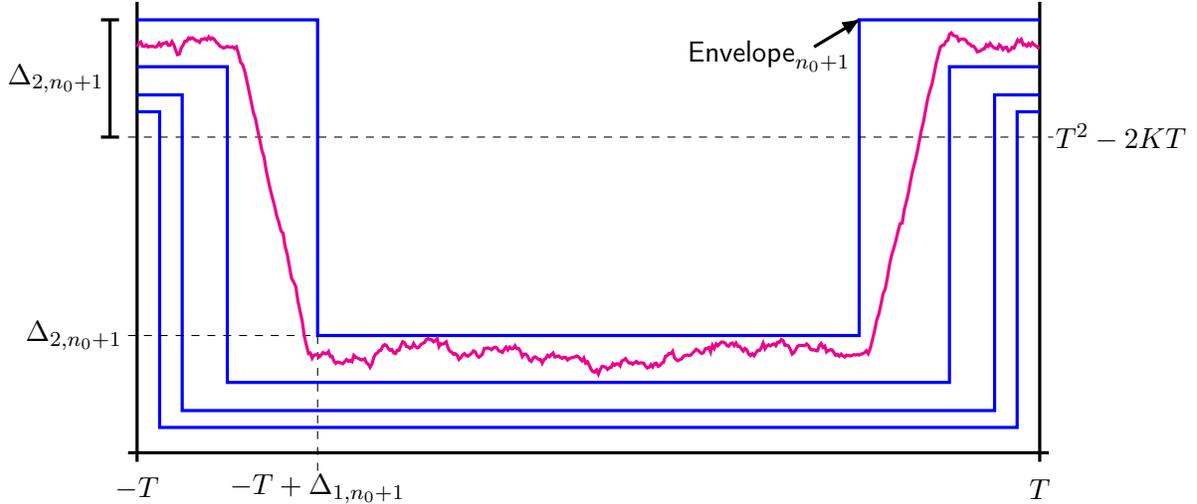}
        \caption{Illustration of the envelopes constructed in the proof of Lemma \ref{lem:crudebnd} to bound successively the heavier paths.}
        \label{fig:envelope}
    \end{figure}

    We are now in a position to complete the proof of  Lemma \ref{lem:crudebnd}.
    \begin{proof}[Proof of Lemma \ref{lem:crudebnd}]
        To prove \eqref{lowerline} we will repeatedly apply \rlem{lem:highenv} using monotonicity.
        Towards that, define for every $k > n_{0}$ (empty sums are $0$),
        \begin{align*}
          \Delta_{1,k} = T \sum_{j = k}^n \lam^{-j/2},\quad
          \Delta_{2,k} &= 2 \cdot T^\del \sum_{j = k}^n \lam^{-(j - n_{0})/5},\quad
                                               q_k = 2\sum_{j = k}^n \Exp{-c\lam^{k\del/2}}, \\
          \Env_k(t) &= \begin{cases}
                         T^2 - 2KT + \Delta_{2,k} & \quad t \in [-T, -T+\Delta_{1,k}],    \\
                         \Delta_{2,k}     & \quad t \in [-T + \Delta_{1,k}, T - \Delta_{1,k}], \\
                         T^2 -2KT + \Delta_{2,k} & \quad t \in [T - \Delta_{1,k}, T].      \\
                       \end{cases}
        \end{align*}
        Here $c>0$ is an absolute constant.
We show by induction that, if $c$ is sufficiently small,  with probability $1 - q_k$ we have $X^k \preceq \Env_k$ for all indices $k \in (n_0, n]$. {Observe that since $k > n_{0}$, $\Delta_{1, k} \leq \T{T^{-15}} \leq T^{-10},$ for all $k$ (and sufficiently large $T$). Similarly $\Delta_{2, k} \leq \T{T^{\del}}$ for all $k>n_0$.} Therefore, this inductive framework is sufficient to establish the desired bound in Lemma \ref{lem:crudebnd}.

        The base case $k=n$ immediately follows from \rlem{lem:highenv} with {$ H = T^2 - 2KT$ and $a = \lam^{n - 1}$}.
        For the inductive step, we assume that this bound is satisfied for $k + 1$ with probability $\geq 1 - q_{k + 1}$. Conditioned on this event, removing all lines above $X^k$, monotonicity implies $X^k \preceq Z^k$ where
        \begin{align*}
          Z^k \sim \Lav{2\lam^{k - 1};T^2 - 2KT + \Delta_{2,k + 1}}{ T}{\Env_{k + 1}}.
        \end{align*}
       {
        Again by monotonicity, $Z^k \preceq Y^k$ where $Y^k$ is  a single area-tilted line with zero boundary conditions and %
        with  %
        constant floor 
        $T^2 -2KT + \Delta_{2,k + 1}$ and parameter $2\lam^{k - 1}$.}
        Then by \eqref{eq:fsmax} we have
        \begin{align}
          &{\P[{Z^{k}(t) - (T^2  - 2KT + \Delta_{2,k + 1}) \le T^\del \lam^{-(k - n_{0})/5},\ \ \forall t \in [-T, -T + \Delta_{1, k + 1}, T - \Delta_{1, k + 1}, T]}]} \nonumber \\
          &\hskip10em \geq 1 - T\lam^{2k/3}\Exp{-\T{T^{3\del/2}\lam^{-3(k - n_{0})/10}\lam^{k/2}}} \nonumber \\
          &\hskip10em = 1 - \Exp{-\T{T^{3\del/2}\lam^{k/5}}}=:1-\eta_k. \label{finalbound1234}
        \end{align}
In particular, with this %
probability,
        \begin{align*}
          Z^{k}(\pm (T - \Delta_{1, k + 1})) \leq T^2 -2KT + \Delta_{2,k + 1} + T^\del \lam^{-(k - n_{0})/5} =: C_{k}
        \end{align*}
        Therefore, in this case,  $Z^{k}|_{[-T + \Delta_{1, k + 1}, T - \Delta_{1, k + 1}]}$ is stochastically dominated by $\hat Z^k$ where
        \begin{align*}
         \hat Z^k &\sim \L^{2\lam^{k - 1};C_k}_{ T - \Delta_{1, k + 1}}[\Delta_{2, k + 1}].
        \end{align*}
        Observe that $$\hat Z^k = V^k + \Delta_{2, k + 1}$$ where $V^k \sim \L^{2\lam^{k - 1};C_k - \Delta_{2, k + 1}}_{ T - \Delta_{1, k + 1}}$.
        Since $\Delta_{1, k + 1} \leq O(T\lam^{-(n_{0} - 1)/2}) = O(T^{-15})$,        \begin{align*}
          C_k - \Delta_{2,k + 1} \leq T^{2} \leq 2(T - \Delta_{1, k + 1})^2 = 2\Rnd{T - O(T^{-15})}^{2}.
        \end{align*}
        This allows us to apply \rlem{lem:highenv}, {with $T$ replaced by $ T - \Delta_{1, k + 1}, H = C_{k} - \Delta_{2, k  + 1}$} yielding
        \begin{align*}
          \P[V^k \preceq \Env'_k] \geq 1 - \Exp{-\T{\lam^{k\del/2}}}=:1-\eta'_k,    \end{align*}
      where we define $\Env'_k$ as
      \begin{align*}
        \Env'_k(t) = \begin{cases}
                       C_{k} - \Delta_{2, k + 1} + T^{\del}\lam^{-(k - n_{0})/5} & t \in [-T, -T + \Delta_{1, k}] \cup [T - \Delta_{1, k}, T] \\
                       T^{\del}\lam^{-(k - n_{0})/5} & t \in [-T + \Delta_{1, k}, T - \Delta_{1, k}].
                     \end{cases}
      \end{align*}
      Therefore $\hat Z^k \preceq \Env'_k + \Delta_{2,k + 1} \preceq \Env_k$ with probability $1-\eta'_k$. This along with \eqref{finalbound1234}, finishes the induction step by the observation that the  %
      probability of $X^k  \preceq \Env_k$ is at least 
      \begin{align*}
      (1-q_{k+1})(1-\eta_k)(1-\eta'_k)\geq 1-q_{k+1}-\eta_k-\eta'_k \ge 1-q_k.
      \end{align*}
      Note that above, the assumption that $\del$ is small enough, in particular $\del<1/20$, and that the constant $c>0$ defining $q_k$ is sufficiently small is used. %
      This ends the proof.
\end{proof}

{

}

\bibliographystyle{plain}

\bibliography{ref}

\begin{thebibliography}{10}

\bibitem{10.1214/09-AOP493}
Mark Adler, Patrik~L. Ferrari, and Pierre van Moerbeke.
\newblock {{A}iry processes with wanderers and new universality classes}.
\newblock {\em The Annals of Probability}, 38(2):714 -- 769, 2010.

\bibitem{aggarwal2020arctic}
Amol Aggarwal.
\newblock Arctic boundaries of the ice model on three-bundle domains.
\newblock {\em Inventiones mathematicae}, 220(2):611--671, 2020.

\bibitem{aggarwal2023strong}
Amol Aggarwal and Jiaoyang Huang.
\newblock Strong characterization for the {A}iry line ensemble.
\newblock {\em preprint arXiv:2308.11908}, 2023.

\bibitem{borodin2010gibbs}
Alexei Borodin and Senya Shlosman.
\newblock Gibbs ensembles of nonintersecting paths.
\newblock {\em Communications in Mathematical Physics}, 293:145--170, 2010.

\bibitem{bricmont1986random}
Jean Bricmont, A~El~Mellouki, and J{\"u}rg Fr{\"o}hlich.
\newblock Random surfaces in statistical mechanics: Roughening, rounding,
  wetting,...
\newblock {\em Journal of statistical physics}, 42:743--798, 1986.

\bibitem{caddeoetal}
Patrizio Caddeo, Yujin~H. Kim, and Eyal Lubetzky.
\newblock On level line fluctuations of {SOS} surfaces above a wall.
\newblock {\em preprint arXiv:2309.09106}, 2023.

\bibitem{geomarea}
Pietro Caputo and Shirshendu Ganguly.
\newblock Uniqueness, mixing, and optimal tails for brownian line ensembles
  with geometric area tilt.
\newblock {\em preprint arXiv:2305.18280}, 2023.

\bibitem{CIW18}
Pietro Caputo, Dmitry Ioffe, and Vitali Wachtel.
\newblock Confinement of brownian polymers under geometric area tilts.
\newblock {\em Electronic Journal of Probability}, 24:1--21, 2019.

\bibitem{CIW19}
Pietro Caputo, Dmitry Ioffe, and Vitali Wachtel.
\newblock Tightness and line ensembles for brownian polymers under geometric
  area tilts.
\newblock In {\em Statistical Mechanics of Classical and Disordered Systems:
  Luminy, France, August 2018}, pages 241--266. Springer, 2019.

\bibitem{caputoetal2016}
Pietro Caputo, Eyal Lubetzky, Fabio Martinelli, Allan Sly, and Fabio~Lucio
  Toninelli.
\newblock Scaling limit and cube-root fluctuations in {SOS} surfaces above a
  wall.
\newblock {\em Journal of the European Mathematical Society}, 18(5):931--995,
  2016.

\bibitem{corwinhammond}
Ivan Corwin and Alan Hammond.
\newblock Brownian {G}ibbs property for {A}iry line ensembles.
\newblock {\em Inventiones mathematicae}, 195(2):441--508, 2014.

\bibitem{dauvergne2022directed}
Duncan Dauvergne, Janosch Ortmann, and B{a}lint Vir{a}g.
\newblock The directed landscape.
\newblock {\em Acta Mathematica}, 229(2):201--285, 2022.

\bibitem{DLZ}
Amir Dembo, Eyal Lubetzky, and Ofer Zeitouni.
\newblock On the limiting law of line ensembles of brownian polymers with
  geometric area tilts.
\newblock {\em preprint arXiv:2201.01635}, 2022.

\bibitem{dimitrov2023uniform}
Evgeni Dimitrov and Christian Serio.
\newblock Uniform convergence of {D}yson-{F}errari-{S}pohn diffusions to the
  {A}iry line ensemble.
\newblock {\em preprint arXiv:2305.03723}, 2023.

\bibitem{ferrari2003step}
P.~Ferrari and H.~Spohn.
\newblock {Step fluctuations for a faceted crystal}.
\newblock {\em J. Stat. Phys}, 113(1):1--46, 2003.

\bibitem{ferrari2023airy2}
Patrik~L Ferrari and Senya Shlosman.
\newblock The {A}iry2 process and the 3{D} {I}sing model.
\newblock {\em Journal of Physics A: Mathematical and Theoretical},
  56(1):014003, 2023.

\bibitem{ferrarispohn2005}
Patrik~L Ferrari and Herbert Spohn.
\newblock Constrained {B}rownian motion: fluctuations away from circular and
  parabolic barriers.
\newblock {\em The Annals of Probability}, 33(4):1302--1325, 2005.

\bibitem{friedli2017statistical}
Sacha Friedli and Yvan Velenik.
\newblock {\em Statistical mechanics of lattice systems: a concrete
  mathematical introduction}.
\newblock Cambridge University Press, 2017.

\bibitem{prewet2}
Shirshendu Ganguly and Reza Gheissari.
\newblock {Local and global geometry of the 2{D} {I}sing interface in critical
  prewetting}.
\newblock {\em The Annals of Probability}, 49(4):2076 -- 2140, 2021.

\bibitem{ganguly2022sharp}
Shirshendu Ganguly and Milind Hegde.
\newblock Sharp upper tail estimates and limit shapes for the {KPZ} equation
  via the tangent method.
\newblock {\em preprint arXiv:2208.08922}, 2022.

\bibitem{georgii2011gibbs}
Hans-Otto Georgii.
\newblock {\em Gibbs measures and phase transitions}, volume~9.
\newblock Walter de Gruyter, 2011.

\bibitem{gheissari2023entropic}
Reza Gheissari and Eyal Lubetzky.
\newblock Entropic repulsion of 3{D} {I}sing interfaces conditioned to stay
  above a floor.
\newblock {\em Electronic Journal of Probability}, 28:1--44, 2023.

\bibitem{prewet1}
Dmitry Ioffe, S{e}bastien Ott, Senya Shlosman, and Yvan Velenik.
\newblock {Critical prewetting in the 2D {I}sing model}.
\newblock {\em The Annals of Probability}, 50(3):1127 -- 1172, 2022.

\bibitem{ioffe2019formation}
Dmitry Ioffe and Senya Shlosman.
\newblock Formation of facets for an effective model of crystal growth.
\newblock In {\em Sojourns in Probability Theory and Statistical Physics-I:
  Spin Glasses and Statistical Mechanics, A Festschrift for Charles M. Newman},
  pages 199--245. Springer, 2019.

\bibitem{ioffeshlosmanvelenik}
Dmitry Ioffe, Senya Shlosman, and Yvan Velenik.
\newblock An invariance principle to {F}errari--{S}pohn diffusions.
\newblock {\em Communications in Mathematical Physics}, 336(2):905--932, 2015.

\bibitem{ioffe2016low}
Dmitry Ioffe and Yvan Velenik.
\newblock Low-temperature interfaces: prewetting, layering, faceting and
  {F}errari-{S}pohn diffusions.
\newblock {\em arXiv preprint arXiv:1611.00658}, 2016.

\bibitem{ioffevelenikwachtel}
Dmitry Ioffe, Yvan Velenik, and Vitali Wachtel.
\newblock {D}yson {F}errari--{S}pohn diffusions and ordered walks under area
  tilts.
\newblock {\em Probability Theory and Related Fields}, pages 11--47, 2017.

\bibitem{johansson2005arctic}
Kurt Johansson.
\newblock The arctic circle boundary and the {A}iry process.
\newblock {\em The Annals of Probability}, 33(1):1--30, 2005.

\bibitem{matetski2021kpz}
Konstantin Matetski, Jeremy Quastel, and Daniel Remenik.
\newblock The {KPZ} fixed point.
\newblock {\em Acta Mathematica}, 227(1):115--203, 2021.

\bibitem{prahofer2002scale}
Michael Pr{\"a}hofer and Herbert Spohn.
\newblock Scale invariance of the {PNG} droplet and the {A}iry process.
\newblock {\em Journal of statistical physics}, 108:1071--1106, 2002.

\bibitem{serio2023scaling}
Christian Serio.
\newblock Scaling limit for line ensembles of random walks with geometric area
  tilts.
\newblock {\em preprint arXiv:2302.12817}, 2023.

\bibitem{sheffield2005random}
Scott Sheffield.
\newblock {\em Random surfaces}.
\newblock Soci{e}t{e} math{e}matique de France, Paris, 2005.

\end{thebibliography}

\end{document}